\documentclass[10pt, reqno]{amsart}
\usepackage{amssymb}
\usepackage{eucal}
\usepackage{amsmath}
\usepackage{amscd}
\usepackage[dvips]{color}
\usepackage{multicol}
\usepackage[all]{xy}           
\usepackage{graphicx}
\usepackage{color}
\usepackage{colordvi}
\usepackage{xspace}
\usepackage{makecell}
\usepackage{tikz}

\usepackage{ifpdf}
\ifpdf
 \usepackage[colorlinks,final,backref=page,hyperindex]{hyperref}
\else
 \usepackage[colorlinks,final,backref=page,hyperindex,hypertex]{hyperref}
\fi

\begin{document}
\newcommand {\emptycomment}[1]{} 

\newcommand{\nc}{\newcommand}
\newcommand{\delete}[1]{}
\nc{\mfootnote}[1]{\footnote{#1}} 
\nc{\todo}[1]{\tred{To do:} #1}

\delete{
\nc{\mlabel}[1]{\label{#1}}  
\nc{\mcite}[1]{\cite{#1}}  
\nc{\mref}[1]{\ref{#1}}  
\nc{\meqref}[1]{\eqref{#1}} 
\nc{\mbibitem}[1]{\bibitem{#1}} 
}

\nc{\mlabel}[1]{\label{#1}  
{\hfill \hspace{1cm}{\bf{{\ }\hfill(#1)}}}}
\nc{\mcite}[1]{\cite{#1}{{\bf{{\ }(#1)}}}}  
\nc{\mref}[1]{\ref{#1}{{\bf{{\ }(#1)}}}}  
\nc{\meqref}[1]{\eqref{#1}{{\bf{{\ }(#1)}}}} 
\nc{\mbibitem}[1]{\bibitem[\bf #1]{#1}} 

\newtheorem{thm}{Theorem}[section]
\newtheorem{lem}[thm]{Lemma}
\newtheorem{cor}[thm]{Corollary}
\newtheorem{pro}[thm]{Proposition}
\newtheorem{conj}[thm]{Conjecture}
\theoremstyle{definition}
\newtheorem{defi}[thm]{Definition}
\newtheorem{ex}[thm]{Example}
\newtheorem{rmk}[thm]{Remark}
\newtheorem{pdef}[thm]{Proposition-Definition}
\newtheorem{condition}[thm]{Condition}

\renewcommand{\labelenumi}{{\rm(\alph{enumi})}}
\renewcommand{\theenumi}{\alph{enumi}}

\nc{\tred}[1]{\textcolor{red}{#1}}
\nc{\tblue}[1]{\textcolor{blue}{#1}}
\nc{\tgreen}[1]{\textcolor{green}{#1}}
\nc{\tpurple}[1]{\textcolor{purple}{#1}}
\nc{\btred}[1]{\textcolor{red}{\bf #1}}
\nc{\btblue}[1]{\textcolor{blue}{\bf #1}}
\nc{\btgreen}[1]{\textcolor{green}{\bf #1}}
\nc{\btpurple}[1]{\textcolor{purple}{\bf #1}}

\nc{\rp}[1]{\textcolor{blue}{Ruipu:#1}}
\nc{\cm}[1]{\textcolor{red}{Chengming:#1}}
\nc{\li}[1]{\textcolor{blue}{#1}}
\nc{\lir}[1]{\textcolor{blue}{Li:#1}}


\nc{\twovec}[2]{\left(\begin{array}{c} #1 \\ #2\end{array} \right )}
\nc{\threevec}[3]{\left(\begin{array}{c} #1 \\ #2 \\ #3 \end{array}\right )}
\nc{\twomatrix}[4]{\left(\begin{array}{cc} #1 & #2\\ #3 & #4 \end{array} \right)}
\nc{\threematrix}[9]{{\left(\begin{matrix} #1 & #2 & #3\\ #4 & #5 & #6 \\ #7 & #8 & #9 \end{matrix} \right)}}
\nc{\twodet}[4]{\left|\begin{array}{cc} #1 & #2\\ #3 & #4 \end{array} \right|}

\nc{\rk}{\mathrm{r}}
\newcommand{\g}{\mathfrak g}
\newcommand{\h}{\mathfrak h}
\newcommand{\pf}{\noindent{$Proof$.}\ }
\newcommand{\frkg}{\mathfrak g}
\newcommand{\frkh}{\mathfrak h}
\newcommand{\Id}{\rm{Id}}
\newcommand{\gl}{\mathfrak {gl}}
\newcommand{\ad}{\mathrm{ad}}
\newcommand{\add}{\frka\frkd}
\newcommand{\frka}{\mathfrak a}
\newcommand{\frkb}{\mathfrak b}
\newcommand{\frkc}{\mathfrak c}
\newcommand{\frkd}{\mathfrak d}
\newcommand {\comment}[1]{{\marginpar{*}\scriptsize\textbf{Comments:} #1}}

\nc{\tforall}{\text{ for all }}

\nc{\svec}[2]{{\tiny\left(\begin{matrix}#1\\
#2\end{matrix}\right)\,}}  
\nc{\ssvec}[2]{{\tiny\left(\begin{matrix}#1\\
#2\end{matrix}\right)\,}} 

\nc{\typeI}{local cocycle $3$-Lie bialgebra\xspace}
\nc{\typeIs}{local cocycle $3$-Lie bialgebras\xspace}
\nc{\typeII}{double construction $3$-Lie bialgebra\xspace}
\nc{\typeIIs}{double construction $3$-Lie bialgebras\xspace}

\nc{\bia}{{$\mathcal{P}$-bimodule ${\bf k}$-algebra}\xspace}
\nc{\bias}{{$\mathcal{P}$-bimodule ${\bf k}$-algebras}\xspace}

\nc{\rmi}{{\mathrm{I}}}
\nc{\rmii}{{\mathrm{II}}}
\nc{\rmiii}{{\mathrm{III}}}
\nc{\pr}{{\mathrm{pr}}}
\newcommand{\huaA}{\mathcal{A}}

\nc{\mcdot}{{}}

\nc{\OT}{constant $\theta$-}
\nc{\T}{$\theta$-}
\nc{\IT}{inverse $\theta$-}


\nc{\asi}{ASI\xspace}
\nc{\dualp}{transposed Poisson\xspace}
\nc{\Dualp}{Transposed Poisson\xspace}
\nc{\dualpop}{{\bf TPois}\xspace}
\nc{\ldualp}{derivation-transposed Poisson\xspace}

\nc{\spdualp}{sp-dual Poisson \xspace} \nc{\aybe}{AYBE\xspace}

\nc{\admset}{\{\pm x\}\cup K^{\times} x^{-1}}

\nc{\dualrep}{gives a dual representation\xspace}
\nc{\admt}{admissible to\xspace}

\nc{\ciri}{\circ_{\rm I}}
\nc{\cirii}{\circ_{\rm II}}
\nc{\ciriii}{\circ_{\rm III}}

\nc{\opa}{\cdot_A}
\nc{\opb}{\cdot_B}

\nc{\post}{positive type\xspace}
\nc{\negt}{negative type\xspace}
\nc{\invt}{inverse type\xspace}

\nc{\pll}{\beta}
\nc{\plc}{\epsilon}

\nc{\ass}{{\mathit{Ass}}}
\nc{\comm}{{\mathit{Comm}}}
\nc{\dend}{{\mathit{Dend}}}
\nc{\zinb}{{\mathit{Zinb}}}
\nc{\tdend}{{\mathit{TDend}}}
\nc{\prelie}{{\mathit{preLie}}}
\nc{\postlie}{{\mathit{PostLie}}}
\nc{\quado}{{\mathit{Quad}}}
\nc{\octo}{{\mathit{Octo}}}
\nc{\ldend}{{\mathit{ldend}}}
\nc{\lquad}{{\mathit{LQuad}}}

 \nc{\adec}{\check{;}} \nc{\aop}{\alpha}
\nc{\dftimes}{\widetilde{\otimes}} \nc{\dfl}{\succ} \nc{\dfr}{\prec}
\nc{\dfc}{\circ} \nc{\dfb}{\bullet} \nc{\dft}{\star}
\nc{\dfcf}{{\mathbf k}} \nc{\apr}{\ast} \nc{\spr}{\cdot}
\nc{\twopr}{\circ} \nc{\tspr}{\star} \nc{\sempr}{\ast}
\nc{\disp}[1]{\displaystyle{#1}}
\nc{\bin}[2]{ (_{\stackrel{\scs{#1}}{\scs{#2}}})}  
\nc{\binc}[2]{ \left (\!\! \begin{array}{c} \scs{#1}\\
    \scs{#2} \end{array}\!\! \right )}  
\nc{\bincc}[2]{  \left ( {\scs{#1} \atop
    \vspace{-.5cm}\scs{#2}} \right )}  
\nc{\sarray}[2]{\begin{array}{c}#1 \vspace{.1cm}\\ \hline
    \vspace{-.35cm} \\ #2 \end{array}}
\nc{\bs}{\bar{S}} \nc{\dcup}{\stackrel{\bullet}{\cup}}
\nc{\dbigcup}{\stackrel{\bullet}{\bigcup}} \nc{\etree}{\big |}
\nc{\la}{\longrightarrow} \nc{\fe}{\'{e}} \nc{\rar}{\rightarrow}
\nc{\dar}{\downarrow} \nc{\dap}[1]{\downarrow
\rlap{$\scriptstyle{#1}$}} \nc{\uap}[1]{\uparrow
\rlap{$\scriptstyle{#1}$}} \nc{\defeq}{\stackrel{\rm def}{=}}
\nc{\dis}[1]{\displaystyle{#1}} \nc{\dotcup}{\,
\displaystyle{\bigcup^\bullet}\ } \nc{\sdotcup}{\tiny{
\displaystyle{\bigcup^\bullet}\ }} \nc{\hcm}{\ \hat{,}\ }
\nc{\hcirc}{\hat{\circ}} \nc{\hts}{\hat{\shpr}}
\nc{\lts}{\stackrel{\leftarrow}{\shpr}}
\nc{\rts}{\stackrel{\rightarrow}{\shpr}} \nc{\lleft}{[}
\nc{\lright}{]} \nc{\uni}[1]{\tilde{#1}} \nc{\wor}[1]{\check{#1}}
\nc{\free}[1]{\bar{#1}} \nc{\den}[1]{\check{#1}} \nc{\lrpa}{\wr}
\nc{\curlyl}{\left \{ \begin{array}{c} {} \\ {} \end{array}
    \right .  \!\!\!\!\!\!\!}
\nc{\curlyr}{ \!\!\!\!\!\!\!
    \left . \begin{array}{c} {} \\ {} \end{array}
    \right \} }
\nc{\leaf}{\ell}       
\nc{\longmid}{\left | \begin{array}{c} {} \\ {} \end{array}
    \right . \!\!\!\!\!\!\!}
\nc{\ot}{\otimes} \nc{\sot}{{\scriptstyle{\ot}}}
\nc{\otm}{\overline{\ot}}
\nc{\ora}[1]{\stackrel{#1}{\rar}}
\nc{\ola}[1]{\stackrel{#1}{\la}}
\nc{\pltree}{\calt^\pl}
\nc{\epltree}{\calt^{\pl,\NC}}
\nc{\rbpltree}{\calt^r}
\nc{\scs}[1]{\scriptstyle{#1}} \nc{\mrm}[1]{{\rm #1}}
\nc{\dirlim}{\displaystyle{\lim_{\longrightarrow}}\,}
\nc{\invlim}{\displaystyle{\lim_{\longleftarrow}}\,}
\nc{\mvp}{\vspace{0.5cm}} \nc{\svp}{\vspace{2cm}}
\nc{\vp}{\vspace{8cm}} \nc{\proofbegin}{\noindent{\bf Proof: }}
\nc{\proofend}{$\blacksquare$ \vspace{0.5cm}}
\nc{\freerbpl}{{F^{\mathrm RBPL}}}
\nc{\sha}{{\mbox{\cyr X}}}  
\nc{\ncsha}{{\mbox{\cyr X}^{\mathrm NC}}} \nc{\ncshao}{{\mbox{\cyr
X}^{\mathrm NC,\,0}}}
\nc{\shpr}{\diamond}    
\nc{\shprm}{\overline{\diamond}}    
\nc{\shpro}{\diamond^0}    
\nc{\shprr}{\diamond^r}     
\nc{\shpra}{\overline{\diamond}^r}
\nc{\shpru}{\check{\diamond}} \nc{\catpr}{\diamond_l}
\nc{\rcatpr}{\diamond_r} \nc{\lapr}{\diamond_a}
\nc{\sqcupm}{\ot}
\nc{\lepr}{\diamond_e} \nc{\vep}{\varepsilon} \nc{\labs}{\mid\!}
\nc{\rabs}{\!\mid} \nc{\hsha}{\widehat{\sha}}
\nc{\lsha}{\stackrel{\leftarrow}{\sha}}
\nc{\rsha}{\stackrel{\rightarrow}{\sha}} \nc{\lc}{\lfloor}
\nc{\rc}{\rfloor}
\nc{\tpr}{\sqcup}
\nc{\nctpr}{\vee}
\nc{\plpr}{\star}
\nc{\rbplpr}{\bar{\plpr}}
\nc{\sqmon}[1]{\langle #1\rangle}
\nc{\forest}{\calf}
\nc{\altx}{\Lambda_X} \nc{\vecT}{\vec{T}} \nc{\onetree}{\bullet}
\nc{\Ao}{\check{A}}
\nc{\seta}{\underline{\Ao}}
\nc{\deltaa}{\overline{\delta}}
\nc{\trho}{\tilde{\rho}}

\nc{\rpr}{\circ}
\nc{\dpr}{{\tiny\diamond}}
\nc{\rprpm}{{\rpr}}

\nc{\mmbox}[1]{\mbox{\ #1\ }} \nc{\ann}{\mrm{ann}}
\nc{\Aut}{\mrm{Aut}} \nc{\can}{\mrm{can}}
\nc{\twoalg}{{two-sided algebra}\xspace}
\nc{\colim}{\mrm{colim}}
\nc{\Cont}{\mrm{Cont}} \nc{\rchar}{\mrm{char}}
\nc{\cok}{\mrm{coker}} \nc{\dtf}{{R-{\rm tf}}} \nc{\dtor}{{R-{\rm
tor}}}
\renewcommand{\det}{\mrm{det}}
\nc{\depth}{{\mrm d}}
\nc{\Div}{{\mrm Div}} \nc{\End}{\mrm{End}} \nc{\Ext}{\mrm{Ext}}
\nc{\Fil}{\mrm{Fil}} \nc{\Frob}{\mrm{Frob}} \nc{\Gal}{\mrm{Gal}}
\nc{\GL}{\mrm{GL}} \nc{\Hom}{\mrm{Hom}} \nc{\hsr}{\mrm{H}}
\nc{\hpol}{\mrm{HP}} \nc{\id}{\mrm{id}} \nc{\im}{\mrm{im}}
\nc{\incl}{\mrm{incl}} \nc{\length}{\mrm{length}}
\nc{\LR}{\mrm{LR}} \nc{\mchar}{\rm char} \nc{\NC}{\mrm{NC}}
\nc{\mpart}{\mrm{part}} \nc{\pl}{\mrm{PL}}
\nc{\ql}{{\QQ_\ell}} \nc{\qp}{{\QQ_p}}
\nc{\rank}{\mrm{rank}} \nc{\rba}{\rm{RBA }} \nc{\rbas}{\rm{RBAs }}
\nc{\rbpl}{\mrm{RBPL}}
\nc{\rbw}{\rm{RBW }} \nc{\rbws}{\rm{RBWs }} \nc{\rcot}{\mrm{cot}}
\nc{\rest}{\rm{controlled}\xspace}
\nc{\rdef}{\mrm{def}} \nc{\rdiv}{{\rm div}} \nc{\rtf}{{\rm tf}}
\nc{\rtor}{{\rm tor}} \nc{\res}{\mrm{res}} \nc{\SL}{\mrm{SL}}
\nc{\Spec}{\mrm{Spec}} \nc{\tor}{\mrm{tor}} \nc{\Tr}{\mrm{Tr}}
\nc{\mtr}{\mrm{sk}}

\nc{\ab}{\mathbf{Ab}} \nc{\Alg}{\mathbf{Alg}}
\nc{\Algo}{\mathbf{Alg}^0} \nc{\Bax}{\mathbf{Bax}}
\nc{\Baxo}{\mathbf{Bax}^0} \nc{\RB}{\mathbf{RB}}
\nc{\RBo}{\mathbf{RB}^0} \nc{\BRB}{\mathbf{RB}}
\nc{\Dend}{\mathbf{DD}} \nc{\bfk}{{K}} \nc{\bfone}{{\bf 1}}
\nc{\base}[1]{{a_{#1}}} \nc{\detail}{\marginpar{\bf More detail}
    \noindent{\bf Need more detail!}
    \svp}
\nc{\Diff}{\mathbf{Diff}} \nc{\gap}{\marginpar{\bf
Incomplete}\noindent{\bf Incomplete!!}
    \svp}
\nc{\FMod}{\mathbf{FMod}} \nc{\mset}{\mathbf{MSet}}
\nc{\rb}{\mathrm{RB}} \nc{\Int}{\mathbf{Int}}
\nc{\Mon}{\mathbf{Mon}}
\nc{\remarks}{\noindent{\bf Remarks: }}
\nc{\OS}{\mathbf{OS}} 
\nc{\Rep}{\mathbf{Rep}}
\nc{\Rings}{\mathbf{Rings}} \nc{\Sets}{\mathbf{Sets}}
\nc{\DT}{\mathbf{DT}}

\nc{\BA}{{\mathbb A}} \nc{\CC}{{\mathbb C}} \nc{\DD}{{\mathbb D}}
\nc{\EE}{{\mathbb E}} \nc{\FF}{{\mathbb F}} \nc{\GG}{{\mathbb G}}
\nc{\HH}{{\mathbb H}} \nc{\LL}{{\mathbb L}} \nc{\NN}{{\mathbb N}}
\nc{\QQ}{{\mathbb Q}} \nc{\RR}{{\mathbb R}} \nc{\BS}{{\mathbb{S}}} \nc{\TT}{{\mathbb T}}
\nc{\VV}{{\mathbb V}} \nc{\ZZ}{{\mathbb Z}}


\nc{\calao}{{\mathcal A}} \nc{\cala}{{\mathcal A}}
\nc{\calc}{{\mathcal C}} \nc{\cald}{{\mathcal D}}
\nc{\cale}{{\mathcal E}} \nc{\calf}{{\mathcal F}}
\nc{\calfr}{{{\mathcal F}^{\,r}}} \nc{\calfo}{{\mathcal F}^0}
\nc{\calfro}{{\mathcal F}^{\,r,0}} \nc{\oF}{\overline{F}}
\nc{\calg}{{\mathcal G}} \nc{\calh}{{\mathcal H}}
\nc{\cali}{{\mathcal I}} \nc{\calj}{{\mathcal J}}
\nc{\call}{{\mathcal L}} \nc{\calm}{{\mathcal M}}
\nc{\caln}{{\mathcal N}} \nc{\calo}{{\mathcal O}}
\nc{\calp}{{\mathcal P}} \nc{\calq}{{\mathcal Q}} \nc{\calr}{{\mathcal R}}
\nc{\calt}{{\mathcal T}} \nc{\caltr}{{\mathcal T}^{\,r}}
\nc{\calu}{{\mathcal U}} \nc{\calv}{{\mathcal V}}
\nc{\calw}{{\mathcal W}} \nc{\calx}{{\mathcal X}}
\nc{\CA}{\mathcal{A}}

\nc{\fraka}{{\mathfrak a}} \nc{\frakB}{{\mathfrak B}}
\nc{\frakb}{{\mathfrak b}} \nc{\frakd}{{\mathfrak d}}
\nc{\oD}{\overline{D}}
\nc{\frakF}{{\mathfrak F}} \nc{\frakg}{{\mathfrak g}}
\nc{\frakm}{{\mathfrak m}} \nc{\frakM}{{\mathfrak M}}
\nc{\frakMo}{{\mathfrak M}^0} \nc{\frakp}{{\mathfrak p}}
\nc{\frakS}{{\mathfrak S}} \nc{\frakSo}{{\mathfrak S}^0}
\nc{\fraks}{{\mathfrak s}} \nc{\os}{\overline{\fraks}}
\nc{\frakT}{{\mathfrak T}}
\nc{\oT}{\overline{T}}
\nc{\frakX}{{\mathfrak X}} \nc{\frakXo}{{\mathfrak X}^0}
\nc{\frakx}{{\mathbf x}}
\nc{\frakTx}{\frakT}      
\nc{\frakTa}{\frakT^a}        
\nc{\frakTxo}{\frakTx^0}   
\nc{\caltao}{\calt^{a,0}}   
\nc{\ox}{\overline{\frakx}} \nc{\fraky}{{\mathfrak y}}
\nc{\frakz}{{\mathfrak z}} \nc{\oX}{\overline{X}}

\font\cyr=wncyr10

\nc{\al}{\alpha}
\nc{\lam}{\lambda}
\nc{\lr}{\longrightarrow}


\title[Dual pre-Poisson  algebras]{Properties, deformation quantizations and bialgebras for dual pre-Poisson algebras}

\author{Dilei Lu}
\address{College of Applied Science, Beijing Information Science and Technology University, Beijing 100192, China}
         \email{ludyray@bistu.edu.cn}

\date{\today}

\begin{abstract}
A dual pre-Poisson algebra is an algebraic structure that integrates a permutative algebra and a Leibniz algebra under certain compatibility conditions. As the Koszul dual notion of the pre-Poisson algebra, this structure serves as a natural generalization of the Poisson algebra. In this paper, we commence a study on dual pre-Poisson algebras from the algebraic point of view and establish a bialgebra theory for dual pre-Poisson algebras. We begin by investigating the fundamental properties of dual pre-Poisson algebras and provide several explicit constructions. In particular, we prove that the operad of dual pre-Poisson algebras is Koszul, and we compute its Hilbert-Poincaré series and codimension. Furthermore, we introduce the notion of diassociative formal deformations of permutative
algebras and show that dual pre-Poisson algebras are the corresponding semi-classical limits.  Moreover, we introduce dual pre-Poisson bialgebras, which are characterized both by Manin triples and by matched pairs of dual pre-Poisson algebras, thus developing a bialgebra theory for dual pre-Poisson algebras. Then our study leads to the permutative-Leibniz Yang-Baxter equation (PLYBE) that is composed of the permutative  and the classical Leibniz Yang-Baxter equation. We conclude by introducing $\mathcal{O}$-operators and pre-dual pre-Poisson algebras, which provide a systematic method for constructing symmetric solutions to the PLYBE and, consequently, dual pre-Poisson bialgebras.


\end{abstract}

\subjclass[2020]{
16S37,  	
16T10,  
17A30  	
17A32,  
17B38,  
17B63,  
18M70,  
53D55.  
}

\keywords{Leibniz algebra; permutative
algebra; Poisson algebra; dual pre-Poisson algebra; Koszul duality;  deformation quantization; bialgebras; Yang-Baxter equation; $\mathcal{O}$-operator}

\maketitle


\tableofcontents

\allowdisplaybreaks

\section{Introduction}

This paper studies dual pre-Poisson algebras, focusing on their properties, constructions, and deformation quantizations, and establishes a corresponding bialgebra theory. Futhermore, we   utilize analogs of the classical Yang-Baxter equation,  $\mathcal{O}$-operators  and pre-structures to give dual pre-Poisson bialgebras.

\subsection{Poisson algebras and pre-Poisson algebras} 
Poisson algebras, which combine the structures of a commutative associative algebra and a Lie algebra linked by the Leibniz rule, arose from the study of Hamiltonian mechanics \cite{AVI,DPAM,OA} and Poisson geometry \cite{LA1,VI,WA}. They play  important roles in many fields in mathematics and mathematical physics, such as  geometric quantization \cite{HJ,KM}, algebraic geometry \cite{GVD,PA}, quantum groups \cite{CV,DVG}, algebraic operads \cite{FB,GMRE,MR} and partial differential equations \cite{CM,SAVKDV}. Poisson algebras are also known as the classical limits of associative formal deformations of commutative associative algebras. Therefore, such associative formal deformations are regarded as  the deformation quantizations of the corresponding Poisson algebras \cite{MR}. In operad theory, the operad of Poisson algebras is quadratic binary,  Koszul, cyclic and Dong  \cite{GK,KPS,MM}. Moreover, it enjoys the remarkable property of being Koszul self-dual \cite{LB}. Building on this algebraic foundation, the bialgebra theory for Poisson algebras was established in   \cite{NB}. This study employs the notion of a pre-Poisson algebra, which was first proposed in \cite{MA}, and led to the construction of coboundary Poisson bialgebras \cite{LLB,LBS,NB}.

A pre-Poisson algebra is defined as a compatible combination of a pre-Lie algebra and a Zinbiel algebra, which induces its sub-adjacent Poisson algebra through the commutator of the pre-Lie algebra and the anti-commutator of the Zinbiel algebra \cite{MA}. It is also known that a Rota-Baxter operator on a Poisson algebra can produce a pre-Poisson algebra \cite{MA,BBG}. Moreover,  the pre-Poisson algebra is also a particular case of the Poisson dialgebra introduced by Loday in \cite{L4} and it is  the semi-classical limit of dendriform formal deformation of a zinbiel algebra \cite{MA}.  At the operadic level, it was shown that the operad of pre-Poisson algebras, as a binary quadratic operad, is isomorphic to the Manin black product of the operad of pre-Lie algebras and the operad of Poisson algebras \cite{BBG,UK}. Futhermore, the operad of pre-Poisson algebras is not only the disuccessor of the operad of Poisson algebras \cite{BBG}, but has also been proven to be Koszul by Dotsenko \cite{DT}, thus sharing fundamental  operadic properties with each other. Recent years have witnessed a surge of interest in pre-Poisson algebras.  The extending structures and bialgebra theory for pre-Poisson algebras were studied in \cite{ZLS}, while the notion of a phase space for a Poisson algebra and its relation to pre-Poisson algebras were introduced and discussed in \cite{WS}.

\subsection{Dual pre-Poisson algebras}
The notion of a dual pre-Poisson algebra was first given by Aguiar in \cite{MA}, which was introduced as the Koszul duality of the pre-Poisson algebra in the sense of algebraic operads \cite{UK}.  It contains a permutative algebra and a Leibniz algebra such that some compatibility conditions are satisfied. As  \cite{MA} indicate that there is a natural construction of dual pre-Poisson algebras  from  Poisson algebras equipped with  average operators. In \cite{JBGN}, it was proved that the operad of dual pre-Poisson algebras is the duplicator of the operad of Poisson algebras. Furthermore,   the operad of dual pre-Poisson algebras, being quadratic binary, is isomorphic to the  Manin white product of the operad of permutative algebras and the operad of Poisson algebras \cite{JBGN,UK}.    Some of the more recent progresses on dual pre-Poisson algebras can be found in \cite{GVPK,LSB,LMS,UK1,WS}.

In this paper, we investigate the properties of dual pre-Poisson algebras, proving that their operad is Koszul and computing the corresponding Hilbert-Poincaré series and codimension (Proposition \ref{Koszulfordpp}). Their importance is demonstrated in two key aspects. On one hand, dual pre-Poisson algebras exhibit close connections to other algebraic structures such as Poisson algebras and pre-Poisson algebras. Note that any Poisson algebra can be viewed as a special dual pre-Poisson algebra (Example \ref{dwdpp} (1)). Especially, Poisson algebras owe their importance crucially to their closure under the tensor product, overcoming the lack of such a property for Lie algebras. The dual pre-Poisson algebra also has this property (Proposition \ref{cfortdppa}). Then we generalize the typical construction of a Poisson algebra  from a commutative associative algebra  equipped with a pair of commuting derivations  to the context of dual pre-Poisson algebras, that is, there is a   construction of a dual pre-Poisson algebra by a pair of commuting derivation  on a permutative algebra (Proposition \ref{sdztodpp}). Moreover, dual pre-Poisson algebras can be obtained either from the tensor product of permutative and Poisson algebras (Proposition \ref{ifdualpp}), or from representations of Poisson algebras (Proposition \ref{dppfromHRofP}).
Specifically, we show that dual pre-Poisson algebras are  the  semi-classical limits of diassociative formal deformations of permutative algebras (Theorem \ref{dppofdeform}), generalizing the well-known fact that Poisson algebras are characterized as the  semi-classical limits of associative formal deformations of commutative associative algebras.

 It is   known that the operads of permutative algebras and pre-Lie algebras are  Koszul duals of each other \cite{CL}, as are the operads of Leibniz algebras and Zinbiel algebras \cite{L0}. Hence, we obtain the following diagram illustrating the relationships among Poisson algebras,  pre-Poisson algebras  and dual pre-Poisson algebras.
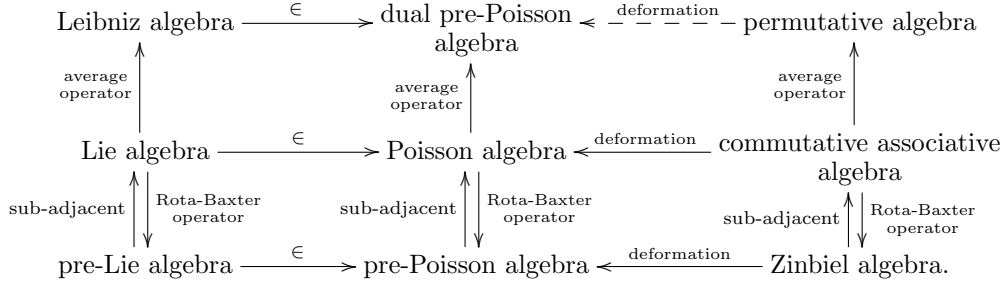
\begin{equation*}    
\xymatrix@C=1.5cm@R=0.8cm{ 
\text{Leibniz algebra}    \ar[r]^{\in\hspace{0.4cm}}&   \txt{dual pre-Poisson \\algebra}  &   \text{permutative algebra}\ar@{-->}[l]_{\tiny{\txt{deformation}}} \\    
\text{Lie algebra}  \ar[r]^-{\in} \ar@<.5ex>[u]^{\tiny{\txt{average \\ operator}}} \ar@<.2ex>[d]^{\tiny{\txt{Rota-Baxter \\ operator}}}  &   \text{Poisson algebra}   \ar@<.2ex>[d]^{\tiny{\txt{Rota-Baxter \\ operator}}}  \ar@<.5ex>[u]^-{\tiny{\txt{average \\ operator}}}     &   \ar[l]_{\tiny{\txt{deformation}\hspace{0.6cm}}} \txt{commutative associative \\algebra} \ar@<.5ex>[u]^{\tiny{\txt{average \\ operator}}} \ar@<.2ex>[d]^-{\tiny{\txt{Rota-Baxter \\ operator}}}     \\    
\text{pre-Lie algebra} \ar[r]^{\in\quad} \ar@<1ex>[u]^{\text{sub-adjacent}} &   \text{pre-Poisson algebra}   \ar@<1ex>[u]^{\text{sub-adjacent}} &  \ar[l]_-{\tiny{\txt{deformation}}} \text{Zinbiel  algebra.}\ar@<1ex>[u]^-{\text{sub-adjacent}}  }
\end{equation*}
From an algebraic perspective, the algebras in the first row can be obtained from those in the second row by applying an averaging operator, whereas the algebras in the third row are constructed from the second-row algebras by employing a Rota-Baxter operator \cite{MA,BBG,PBG,JBGN}. At the levels of  algebraic operads,  the Koszul duality  in the diagram is given by a central symmetry.   Furthermore, the operads of the algebras in the first and third rows are   the duplicator and the disuccessor, respectively, of the operads of those in the middle row \cite{BBG,JBGN}. More importantly,  the operads of  all the algebras in the diagram are Koszul \cite{BGL,Gik}.

On the other hand, dual pre-Poisson algebras serve as an effective framework for constructing compatible pre-Lie and Lie algebras, offering an alternative to the classical Poisson algebra setting. This utility is underpinned by a general fact that there exists a  Lie algebra structure on the tensor product of an algebra over a binary quadratic operad and an algebra over its Koszul dual \cite{Gik,L0}. Moreover, under the hypothesis specified in \cite[Proposition 13.9.1]{L0}, the tensor product construction of  the Lie algebra   is just the sub-adjacent Lie algebra  of  a pre-Lie  algebra. We construct two distinct pre-Lie algebra structures on the tensor product of a pre-Poisson algebra and a dual pre-Poisson algebra. The first arises from the Koszul duality between the operads of pre-Lie and permutative algebras, while the second originates from the Koszul duality between the operads of Leibniz and Zinbiel algebras. Furthermore, we prove that these two distinct pre-Lie algebra structures are compatible, meaning that any linear combination of their products again defines a pre-Lie algebra (Proposition \ref{ComFdualpp}). In particular, this yields a compatible Lie algebra. Such process can be understood as applying Koszul duality to a dual pre-Poisson algebra to produce a compatible Lie algebra.

\subsection{A bialgebra theory for dual pre-Poisson algebras}

A bialgebra is an algebraic structure equipped with a comultiplication, imposing compatibility conditions with the existing multiplication. These structures are pivotal in numerous fields, linking diverse concepts in both mathematics and physics. Due to the significance of bialgebra structures across mathematics and physics \cite{B1,BAX,CV,Dr}, developing a bialgebra theory for dual pre-Poisson algebras is a natural and motivated undertaking.  Note that except for the dual pre-Poisson algebra, all algebraic structures shown in the previous diagram have  established bialgebra theories, such as Leibniz algebras \cite{TS},  permutative algebras \cite{BYZ},  Lie algebras \cite{CV,Dr}, commutative associative algebras \cite{B2}, Poisson algebras \cite{NB}, pre-Lie algebras \cite{B1},   Zinbiel algebras \cite{WNB}, pre-Poisson algebras \cite{WS,ZLS}.

Motivated by the absence of a bialgebra theory for dual pre-Poisson algebras, this paper establishes such a theory independently and systematically.  Explicitly, we introduce the notion of a dual pre-Poisson bialgebra, which is defined as a structure that combines a permutative bialgebra and a Leibniz bialgebra in a compatible way. It is equivalent to a Manin triple associated to a nondegenerate skew-symmetric invariant bilinear form  and also corresponds to a certain matched pair of dual pre-Poisson algebras (Theorem \ref{sandengjia}). Moreover, the study of dual pre-Poisson bialgebras naturally leads to the permutative-Leibniz Yang-Baxter equation (PLYBE), an equation combining the permutative and the classical Leibniz Yang-Baxter equations \cite{BYZ,TS}. We show that a symmetric solution to the PLYBE yields such a bialgebra (Theorem \ref{tcdppab:1}). Furthermore, by introducing the notions of an $\mathcal{O}$-operator and a pre-dual pre-Poisson algebra, we provide explicit methods to construct these symmetric solutions and, consequently, dual pre-Poisson bialgebras (Corollary \ref{pretoBialgebra}). We summarize these results in the following diagram:

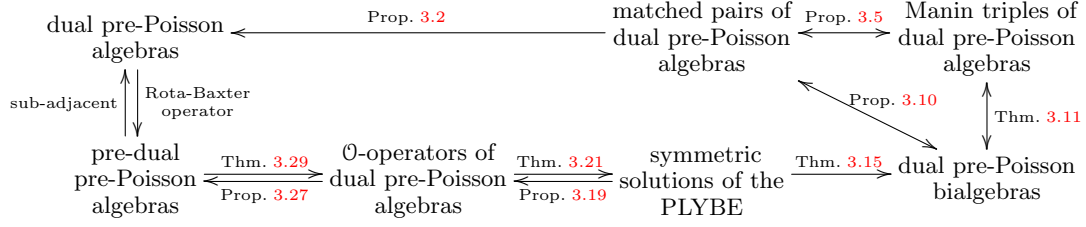
\begin{small}
    \begin{equation*}
\begin{gathered}
    \xymatrix@C=1.2cm@R=0.7cm{
  \txt{dual pre-Poisson \\algebras }\ar@{<-}[rr]^{\mathrm{Prop.\;}\ref{djdandm}\quad} \ar@<.2ex>[d]^-{\tiny{\txt{Rota-Baxter \\ operator}}}     &&\txt{matched pairs of \\dual pre-Poisson \\algebras } \ar@{<->}[dr]^-{\mathrm{Prop.\;}{\ref{mpdppba}}}\ar@{<->}[r]^-{\mathrm{Prop.\;}{\ref{mpandmtofdppa}}}&\txt{ Manin triples of \\dual pre-Poisson \\algebras }\ar@{<->}[d]^-{\mathrm{Thm.\;}{\ref{sandengjia}}}\\
		\txt{pre-dual\\ pre-Poisson \\algebras } \ar@{->}[r]^{\mathrm{Thm.\;}\ref{iogpdpp:0}\quad} \ar@<1ex>[u]^{\text{sub-adjacent}}  &
		\txt{$\mathcal{O}$-operators of \\dual pre-Poisson \\algebras } \ar@{->}[r]^{ \mathrm{Thm.\;}\ref{tcdppab:3}}   \ar@{->}[l]<0.7ex>^{\mathrm{Prop.\;}{\ref{OonVector}}\quad} &
		\txt{ symmetric \\ solutions of the \\ PLYBE}   \ar@{->}[l]<0.7ex>^{\;\mathrm{Prop.\;}{\ref{wssqof:0}}}   \ar@{->}[r]^{\mathrm{Thm.\;}{\ref{tcdppab:1}}\;\;}&\txt{dual pre-Poisson\\ bialgebras} 
 }
\end{gathered}
\end{equation*}
\end{small}

The paper is organized as follows. In Section \ref{Property},   we exhibit fundamental properties of dual pre-Poisson algebras and provide several explicit constructions. In particular, we consider the operadic properties of dual pre-Poisson algebras. Then we apply dual pre-Poisson algebras to construct compatible  pre-Lie algebras and compatible  Lie algebras via Koszul duality.   With introducing  the diassociative formal deformations of permutative algebras, we show that dual pre-Poisson algebras are the semi-classical limits of the diassociative formal deformations of permutative algebras. Finally, we introduce the representations of dual pre-Poisson algebras and their invariant bilinear forms. Subsequently, we define the notion of a quadratic dual pre-Poisson algebra that amalgamates a quadratic permutative algebra with a quadratic Leibniz algebra. In Section \ref{Bialgebra},  we introduce the notions of matched pairs, Manin triples of dual pre-Poisson algebras and dual pre-Poisson bialgebras, and establish their equivalence. Furthermore, we introduce the permutative-Leibniz Yang-Baxter equation (PLYBE) and show that its symmetric solutions give a   class of dual pre-Poisson bialgebras.  To construct such solutions, we finally introduce the concepts of $\mathcal{O}$-operators and pre-dual pre-Poisson algebras.

 Unless otherwise specified,  we work with finite-dimensional vector spaces and algebras over a field $\mathbb{F}$ of characteristic $0$ throughout this paper, although many results and definitions remain valid in the infinite-dimensional case.
\begin{enumerate}
\item For a vector space $V$,
let
$$\tau:V\otimes V\rightarrow V\otimes V,\quad u\otimes v\mapsto v\otimes u,\;\;\; \forall u,v\in V
$$
be the flip operator. 
\item  Let $A$ be a vector space with a bilinear
operation $\diamond$. Define linear maps $L_{\diamond},
R_{\diamond}:A\rightarrow {\rm End}(A)$ respectively by
\begin{eqnarray*}
L_{\diamond}(x)y:=x\diamond y,\;\; R_{\diamond}(x)y:= y \diamond x, \;\;\; \forall x, y\in A.
\end{eqnarray*}
 If there is a bilinear operation $\odot$ on the dual space $A^*$, we denote the linear maps $\mathcal{L}_{\odot},
\mathcal{R}_{\odot}:A^*\rightarrow {\rm End}(A^*)$ respectively by
 \begin{eqnarray*}
\mathcal{L}_{\odot}(a^*)b^*:= a^* \odot b^*,\;\; \mathcal{R}_{\odot}(a^*)b^*:= b^* \odot a^*, \;\;\;\forall a^*,b^*\in A^*.
\end{eqnarray*}

\item  Let  $A$  and  $V$  be vector spaces. For a linear map  $f: A \rightarrow {\rm End}(V)$, define a linear map  $f^{*}: A \rightarrow {\rm End}(V^*)$  by 
\begin{align}
\left\langle f^{*}(x) v^{*}, u\right\rangle=-\left\langle v^{*}, f(x) u\right\rangle, \quad \forall x \in A, u \in V, v^{*} \in V^{*}, \label{dualmap}
\end{align}
where $\left\langle \cdot,\cdot\right\rangle$ is the standard  pair between $V$ and $V^*$.  
\end{enumerate}

\section{Dual pre-Poisson algebras}\label{Property}

In this section, we begin by presenting both the fundamental and the operadic properties of dual pre-Poisson algebras. Then we provide several explicit constructions, which include  applying averaging operators to Poisson algebras, taking tensor products of permutative algebras with Poisson algebras, using representations of Poisson algebras, and employing pairs of commuting derivations on permutative algebras. Furthermore we utilize dual pre-Poisson algebras, through the framework of Koszul duality, to construct compatible Lie algebras. Moreover, we introduce the notion of diassociative  formal deformations of permutative algebras   and show that dual pre-Poisson algebras are the corresponding semi-classical limits.  Finally, by introducing the concept of an invariant bilinear form on a dual pre-Poisson algebra, we define a quadratic dual pre-Poisson algebra as a combination of a quadratic permutative algebra and a quadratic Leibniz algebra.

\subsection{Properties and constructions   of dual pre-Poisson algebras}  We  first  introduce the notion of dual pre-Poisson algebras and exhibit their properties in several aspects.

\begin{defi}\cite{MA}\label{def1}
A   \textbf{dual pre-Poisson algebra} is a triple $(A, \circ,[\cdot,\cdot])$ where $A$ is a vector space with two bilinear operations $\circ,[\cdot,\cdot]$ such that
\begin{enumerate}
\item $(A,\circ)$ is a {\bf permutative algebra}:
\begin{eqnarray}
x \circ (y \circ z) = (x \circ y) \circ z =  (y \circ x) \circ z , \quad \forall x,y,z \in A. \label{perm}
\end{eqnarray}
\item $(A,[\cdot,\cdot])$ is a {\bf Leibniz algebra}:
\begin{eqnarray}
[x , [y ,z]] = [[x ,y] ,z] + [y, [x,z]], \quad \forall x,y,z \in A. \label{Leibniz}
\end{eqnarray}
\item The following three compatibility conditions hold:
\begin{align}
[x, y \circ z] &=[x, y ]\circ z + y \circ[x,  z], \label{xdpp1} \\
[x \circ y,   z] &= x \circ [ y,   z] + y \circ [ x,   z], \label{xdpp2}\\
[x, y ]\circ z  &=  - [y, x]\circ z, \quad \forall x,y,z \in A. \label{xdpp3}
\end{align}
\end{enumerate}
Let $(A, \circ_A,[\cdot,\cdot]_A)$ and $(B, \circ_B,[\cdot,\cdot]_B)$ be dual pre-Poisson algebras. A linear map $f: A \to B$ is called a {\bf homomorphism of dual pre-Poisson   algebras} if  for all $x,y \in A$,
\begin{align}
f(x \circ_A y) = f(x) \circ_B f(y),\quad f([x,y]_A) = [f(x),f(y)]_B.
\end{align}
\end{defi}

\begin{pro}
Let $(A,\circ,[\cdot,\cdot])$  be a dual pre-Poisson algebra. Then the following equations hold
\begin{align*}
[x \circ y,   z] &= [y \circ x,   z],\\
[x, y \circ z] + x \circ [ y,   z]&=    [x \circ y,   z] + [x, y ]\circ z ,\\
 [x \circ y,   z]  &= [x, y \circ z] + [y, x\circ z], \quad \forall x,y,z \in A.
\end{align*}
\end{pro}
\begin{proof}
It directly follows  from the Definition~\ref{def1}.
\end{proof}

There are some examples for dual pre-Poisson algebras.

\begin{ex}\label{dwdpp}
(1) Recall that $(A, \bullet,\{\cdot,\cdot\})$  is called a   \textbf{Poisson algebra} where $(A,\bullet)$ is a commutative associative algebra  and $(A,\{\cdot,\cdot\})$ is a Lie algebra  satisfying  
\begin{align}
\{x, y \bullet z\} =&\{x, y \}\bullet z + y \bullet \{x,  z\},\quad \forall x,y \in A. \label{p1} 
\end{align}
Since commutative associative algebras are special permutative algebras and Lie algebras are Leibniz algebras satisfying anti-commutativity, thus all Poisson algebras are  dual pre-Poisson algebras.

(2) Let $A$ be a 2-dimensional vector space  with a basis $\{e_1,e_2\}$. Then there is a dual pre-Poisson algebra $(A,\circ,[\cdot,\cdot])$ whose the  nonzero products of $\circ,[\cdot,\cdot]$ are defined by
\begin{align}
   e_2 \circ e_2  = [e_2, e_2 ] = e_1. \label{dwdpp:1}
\end{align}

\end{ex}

{}

\begin{rmk}
Recall   that a \textbf{pre-Poisson algebra} is a triple $(A, \diamond,\odot )$ where $(A,\diamond)$ is a {\bf pre-Lie algebra}:
\begin{eqnarray}
x \diamond (y \diamond z) -(x \diamond y) \diamond z   =  y\diamond (x \diamond z) - (y \diamond x) \diamond z  , \quad \forall x,y,z \in A. \label{pre}
\end{eqnarray}
and $(A,\odot)$ is a {\bf Zinbiel algebra}:
\begin{eqnarray}
 x \odot (y \odot z) = (x \odot y + y \odot x) \odot z, \quad \forall x,y,z \in A. \label{Zinbiel}
\end{eqnarray}
such that  the compatibility conditions are satisfied: 
\begin{align}
(x \diamond y - y \diamond x) \odot z &=  x \diamond (y \odot z) -  y \odot (x \diamond z), \label{xpp1} \\
(x \odot y + y \odot x) \diamond z &=  x \odot  (y \diamond z) +  y \odot (x \diamond z), \quad \forall x,y,z \in A.\label{xpp2}
\end{align}
 The word “dual” in the name of dual pre-Poisson algebra means that it is the Koszul duality of the pre-Poisson algebra, which was introduced as the underlying algebraic structure of a Poisson algebra  equipped with a Rota-Baxter operator  \cite{MA}.  
\end{rmk}

It is well known that  Poisson algebras are closed under taking tensor products \cite{XX}. This property is also shared by dual pre-Poisson algebras. 
\begin{pro}\label{cfortdppa}
Let $(A, \circ_A,[\cdot,\cdot]_A)$  and $(B, \circ_B,[\cdot,\cdot]_B)$  be two dual pre-Poisson algebras. Define two bilinear operations $\circ_{A \otimes B}$ and  $[\cdot,\cdot]_{A \otimes B}$ on $A \otimes B$  by
\begin{align*}
 (x \otimes a) \circ_{A \otimes B} (y \otimes b) &=  x \circ_A y  \otimes  a \circ_B b,\\
    [ x \otimes a ,  y \otimes b]_{A \otimes B} &=  [x ,y]_A  \otimes  a \circ_B b +  x \circ_A y  \otimes [a , b]_B,\quad \forall x,y \in A, a,b \in B.
\end{align*}
Then  $(A \otimes B, \circ_{A \otimes B},[\cdot,\cdot]_{A \otimes B})$  is a dual pre-Poisson algebra.
\end{pro}
\begin{proof}
For brevity, the subscripts $A$ and $B$ in the operations $\circ$ and $[\cdot,\cdot]$ will suppressed, since their meaning should be clear from the context. 

It is obvious that $(A \otimes B, \circ_{A \otimes B})$ is a permutative algebra. Let $x,y,z \in A, a,b,c \in B$. Then we have
\begin{align*}
&[x \otimes a, [y \otimes b ,z \otimes c ]_{A \otimes B}]_{A \otimes B} -[[x \otimes a, y \otimes b],z \otimes c ]_{A \otimes B}]_{A \otimes B}- [y \otimes b, [x \otimes a ,z \otimes c ]_{A \otimes B}]_{A \otimes B}\\
&= x \circ [y,z] \otimes [a,b \circ c] +  [x ,y \circ  z] \otimes a \circ[b, c] - [x ,y] \circ  z  \otimes [a \circ b, c] -  [x \circ y,z] \otimes [a,b ]\circ c\\
&\quad - y \circ [x,z] \otimes [b,a \circ c] -  [y ,x \circ  z] \otimes b \circ[a, c] \\
&=x \circ [y,z] \otimes [a,b] \circ c  + x \circ [y,z] \otimes b \circ  [a,c] +   [x ,y] \circ  z  \otimes a \circ[b, c]  + y \circ [x ,z]    \otimes a \circ[b, c]\\
&\quad - [x ,y] \circ  z  \otimes a \circ [ b, c] - [x ,y] \circ  z  \otimes b \circ [ a, c] -    x \circ [ y,z] \otimes [a,b ]\circ c - y \circ [x,z] \otimes [a,b ]\circ c \\
&\quad  - y \circ [x,z] \otimes [b,a] \circ c - y \circ [x,z] \otimes b \circ  [a,c]  -  [y ,x ] \circ  z  \otimes b \circ[a, c] -  x \circ [y , z] \otimes b \circ[a, c] \\
&=0.
\end{align*}
Then $(A \otimes B,  [\cdot,\cdot]_{A \otimes B})$  is a Leibniz algebra. Furthermore, we have
\begin{align*}
&[x \!\otimes\! a, (y \!\otimes\! b) \circ_{A \otimes B} (z \!\otimes\! c)]_{A \otimes B} - [x \!\otimes\! a, y \!\otimes\! b ]_{A \otimes B} \circ_{A \otimes B} (z \!\otimes\! c) - (y \!\otimes\! b) \circ_{A \otimes B} [x \!\otimes\! a,  z\!\otimes\! c]_{A \otimes B}\\
&= [x  ,y \circ z] \!\otimes\! a \circ (b \circ c)   +  x  \circ (y \circ z) \!\otimes\! [a, b \circ c  ] - [x,y]\circ z \!\otimes\! (a \circ b) \circ c - (x  \circ y)\circ z \!\otimes\! [a , b] \circ c\\
&\quad - y\circ [x,z] \otimes b\circ (a \circ c) -y\circ(x\circ z) \otimes b\circ[a,c] \\
& = 0.\\
&[(x \!\otimes\! a) \circ_{A \otimes B} (y \!\otimes\! b) ,    z \!\otimes\! c ]_{A \otimes B} -  (x \!\otimes\! a)  \circ_{A \otimes B} [ y \!\otimes\! b,    z \!\otimes\! c ]_{A \otimes B} - (y \!\otimes\! b) \circ_{A \otimes B} [  x \!\otimes\! a  ,    z \!\otimes\! c ]_{A \otimes B}\\
&= [x  \circ y , z] \!\otimes\! (a \circ b) \circ c    +  (x  \circ  y) \circ z  \!\otimes\! [a\circ b , c  ]\! -\! x\circ [y,z] \!\otimes\! a\circ (b \circ c) \!-\!x\circ(y\circ z) \!\otimes\! a\circ[b,c] \! \\
&\quad  -\! y\circ [x,z] \!\otimes\! b\circ (a \circ c) \!-\!y\circ(x\circ z) \!\otimes\! b\circ[a,c] \\
&=0.\\
&[x  \otimes  a, y  \otimes  b ]_{A \otimes B}\circ_{A \otimes B} (z  \otimes  c) + [y  \otimes  b, x  \otimes  a]_{A \otimes B}\circ_{A \otimes B} (z  \otimes  c)\\
&= [x,y]\circ z \!\otimes\! (a \circ b) \circ c + (x  \circ y)\circ z \!\otimes\! [a , b] \circ c + [y,x]\circ z \!\otimes\! (b \circ a) \circ c + (y \circ x)\circ z \!\otimes\! [b, a] \circ c\\
&=0.
\end{align*}
Thus  $(A \otimes B, \circ_{A \otimes B},[\cdot,\cdot]_{A \otimes B})$  is a dual pre-Poisson algebra.
\end{proof}

\begin{rmk}
Note that same property holds for other Poisson-type  algebras such as Novikov-Poisson algebras and transposed Poisson algebras \cite{BBGW}.
\end{rmk}

In the following, we consider the properties of dual pre-Poisson algebras at the level of algebraic operads. Let $\mathcal{P}$ be a  binary quadratic (symmetric) operad and $\mathcal{P}^!$ be its Koszul dual operad (see \cite{BD1,Gik,LB} for more details on algebraic operads). For $n \in \mathbb{N}$, denote the vector
space consisting of elements of arity $n$ in the operad $\mathcal{P}$ by $\mathcal{P}(n)$ and  the {\bf codimension} of arity $n$ of  $\mathcal{P}$  by $\text{dim} \mathcal{P}(n) $. The   {\bf generating series} (or {\bf Hilbert-Poincaré series}) $f_{\mathcal{P}}(t)$  of the symmetric operad  $\mathcal{P}$  is given by
\begin{align*}
f_{\mathcal{P}}(t)=\sum_{n=1}^{\infty} \frac{\operatorname{dim} \mathcal{P}(n)}{n !} t^{n}.
\end{align*}

An operad is {\bf Koszul} if it is binary quadratic and further satisfies an homological criterion \cite{Gik,LB}. One motivation for studying the Koszulity of an operad $\mathcal{P}$ is that this property guarantees the existence of a small chain complex, which can be used to compute the (co)homology of $\mathcal{P}$-algebra \cite{BGL}. Denote the operads of  pre-Poisson algebras and dual pre-Poisson algebras by $\mathsf{prePois}$ and $\mathsf{DualprePois}$ respectively. Then we  have $\mathsf{prePois}^! = \mathsf{DualprePois}$.

\begin{pro}\label{Koszulfordpp}
The operad $\mathsf{DualprePois}$ of dual pre-Poisson algebras is Koszul. Moreover, the Hilbert-Poincaré series of $\mathsf{DualprePois}$ is given by
\begin{align*}
f_{\mathsf{DualprePois}}(t)= \frac{t}{(1-t)^2}
\end{align*}
and $\operatorname{dim}\mathsf{DualprePois}(n) = n \cdot n!$.
\end{pro}
\begin{proof} It was shown that the operad $\mathsf{prePois}$ of pre-Poisson algebras is Koszul \cite[Corollary 4.7]{DT}.  Thus by \cite[Proposition 4.1.4]{Gik}, as the Koszul dual operad of $\mathsf{prePois}$, the operad $\mathsf{DualprePois}$ is also Koszul. Let $\mathsf{Dend}$ and $\mathsf{Dias}$ be the (symmetric) operad of dendriform algebras and dialgebras  respectively (see \cite{B2,BGL,L4} for more details). Then by  \cite[Theorem 4.6]{DT}, the operads $\mathsf{prePois}$ and  $\mathsf{Dend}$ share the same Hilbert-Poincaré series, which is given by
\begin{align*}
f_{\mathsf{prePois}}(t) = f_{\mathsf{Dend}}(t) = \frac{1-2t-\sqrt{1-4t} }{2t}.
\end{align*}
 Since the operad $\mathsf{prePois}$ is Koszul, then by the Ginzburg-Kapranov criterion \cite[Proposition 4.1.4]{Gik}, the following functional equation holds
\begin{align} 
f_{\mathsf{DualprePois}}\left(-f_{\mathsf{prePois}}(-t)\right) =  t. \label{GKcri}
\end{align}
With a direct calculation, we have
\begin{align*}
f_{\mathsf{DualprePois}}(t)= \frac{t}{(1-t)^2}
\end{align*}
which coincides with the Hilbert-Poincaré series of the operad $\mathsf{Dias}$ \cite{BGL,L4}. Thus we conclude that
$$
\operatorname{dim}\mathsf{DualprePois}(n) = \operatorname{dim}\mathsf{Dias}(n)  = n \cdot n!.
$$
 This completes the proof.  \end{proof}
 
 \begin{rmk}In fact, the codimension of $\mathsf{DualprePois}$ can   be derived from its relation to the  operad $\mathsf{Pois}$   of Poisson algebras. It is known that the  operad $\mathsf{DualprePois}$ is the duplicator of the operad  $\mathsf{Pois}$ \cite{JBGN}. Applying  \cite[Proposition 4.1]{GVPK},  we have 
 $$
 \operatorname{dim}\mathsf{DualprePois}(n) = n\operatorname{dim}\mathsf{Pois}(n)  = n \cdot n!.
 $$
 \end{rmk}

Next we give several constructions of dual  pre-Poisson algebras. An important class of dual  pre-Poisson algebras come from average operators on Poisson algebras. 

\begin{defi}
Let $(A, \bullet,\{\cdot,\cdot\})$ be a Poisson algebra and $P:A \to A$  be a linear map. If $P$ satisfies the following equations
\begin{align}
P(x) \bullet P(y) = P(P(x) \bullet y ),\quad \{P(x),P(y)\}= P(\{P(x),y\}), \quad \forall x,y \in A,
\end{align} 
then  $P$ is called an {\bf average operator} on $(A, \bullet,\{\cdot,\cdot\})$.
\end{defi}

\begin{pro} \cite{MA} Let $(A, \bullet,\{\cdot,\cdot\})$ be a Poisson algebra and $P: A \to A$ be an average operator on $(A, \bullet,\{\cdot,\cdot\})$. Define new operations on $A$ by
\begin{align}
 x \circ y = P(x) \bullet y , \quad  [x,y]=  \{P(x),y\},\quad \forall x,y \in A.
\end{align} 
Then  $(A, \circ,[\cdot,\cdot])$ is a dual  pre-Poisson algebra.
\end{pro}

\begin{ex}\label{Ex:Pois}
Let  $(A, \bullet,\{\cdot,\cdot\})$ be the 3-dimensional Poisson algebra with a basis $\{e_1, e_2, e_3\}$ whose nonzero products are
\begin{align}
e_1 \bullet e_2 = e_3,\quad \{e_1, e_2 \} = 2e_3. \label{Ex:Pois:cf}
\end{align}
Let $P : A \to A$ be a linear map given by
\begin{align}
P(e_1) = e_1, \quad P(e_2) =  P(e_3) = 0.
\end{align}
Then  $P$ is an average operator on $(A, \bullet,\{\cdot,\cdot\})$. Moreover, there is a 3-dimensional dual  pre-Poisson algebra $(A, \circ,[\cdot,\cdot])$ whose nonzero products are explicitly given by
\begin{align}
 e_1 \circ e_2 = P(e_1) \bullet e_2 = e_3, \quad [e_1 ,e_2] = \{P(e_1)  ,e_2\} =  2e_3.
\end{align}
\end{ex}

Moreover, there is a dual pre-Poisson algebra  structure on the tensor product of
 a permutative algebra and a  Poisson algebra.

\begin{pro}\label{ifdualpp}
 Let $(A,\star)$ be a permutative algebra and $(B, \bullet,\{\cdot,\cdot\})$ be a Poisson algebra. Then $(A \otimes B , \circ_{A \otimes B},[\cdot,\cdot]_{A \otimes B})$ is a dual  pre-Poisson algebra where $\circ_{A \otimes B},[\cdot,\cdot]_{A \otimes B}$ are respectively  given by
 \begin{align}
 (x \otimes a) \circ_{A \otimes B} (y \otimes b) &=  x \star y  \otimes  a \bullet b,\\
    [ x \otimes a ,  y \otimes b]_{A \otimes B} &=   x \star y  \otimes  \{a ,b\},\quad \forall x,y \in A, a,b \in B.
\end{align} 
\end{pro}
\begin{proof}
It is known that $(A \otimes B , \circ_{A \otimes B})$ is a permutative algebra. For all $x,y \in A, a,b \in B$, we have
 \begin{align*}
    & [ x \otimes a ,  [ y \otimes b,z \otimes c]_{A \otimes B}]_{A \otimes B}  \\
    &=  [x \otimes a, (y \star z)  \otimes  \{b ,c\}]_{A \otimes B} =   x \star (y \star z)\otimes \{a,  \{b ,c\}\} \\
    &=     x \star (y \star z)\otimes \{\{a,  b\} ,c\} +   x \star (y \star z)\otimes \{b,  \{a ,c\}\} \\
        &=   (  x \star  y) \star z \otimes \{\{a,  b\} ,c\} +   y \star (x \star z)\otimes \{b,  \{a ,c\}\} \\
&=  [ [ x \otimes a ,  y \otimes b]_{A \otimes B},z \otimes c]_{A \otimes B}  +  [ y \otimes b ,  [ x \otimes a,z \otimes c]_{A \otimes B}]_{A \otimes B}.
\end{align*} 
Moreover, we have
 \begin{align*}
  &  [ x \otimes a ,  ( y \otimes b) \circ_{A\otimes B} (z \otimes c)]_{A \otimes B}    \\
  &= [ x \otimes a ,  (y  \star z) \otimes (b \bullet c)]_{A \otimes B} = x\star (y  \star z) \otimes \{a, b \bullet c\}\\
 &=    x\star (y  \star z) \otimes  \{a, b \}\bullet c +   x\star (y  \star z) \otimes  b \bullet \{a,  c\}\\
  &=    (x\star  y)  \star z  \otimes  \{a, b \}\bullet c +   y\star (x  \star z) \otimes  b \bullet \{a,  c\}\\ 
  &= [ x \otimes a ,    y \otimes b  ]_{A \otimes B} \circ_{A\otimes B} (z \otimes c) + (y \otimes b)  \circ_{A\otimes B}  [ x \otimes a ,  z \otimes c   ]_{A \otimes B},\\
    &   [ (x \otimes a)  \circ_{A\otimes B}  (y \otimes b),  z \otimes c ]_{A \otimes B}   \\
    &=    [ (x  \star y) \otimes (a \bullet b),  z \otimes c ]_{A \otimes B}   =     (x  \star y) \star z \otimes \{  a \bullet b , c \}  \\
   &   =   y  \star (x \star z ) \otimes  b  \bullet \{  a , c \} +   x  \star (y  \star z) \otimes a \bullet  \{ b , c \} \\
   & = (y \otimes b)  \circ_{A\otimes B}  [  x \otimes a,  z \otimes c ]_{A \otimes B}  +  (x \otimes a)  \circ_{A\otimes B} [  y \otimes b,  z \otimes c ]_{A \otimes B},  \\
    &  [ x \otimes a ,    y \otimes b  ]_{A \otimes B} \circ_{A\otimes B} (z \otimes c)   = (x  \star y) \star z  \otimes \{a ,b\} \bullet c    =  -(y  \star x) \star z  \otimes \{b ,a\} \bullet c   \\
       &   =- [ y \otimes b ,     x \otimes a ]_{A \otimes B} \circ_{A\otimes B} (z \otimes c). 
\end{align*} 
Thus $(A \otimes B , \circ_{A \otimes B},[\cdot,\cdot]_{A \otimes B})$ is a dual  pre-Poisson algebra.
\end{proof}

We give the following example to illustrate the above construction explicitly.

\begin{ex}\label{infdualppa}
Let $A=\mathbb{F}\left[t^{ \pm}\right] \partial_{1} \oplus \mathbb{F}\left[t^{ \pm}\right] \partial_{2}$, and define a  bilinear    operation   $\star: A \otimes A \rightarrow A $ by
$$
 t^{i} \partial_{m}  \star t^{j} \partial_{n} =  t^{i+j+\delta_{m,1}}\partial_{n},  \quad \forall i, j \in \mathbb{Z},\; m,n \in \{1,2\}.
$$
Then $ (A, \star)$  is a   perm algebra. Moreover, let  $B=\mathbb{F}\left[x^{ \pm}, y^{ \pm}\right]$ and define two bilinear  operations  $\bullet,\{\cdot,\cdot\}: B \otimes B \rightarrow B $ respectively by
$$
x^{a} y^{b} \bullet x^{c} y^{d}=x^{a+c} y^{b+d}, \quad \left\{x^{a} y^{b}, x^{c} y^{d}\right\}=(a d-b c) x^{a+c-1} y^{b+d-1}, \quad \forall a, b, c, d \in \mathbb{Z}.
$$
Then $ (B, \bullet,\{\cdot,\cdot\})$  is a   Poisson algebra. By Proposition \ref{ifdualpp}, there is an  infinite-dimensional dual pre-Poisson algebra $ (A \otimes B, \circ_{A \otimes B},[\cdot,\cdot]_{A \otimes B})$ with $\circ_{A \otimes B},[\cdot,\cdot]_{A \otimes B}$ respectively given by 
\begin{align*}
( t^{a_1} \partial_{m}   \otimes x^{b_1} y^{c_1}) \circ_{A \otimes B} (t^{a_2} \partial_{n}  \otimes x^{b_2} y^{c_2}) &=  t^{a_1+a_2+\delta_{m,1}}\partial_{n} \otimes x^{b_1+b_2} y^{c_1+c_2},\\
[t^{a_1} \partial_{m} \otimes x^{b_1} y^{c_1} , t^{a_2}\partial_{n}  \otimes x^{b_2} y^{c_2}]_{A \otimes B}  &=   (b_1c_2-b_2c_1)t^{a_1+a_2+\delta_{n,1}}\partial_{m}  \otimes x^{b_1+b_2-1} y^{c_1+c_2-1},
\end{align*}
where  $a_1, a_2, b_1, b_2,c_1,c_2  \in  \mathbb{Z},\; m,n \in \{1,2\}$.
\end{ex}

Next we give a construction of dual  pre-Poisson algebras  from representations of Poisson algebras, displaying a close relationship between dual  pre-Poisson algebras and Poisson algebras.

Let $(A, \bullet,\{\cdot,\cdot\})$ be a Poisson algebra,  $V$  be a  vector space and  $\mu, \rho: A \rightarrow \operatorname{End}(V)$  be two linear maps. Recall \cite{NB} that a {\bf representation} of  $(A, \bullet,\{\cdot,\cdot\})$ is a triple $(V; \mu, \rho)$ where $ \left(V; \mu\right)$  is a representation of commutative associative algebra $(A,\bullet)$ and $( V; \rho)$ is a representation of Lie algebra $(A, \{\cdot,\cdot\})$ satisfying the following  compatible conditions:
\begin{align}
\rho(x \bullet y) &= \mu(y) \rho(x)+\mu(x) \rho (y),\\
\mu(\{x, y\}) &=\rho(x)\mu(y)-\mu(y)\rho(x), \quad \forall x, y \in A.
\end{align}

\begin{pro}\label{dppfromHRofP}
 Let $(A, \bullet,\{\cdot,\cdot\})$ be a Poisson algebra and $(V; \mu, \rho)$ be a representation of $(A, \bullet,\{\cdot,\cdot\})$. Then there is a dual  pre-Poisson algebra structrue on the direct sum  $A \oplus V$  of vector spaces given by
 \begin{align}
(x+u) \circ_{A \oplus V}(y+v) &:=x \bullet  y+ \mu(x) v  ,\\
[ x+u , y+v ]_{A \oplus V} &:=\{x,y\}+\rho(x) v, \quad \forall x, y \in A, u, v \in V.
\end{align}
We denote the dual pre-Poisson algebra $(A \oplus V,\circ_{A \oplus V}, [\cdot , \cdot]_{A \oplus V})$ by $A \rightthreetimes_{\mu,\rho} V$.
\end{pro}
\begin{proof}
It is straightforward to check that $(A \oplus V, \circ_{A \oplus V})$ is a permutative algebra and  $(A \oplus V, [\cdot,\cdot]_{A \oplus V})$ is a Leibniz algebra. Moreover, for all $x, y,z \in A, u, v,w \in V$, we have
\begin{align*}
&[x +u, (y+v) \circ_{A\oplus V} (z+w)]_{A\oplus V}\\
& =[x +u,  y \bullet z+ \mu(y)w)  ]_{A\oplus V} =\{x ,   y \bullet z\}+\rho(x) \mu(y)w  \\
&=\{x, y \}\bullet z + y \bullet \{x,  z\} +\mu(\{x, y\})w + \mu(y)\rho(x)w \\
&= [x +u,  y+v  ]_{A\oplus V}\circ_{A\oplus V} (z+w)  + (y+v) \circ_{A\oplus V} [x +u,  z+w]_{A\oplus V},  \\
&[(x +u) \circ_{A\oplus V} (y+v)  , z+w ]_{A\oplus V}\\
&= [ x \bullet y + \mu(x)v , z+w ]_{A\oplus V} = \{ x \bullet y,z\}+\rho(x \bullet y)  w  \\
&= x \bullet \{y,z\}+ y  \bullet \{x,z\}    + \mu(y) \rho(x)w +\mu(x) \rho (y)w \\
&=    (x+u) \circ_{A\oplus V} [y +v,  z+w]_{A\oplus V}  + (y+v) \circ_{A\oplus V} [x +u,  z+w]_{A\oplus V},  \\
&[x +u,  y+v  ]_{A\oplus V}\circ_{A\oplus V} (z+w)  \\
&= \{x, y \}\bullet z +\mu(\{x, y\})w = -\{y, x\}\bullet z -\mu(\{y, x\})w   \\
&=- [  y+v ,x +u ]_{A\oplus V}\circ_{A\oplus V} (z+w).
\end{align*}
Thus $(A \oplus V, \circ_{A \oplus V}, [\cdot,\cdot]_{A \oplus V})$ is a  dual  pre-Poisson algebra.
\end{proof}

\begin{ex}\label{bigdualpre}
(1) Let $(A, \bullet,\{\cdot,\cdot\})$ be a Poisson algebra,  Then  $(A;  L_{\bullet},\ad_{\{\cdot,\cdot\}})$ is a representation of  $(A, \bullet,\{\cdot,\cdot\})$, where $\ad_{\{\cdot,\cdot\}}: A \to \gl(A)$ is  the  adjoint representation  of Lie algebra $(A,  \{\cdot,\cdot\})$ and is given by $\ad_{\{\cdot,\cdot\}}(x)(y): = \{x,y\}$ for all $x,y \in A$. In particular, $(A^*; -L_{\bullet}^*,\ad_{\{\cdot,\cdot\}}^*)$ is also a representation of  $(A, \bullet,\{\cdot,\cdot\})$ \cite{NB}. Thus by Proposition \ref{dppfromHRofP}, there are two dual  pre-Poisson algebras $A \rightthreetimes_{L_{\bullet},\ad_{\{\cdot,\cdot\}}} A$ and $A \rightthreetimes_{-L_{\bullet}^*,\ad_{\{\cdot,\cdot\}}^*} A^*$.

(2) Continuing with Example \ref{Ex:Pois}, let $\{e_1^*, e_2^*, e_3^*\}$ be the dual basis of $\{e_1, e_2, e_3\}$. Then there is a 6-dimensional dual pre-Poisson algebra  $A \rightthreetimes_{-L_{\bullet}^*,\ad_{\{\cdot,\cdot\}}^*} A^*$ whose nonzero products are explicitly given by 
\begin{align}
e_1 \circ e_2 &= e_2 \circ e_1 = e_3,  &&   e_1 \circ e_3^* = e_2^*,  &&    e_2 \circ e_3^* = e_1^*,  \label{exdual1}\\
[e_1, e_2] &= -[e_2, e_1] = 2e_3,  &&   [e_1 , e_3^*] = -2e_2^*,  &&   [e_2 , e_3^*]= 2e_1^*,  \label{exdual2}
\end{align}

\end{ex}

Next we give a construction of dual pre-Poisson algebras by a pair of commuting derivations on permutative algebras in parallel to the classical construction of Poisson algebras by a pair of commuting  derivations on commutative  associative algebras.

Recall that a {\bf derivation} of a permutative algebra $(A,\circ)$ is defined to be a linear map $D:A \to A$ satisfying the Leibniz rule, i.e., $D(x \circ y) = D(x) \circ y + x \circ D(y)$ for all $x,y \in A$.  

\begin{pro}\label{sdztodpp}
Let $(A,\circ)$ be a permutative algebra and $D_1,D_2: A \to A$ be commuting derivations on $(A,\circ)$. Define new operation $[\cdot,\cdot]: A \otimes A \to A$ on $A$ by
\begin{align}
 [x,y] = D_1(x)\circ D_2(y) - D_2(x) \circ D_1(y) ,\quad \forall x,y \in A. \label{DDdualprep}
\end{align} 
Then $(A,  [\cdot,\cdot])$ is a  Leibniz algebra. Moreover, $(A, \circ,[\cdot,\cdot])$ is a dual  pre-Poisson algebra.
\end{pro}

\begin{proof}
Let $x,y,z\in A$, we have
\begin{align*}
 &[x,[y,z]] -  [[x,  y],z] - [y,[x,z]]\\
 &=D_1(x)\circ \big(D_2D_1(y) \circ D_2(z) + D_1(y)\circ D_2D_2(z) - D_2D_2(y) \circ D_1(z) -D_2(y) \circ D_2D_1(z)\big)\\
 &\hspace{0.5cm}-D_2(x)\!\circ\! \big(D_1D_1(y) \circ D_2(z) + D_1(y)\circ D_1D_2(z) - D_1D_2(y) \circ D_1(z) -D_2(y)\!\circ\! D_1D_1(z)\big)\\
 &\hspace{0.5cm}-\big(D_1D_1(x) \circ D_2(y) + D_1(x) \circ D_1D_2(y) - D_1D_2(x) \circ D_1(y) -D_2(x) \circ D_1D_1(y)\big)D_2(z)\\
  &\hspace{0.5cm}+\big(D_2D_1(x) \circ D_2(y) + D_1(x)\circ D_2D_2(y) - D_2D_2(x) \circ D_1(y) -D_2(x) \circ D_2D_1(y)\big)D_1(z)\\
&\hspace{0.5cm}-D_1(y)\!\circ \!\big(D_2D_1(x) \circ D_2(z) + D_1(x) \!\circ\! D_2D_2(z) - D_2D_2(x) \circ D_1(z) -D_2(x) \circ D_2D_1(z)\big)\\
 &\hspace{0.5cm}+D_2(y) \!\circ \!\big(D_1D_1(x) \circ D_2(z) + D_1(x) \!\circ\! D_1D_2(z) - D_1D_2(x) \circ D_1(z) -D_2(x) \circ D_1D_1(z)\big)\\
 &=0.
\end{align*}
Thus $(A,  [\cdot,\cdot])$ is a  Leibniz algebra.  Moreover,
\begin{align*}
&[x, y \circ z] - [x, y ]\circ z - y \circ[x,  z]  \\
&=D_1(x)\circ (D_2(y) \circ z +   y \circ D_2(z) ) - D_2(x)\circ (D_1(y)\circ z +  y \circ D_1(z))\\
&\hspace{0.5cm}-(D_1(x)\circ D_2(y)) \circ z + (D_2(x) \circ D_1(y)) \circ z - y \!\circ\! (D_1(x) \!\circ\! D_2(z))  + y \!\circ\! (D_2(x) \circ D_1(z))\\
&=0.\\
&[x \circ y,   z] - x \circ [ y,   z] - y \circ [ x,   z]  \\
&=( D_1(x) \circ y) \circ D_2(z) +  ( x \circ D_1(y)) \circ D_2(z) - ( D_2(x) \circ y) \circ D_1(z) - ( x \circ D_2(y)) \circ D_1(z)\\
&\hspace{0.5cm}\!- \!x \!\circ\! (D_1(y)\! \circ\! D_2(z) )+   x \circ (D_2(y)  \circ D_1(z))  - y \circ (D_1(x) \circ D_2(z) )+   y \circ (D_2(x) \circ D_1(z))  \\
&=0.\\
&[x, y ]\circ z + [y, x]\circ z \\
&=(D_1(x)\circ D_2(y) - D_2(x) \circ D_1(y)+D_1(y)\circ D_2(x) - D_2(y) \circ D_1(x) ) \circ z  \\
&=0.
\end{align*}
Thus $(A, \circ,[\cdot,\cdot])$ is a dual  pre-Poisson algebra.
\end{proof}

{}

As an application of dual pre-Poisson algebras, we  construct  compatible Lie algebras, or more precisely, compatible pre-Lie algebras from dual  pre-Poisson algebras by Koszul duality.  

Recall \cite{IZGVVS1,IZGVVS2} that a {\bf compatible  Lie algebra} is a triple $(A,\{\cdot,\cdot\},\{\cdot,\cdot\}')$ where $(A,\{\cdot,\cdot\})$ and $(A,\{\cdot,\cdot\}')$ are  Lie algebras such that the following bilinear operation 
\begin{align*}
[x, y] = k_1\{x, y\} + k_2\{x, y\}',\quad \forall k_1, k_2 \in \mathbb{F}, x, y \in A, 
\end{align*}
also defines a Lie algebra structure on $A$.

{}

\begin{defi}\cite{WB}  A {\bf compatible  pre-Lie algebra} is a triple $(A,\diamond,\diamond')$ where $(A,\diamond)$ and $(A,\diamond')$ are  pre-Lie algebras such that the following bilinear operation 
$$
x \circ y=k_{1} x \diamond y+k_{2} x \diamond' y,  \quad \forall k_1, k_2 \in \mathbb{F}, x, y \in A, 
$$
also defines a pre-Lie algebra structure on  $A$.
\end{defi}

\begin{pro}\cite{WB}\label{Comprelie}
Let $(A, \diamond)$  and  $(A, \diamond')$ be two pre-Lie algebras.  Then  $( A,\diamond, \diamond')$  is a compatible pre-Lie algebra if and only if  the following condition holds
\begin{align*}
&(x \diamond y) \diamond' z-x \diamond'(y \diamond z)+(x \diamond' y) \diamond z-x \diamond(y \diamond' z)\\
& =(y \diamond x) \diamond' z-y \diamond'(x \diamond z)+(y \diamond' x) \diamond z-y \diamond(x \diamond' z),\quad \forall    x, y, z \in A.
\end{align*}
\end{pro}

\begin{pro}\cite{WB}\label{compreyoLie}
Let  $( A,\diamond, \diamond')$  be a compatible pre-Lie algebra. Then $(A,\{\cdot,\cdot\},\{\cdot,\cdot\}')$ is a compatible Lie algebra with $\{\cdot,\cdot\},\{\cdot,\cdot\}'$ respectively  given by
$$
\{x,y\}=x \diamond y - y \diamond x, \quad \{x,y\}'=x \diamond' y-y \diamond' x, \quad \forall x,y \in A . 
$$
We  denoted this compatible Lie algebra by  $\left(\mathfrak{g}(A),\{\cdot,\cdot\},\{\cdot,\cdot\}'\right)$, or simply by $\mathfrak{g}(A)$.
\end{pro}

\begin{pro}\label{ComFdualpp}
 Let $(A, \circ,[\cdot,\cdot])$  be a  dual pre-Poisson algebra  and  $(B, \diamond,\odot )$ be a pre-Poisson algebra. Then $(A \otimes B ,\diamond_{1})$ and $(A \otimes B ,\diamond_{2})$ are  pre-Lie algebras where $\diamond_{1},\diamond_{2}$ are respectively  given by
 \begin{align}
 (x \otimes a) \diamond_{1} (y \otimes b) &= (x \circ y) \otimes (a \diamond b),\\
  (x \otimes a) \diamond_{2} (y \otimes b) &=  [x , y]  \otimes  (a \odot b),\quad \forall x,y \in A, a,b \in B.
\end{align} 
Moreover, $(A \otimes B ,\diamond_{1} ,\diamond_{2})$ is a  compatible pre-Lie algebra and  $\mathfrak{g}(A \otimes B)$ is a  compatible  Lie algebra.


{}
\end{pro}
\begin{proof}
 Let $x,y,z\in A$, $a,b,c \in B$, we have
\begin{align*}
&((x \otimes a) \diamond_{1} (y \otimes b)) \diamond_{2} (z \otimes c)+((x \otimes a)  \diamond_{2} (y \otimes b)) \diamond_{1} (z \otimes c)\\
&-(x \otimes a) \diamond_{2}((y \otimes b) \diamond_{1} (z \otimes c))-(x \otimes a) \diamond_{1}((y \otimes b) \diamond_{2} (z \otimes c))\\
&=[x \circ y,z] \otimes (a  \diamond b)  \odot c + [x,y]\circ  z  \otimes (a  \odot b) \diamond   c\\
&\quad-[x , y \circ z] \otimes  a  \odot  (b  \diamond c) - x \circ  [ y,z] \otimes a  \diamond  (b  \odot  c)\\
&=x \circ[ y,z] \otimes (a  \diamond b)  \odot c + y \circ[ x,z] \otimes (a  \diamond b)  \odot c + [x,y]\circ  z  \otimes (a  \odot b) \diamond   c\\
&\quad-[x , y ]\circ z  \otimes  a  \odot  (b  \diamond c)-y \circ  [x , z]   \otimes  a  \odot  (b  \diamond c)  - x \circ  [ y,z] \otimes a  \diamond  (b  \odot  c)\\
&=x \circ[ y,z] \otimes \big((a  \diamond b)  \odot c - a  \diamond  (b  \odot  c)\big) + y \circ[ x,z] \otimes \big((a  \diamond b)  \odot c - a  \odot  (b  \diamond c) \big)    \\
&\quad + [x , y ]\circ z  \otimes \big( (a  \odot b) \diamond   c - a  \odot  (b  \diamond c) \big)   \\
&=y \circ[ x,z] \otimes \big((b \diamond a)  \odot c - b  \diamond  (a \odot  c)\big) + x \circ[ y,z] \otimes \big((b  \diamond a)  \odot c - b  \odot  (a  \diamond c) \big)    \\
&\quad + [y , x]\circ z  \otimes \big( (b  \odot a) \diamond   c - b  \odot  (a  \diamond c) \big)   \\
&=((y \otimes b) \diamond_{1} (x \otimes a)) \diamond_{2} (z \otimes c)+((y \otimes b)  \diamond_{2} (x \otimes a)) \diamond_{1} (z \otimes c)\\
&\quad-(y \otimes b) \diamond_{2}((x \otimes a) \diamond_{1} (z \otimes c))-(y \otimes b) \diamond_{1}((x \otimes a) \diamond_{2} (z \otimes c)).
\end{align*}
By Proposition  \ref{Comprelie}, $(A \otimes B ,\diamond_{1} ,\diamond_{2})$ is a  compatible pre-Lie algebra. The rest follows from Proposition \ref{compreyoLie}. This completes the proof.  \end{proof}

\subsection{Diassociative formal deformation of permutative algebras}

Observing that Poisson and  pre-Poisson algebras emerge as deformation limits of commutative associative   and  Zinbiel algebras respectively \cite{MA,MR}, we are motivated to explore whether the dual pre-Poisson algebra can be realized as a deformation limit of a suitable algebra as well. The answer is affirmative. In this subsection, we introduce diassociative deformation quantization of permutative algebras and show that dual pre-Poisson algebras arise as their semi-classical limits. This illustrates that dual pre-Poisson algebras are the same to dialgebras as Poisson algebras to associative algebras.

\begin{defi}\cite{L4}
 A {\bf diassociative algebra} (or dialgebra) is a triple $(A,\triangleright,\triangleleft)$ where $(A,\triangleright)$ and $(A,\triangleleft)$ are associative algebras satisfying the following conditions 
\begin{align*}
(x \triangleleft y) \triangleleft z  = x \triangleleft (y \triangleright z),\quad (x \triangleright y) \triangleleft z  = x \triangleright (y \triangleleft z),\quad
(x \triangleleft y) \triangleright z  = x  \triangleright (y \triangleright z),
\end{align*}
for all $x ,y,z \in A$.
\end{defi}

\begin{pro}\cite{L4}\label{diatoLeib}
    Let $(A,\triangleright,\triangleleft)$ be a dialgebra. Then the bilinear operation $[\cdot,\cdot]: A \otimes A \to A$ given by 
\begin{align}
  [x, y]  = x \triangleright y - y \triangleleft x, \quad \forall x,y \in A, \label{inducedLeib}
\end{align}
defines a Leibniz algebra $(A,[\cdot, \cdot])$, which is called the {\bf sub-adjacent Leibniz algebra} of $(A,\triangleright,\triangleleft)$
\end{pro}

\begin{rmk}\label{baesF}
In fact, the definition  and result  above remain valid if we replace the base field $\mathbb{F}$ with a commutative ring. 
\end{rmk}

\begin{defi}
 Let $(  A, \circ )$ be a permutative algebra. A {\bf diassociative formal deformation} of  $A$  is a sequence of bilinear operations  $\triangleright_{i},\triangleleft_{i}: A \otimes A \rightarrow A$  for  $i \geqslant 0$ with $x \triangleright_{0} y := x \circ y = :y \triangleleft_{0}  x$ for all $x,y \in A$ such that   $(A[[\hbar]],\triangleright_{\hbar},\triangleleft_{\hbar} )$ is a dialgebra where    the  $\mathbb{F}[[\hbar]]$-bilinear operations   $\triangleright_{\hbar},\triangleleft_{\hbar}$  on  $A[[\hbar]]$ are defined by 
\begin{align*}
x \triangleright_{\hbar} y=\sum_{i=0}^{+\infty}   x \triangleright_{i} y \hbar^{i} , \quad x \triangleleft_{\hbar} y=\sum_{i=0}^{+\infty}   x \triangleleft_{i} y \hbar^{i}, \quad \forall x, y \in A.
\end{align*}
\end{defi}

Note that the rule of dialgebra operations  $\triangleright_{\hbar},\triangleleft_{\hbar}$   on  $A[[\hbar]] $ is equivalent to
\begin{align}
&\sum_{i+j=k} \big( (x \triangleright_{i} y) \triangleright_{j}  z- x \triangleright_{i} (y \triangleright_{j}  z)  \big)=0,\label{deffordi1}\\
&\sum_{i+j=k} \big( (x \triangleleft_{i} y) \triangleleft_{j}  z- x \triangleleft_{i} (y \triangleright_{j}  z)  \big)=0,\\
&\sum_{i+j=k} \big( (x \triangleright_{i} y) \triangleleft_{j}  z- x \triangleright_{i} (y \triangleleft_{j}  z)  \big)=0,\\
&\sum_{i+j=k} \big( (x \triangleleft_{i} y) \triangleright_{j}  z- x \triangleright_{i} (y \triangleright_{j}  z)  \big)=0,\\
&\sum_{i+j=k} \big( (x \triangleleft_{i} y) \triangleleft_{j}  z- x \triangleleft_{i} (y \triangleleft_{j}  z)  \big)=0,\quad \forall x,y,z \in A, k \geqslant 0.\label{deffordi5}
\end{align}

\begin{thm}\label{dppofdeform}
 Let  $(A, \circ)$  be a permutative algebra and   $(A[[\hbar]],\triangleright_{\hbar},\triangleleft_{\hbar} )$  be a  diassociative formal deformation of $A$. Define a new $\mathbb{F}$-bilinear operation $[\cdot,\cdot]: A \otimes A \rightarrow A$ by 
\begin{align*}
[x , y] = x \triangleright_{1} y - y \triangleleft_{1} x, \quad \forall x, y \in A.
\end{align*}
Then  $(A, \circ, [\cdot,\cdot])$ is a dual pre-Poisson algebra, which is called   the {\bf semi-classical limit} of $(A[[\hbar]],\triangleright_{\hbar},\triangleleft_{\hbar} )$.  The  dialgebra  $(A[[\hbar]],\triangleright_{\hbar},\triangleleft_{\hbar} )$ is called the {\bf diassociative deformation quantization} of   $(A, \circ, [\cdot,\cdot])$.
\end{thm}
\begin{proof}
Define the  $\mathbb{F}[[\hbar]]$-bilinear operation $[\cdot,\cdot]_{\hbar}$ on $A[[\hbar]]$ by 
\begin{align*}
[x,y]_{\hbar} = x \triangleright_{h} y - y \triangleleft_{h} x = [x,y]\hbar + (x \triangleright_{2} y - y \triangleleft_{2} x)\hbar^2+\cdots
\end{align*}
By Proposition \ref{diatoLeib} and Remark \ref{baesF}, $(A[[\hbar]],[\cdot,\cdot]_{\hbar})$ is a Leinbiz algebra. It is obvious that the $\hbar^2$-terms of the Leibniz rule (Eq.~\eqref{Leibniz}) for $[\cdot,\cdot]_{\hbar}$ can imply  the Leibniz rule  for $[\cdot,\cdot]$. Thus $(A,  [\cdot,\cdot])$ is a Leibniz algebra.

For $k = 1$ in Eqs.~\eqref{deffordi1}-\eqref{deffordi5}, we have
\begin{align*}
(x \circ  y)  \triangleright_{1} z &= x \circ (y  \triangleright_{1} z) = (y \circ  x)  \triangleright_{1} z  = y \circ (x \triangleright_{1} z) = (y \triangleright_{1}  x)  \circ  z \\
&=  y \triangleright_{1}  (x  \circ  z) = x \triangleright_{1}  (y  \circ  z) = (x \triangleright_{1}   y)  \circ  z,\\
(x \circ  y)  \triangleleft_{1} z &= x \circ (y  \triangleleft_{1} z) = y \triangleleft_{1} (x \circ z) =(x \triangleleft_{1} z )\circ y = (x \triangleright_{1} z )\circ y, \quad \forall x,y,z \in A.
\end{align*}
Then 
\begin{align*}
&[x, y \circ z] -[x, y ]\circ z - y \circ[x,  z]   \\
&= x \triangleright_{1} (y \circ z) - (y \circ z) \triangleleft_{1} x - (x \triangleright_{1} y -y \triangleleft_{1} x)\circ z - y \circ (x \triangleright_{1} z -z \triangleleft_{1} x)  \\
&= x \triangleright_{1} (y \circ z) - (x \triangleright_{1} y )\circ z   + (y \triangleleft_{1} x)\circ z - (y \circ z) \triangleleft_{1} x\\
&=0.\\
&[x \circ y,   z] - x \circ [ y,   z] - y \circ [ x,   z] \\
&= (x \circ y)\triangleright_{1} z - z \triangleleft_{1} (x \circ y) - x \circ (y \triangleright_{1} z -z \triangleleft_{1} y)  - y \circ (x \triangleright_{1} z -z \triangleleft_{1} x)  \\
&=0.\\
&[x, y ]\circ z  + [y, x]\circ z\\
&= (x \triangleright_{1} y -y \triangleleft_{1} x + y \triangleright_{1} x -x \triangleleft_{1} y)\circ z \\
&=0.
\end{align*}
Thus $(A, \circ, [\cdot,\cdot])$ is a dual pre-Poisson algebra.
\end{proof}
 
\begin{rmk}
As Proposition \ref{Koszulfordpp} indicates that the operads $\mathsf{DualprePois}$ and $\mathsf{Dias}$ have the same Hilbert-Poincaré series, it follows that their associated operad dimensions coincide. Thus  Proposition \ref{dppofdeform} establishes dual pre-Poisson algebras as the semi-classical limits of diassociative formal deformations of permutative algebras, thereby lifting the classical relationship in which Poisson algebras arise as the classical limits of associative formal deformations of commutative associative algebras to the level of operadic duplicator \cite{JBGN}.
\end{rmk}

\subsection{Quadratic dual pre-Poisson algebras}

In this subsection, we introduce the notions of representations of  dual pre-Poisson algebras and quadratic dual pre-Poisson algebras.

\begin{defi} \label{defrep} A \textbf{representation} of a dual pre-Poisson algebra $(A, \circ,[\cdot,\cdot])$  is a quintuple $(V;l_{\circ}, r_{\circ}, l_{[\cdot,\cdot]}, $ $r_{[\cdot,\cdot]})$ where $V$ is a vector space and  $l_{\circ}, r_{\circ}, l_{[\cdot,\cdot]}, r_{[\cdot,\cdot]}: A \rightarrow \End(V)$ are  linear maps satisfying  
\begin{enumerate}
\item $(V;l_{\circ}, r_{\circ})$ is a {\bf representation of permutative algebra $(A, \circ)$}:
\begin{align}
r_{\circ}(x)r_{\circ}(y) &= r_{\circ}(y \circ x) = l_{\circ}(y)r_{\circ}(x) =r_{\circ}(x)l_{\circ}(y) ,\label{permrep1}\\
l_{\circ}(x \circ y)&= l_{\circ}(x)l_{\circ}(y)  = l_{\circ}(y)l_{\circ}(x), \quad \forall x,y  \in A. \label{permrep2}
\end{align}
\item $(V; l_{[\cdot,\cdot]}, r_{[\cdot,\cdot]})$ is a {\bf representation of  Leibniz algebra $(A,[\cdot,\cdot])$}:
\begin{align}
 l_{[\cdot,\cdot]}( [x,y])&= l_{[\cdot,\cdot]}(x)l_{[\cdot,\cdot]}(y) -  l_{[\cdot,\cdot]}(y)l_{[\cdot,\cdot]}(x),\label{Leibrep1}\\
r_{[\cdot,\cdot]}(  [x,y]  )&= r_{[\cdot,\cdot]}(y)r_{[\cdot,\cdot]}(x) + l_{[\cdot,\cdot]}(x)r_{[\cdot,\cdot]}(y),\\
r_{[\cdot,\cdot]}(x)r_{[\cdot,\cdot]}(y) &= -r_{[\cdot,\cdot]}(x)l_{[\cdot,\cdot]}(y), \quad \forall x,y  \in A. \label{Leibrep2}
\end{align}
\item The following conditions:
\begin{align}
r_{[\cdot,\cdot]}(x \circ y) & =r_{\circ}(y)r_{[\cdot,\cdot]}(x)+l_{\circ}(x)r_{[\cdot,\cdot]}(y), \label{dpp1} \\
l_{[\cdot,\cdot]}(x)r_{\circ}(y) & =r_{\circ}(y)l_{[\cdot,\cdot]}(x)+ r_{\circ}([x,y]), \label{dpp2} \\
l_{[\cdot,\cdot]}(x)l_{\circ}(y) & =l_{\circ}([x,y])+  l_{\circ}(y)l_{[\cdot,\cdot]}(x), \label{dpp3} \\
r_{[\cdot,\cdot]}(x)r_{\circ}(y) & =r_{\circ}([y,x])+  l_{\circ}(y)r_{[\cdot,\cdot]}(x), \label{dpp4} \\
l_{[\cdot,\cdot]}(x \circ y) & =l_{\circ}(x)l_{[\cdot,\cdot]}(y)+l_{\circ}(y)l_{[\cdot,\cdot]}(x), \label{dpp5} \\
r_{[\cdot,\cdot]}(x)(l_{\circ}-r_{\circ})(y) & = 0 , \label{dpp6} \\
r_{\circ}(x)(l_{[\cdot,\cdot]} + r_{[\cdot,\cdot]})(y) & = 0, \label{dpp7} \\
l_{\circ}([x,y]+[y,x]) & =0, \quad  \forall x, y \in A. \label{dpp8} 
\end{align}
\end{enumerate}
Representations  $(V_1;l_{\circ_1}, r_{\circ_1}, l_{[\cdot,\cdot]_1}, r_{[\cdot,\cdot]_1})$  and $(V_2;l_{\circ_2}, r_{\circ_2}, l_{[\cdot,\cdot]_2}, r_{[\cdot,\cdot]_2})$ of a dual pre-Poisson  algebra $(A,\circ,[\cdot,\cdot])$ are
\textbf{equivalent} if there exists a linear isomorphism $\varphi:
V_{1} \rightarrow V_{2}$  such that
\begin{align*}
\varphi\left(l_{\circ_1}(x) v\right) &=l_{\circ_2}(x) \varphi(v),&&
\varphi\left(r_{\circ_1}(x) v\right)=r_{\circ_2}(x) \varphi(v), \\
\varphi\left(l_{[\cdot,\cdot]_1}(x) v\right) &=l_{[\cdot,\cdot]_2}(x) \varphi(v),&&\varphi\left(r_{[\cdot,\cdot]_1}(x) v\right)=r_{[\cdot,\cdot]_2}(x) \varphi(v),
\;\;\forall x \in A, v \in V_{1}.
\end{align*}
\end{defi}

\begin{ex}
Let  $(A, \circ,[\cdot,\cdot])$  be a dual pre-Poisson algebra. Then  $(A; L_{\circ},R_{\circ},L_{[\cdot,\cdot]},R_{[\cdot,\cdot]})$  is a representation of  $(A, \circ,[\cdot,\cdot])$, which is called the \textbf{regular representation} of $(A, \circ,[\cdot,\cdot])$.
\end{ex}

It is straightforward to  obtain the following conclusion.
\begin{pro} \label{repofdppa}
Let  $(A, \circ,[\cdot,\cdot])$  be a dual pre-Poisson algebra. Let $V$  be a vector space and  $l_{\circ}, r_{\circ}, l_{[\cdot,\cdot]}, r_{[\cdot,\cdot]}: A \rightarrow \End(V)$ be linear maps. Then $(V;l_{\circ}, r_{\circ}, l_{[\cdot,\cdot]}, r_{[\cdot,\cdot]})$ is a representation of $(A, \circ,[\cdot,\cdot])$ if and only if there is a dual pre-Poisson algebra structure on the direct sum  $A \oplus V$  of vector spaces given by
\begin{align}
(x+u) \circ_{A \oplus V}(y+v) &:=x \circ  y+ l_{\circ}(x) v+r_{\circ}(y) u,\\
[ x+u , y+v ]_{A \oplus V} &:=[x,y]+l_{[\cdot,\cdot]}(x) v+r_{[\cdot,\cdot]}(y) u, \quad \forall x, y \in A, u, v \in V.
\end{align}
We denote the dual pre-Poisson algebra  $\left(A \oplus V,  \circ_{A \oplus V}, [\cdot, \cdot]_{A \oplus V} \right)$  by  $A \ltimes_{l_{\circ}, r_{\circ}, l_{[\cdot,\cdot]}, r_{[\cdot,\cdot]}} V$.
\end{pro}

Now we investigate the dual representation in the context of dual pre-Poisson algebras.


\begin{pro}
Let  $(A, \circ,[\cdot,\cdot])$  be a dual pre-Poisson algebra and $(V;l_{\circ}, r_{\circ}, l_{[\cdot,\cdot]}, r_{[\cdot,\cdot]})$ be a representation of $(A, \circ,[\cdot,\cdot])$. Then
$(V^*;-l_{\circ}^*, -l_{\circ}^*+r_{\circ}^*, l_{[\cdot,\cdot]}^*, -l_{[\cdot,\cdot]}^*-r_{[\cdot,\cdot]}^*)$ is a representation of $(A, \circ,[\cdot,\cdot])$, which is called the {\bf dual representation} of $(V;l_{\circ}, r_{\circ}, l_{[\cdot,\cdot]}, r_{[\cdot,\cdot]})$. Moreover, there is a dual pre-Poisson algebra structure  $A \ltimes_{-l_{\circ}^*, -l_{\circ}^*+r_{\circ}^*, l_{[\cdot,\cdot]}^*, -l_{[\cdot,\cdot]}^*-r_{[\cdot,\cdot]}^*} V^*$ on the direct sum  $A \oplus V^*$  of vector spaces.
\end{pro}
\begin{proof}
It is known that $(V^*;-l_{\circ}^*, -l_{\circ}^*+r_{\circ}^*)$ is a representation of permutative algebra $(A, \circ)$ \cite[Proposition 3.14]{BYZ} and $(V^*; l_{[\cdot,\cdot]}^*, -l_{[\cdot,\cdot]}^*-r_{[\cdot,\cdot]}^*)$ is a representation of Leibniz algebra $(A, [\cdot,\cdot])$ \cite[Lemma 2.10]{TS}. Moreover, for all $x, y  \in A, v\in V,  w^* \in V^*$, we have
\begin{align*}
&\left \langle \Big(( -l_{[\cdot,\cdot]}^*-r_{[\cdot,\cdot]}^*)(x \circ y) + (l_{\circ}^*-r_{\circ}^*) (y)  (-l_{[\cdot,\cdot]}^*-r_{[\cdot,\cdot]}^*) (x)- l_{\circ}^*(x)(  l_{[\cdot,\cdot]}^*+r_{[\cdot,\cdot]}^*)(y) \Big)w^*,v \right \rangle    \\
&=\left \langle w^*,\Big((  l_{[\cdot,\cdot]} +r_{[\cdot,\cdot]} )(x \circ y) - ( l_{[\cdot,\cdot]} +r_{[\cdot,\cdot]} ) (x)(l_{\circ} -r_{\circ} ) (y)- (  l_{[\cdot,\cdot]} +r_{[\cdot,\cdot]} )(y)l_{\circ} (x) \Big)v \right \rangle    \\
&=\left \langle w^*,\Big((  l_{[\cdot,\cdot]} +r_{[\cdot,\cdot]} )(x \circ y) -  l_{[\cdot,\cdot]}  (x)(l_{\circ} -r_{\circ} ) (y)- (  l_{[\cdot,\cdot]} +r_{[\cdot,\cdot]} )(y)l_{\circ} (x) \Big)v \right \rangle    \\
&=  \Big\langle w^*,\Big(( r_{\circ}(y)r_{[\cdot,\cdot]}(x)+l_{\circ}(x)r_{[\cdot,\cdot]}(y)+l_{\circ}(x)l_{[\cdot,\cdot]}(y)+l_{\circ}(y)l_{[\cdot,\cdot]}(x)+r_{\circ}(y)l_{[\cdot,\cdot]}(x)\\
& \hspace{1.4cm} + r_{\circ}([x,y]) - l_{\circ}([x,y])- l_{\circ}(y)l_{[\cdot,\cdot]}(x)- (  l_{[\cdot,\cdot]} +r_{[\cdot,\cdot]} )(y)l_{\circ} (x) \Big)v   \Big\rangle   \\
 &=  \langle w^*,\Big((  l_{\circ}(x)l_{[\cdot,\cdot]}(y)    - l_{\circ}([x,y]) -    l_{[\cdot,\cdot]}(y) l_{\circ} (x)  \Big)v   \rangle  = 0, \\
& \left \langle \Big(l_{[\cdot,\cdot]}^*(x)(-l_{\circ}^*+r_{\circ}^*)(y) +( l_{\circ}^*-r_{\circ}^*)(y)l_{[\cdot,\cdot]}^*(x)+ (l_{\circ}^*-r_{\circ}^*)([x,y]) \Big)w^*,v \right \rangle    \\
&=\left \langle w^*,\Big((r_{\circ}-l_{\circ})(y) l_{[\cdot,\cdot]}(x)+l_{[\cdot,\cdot]} (x)( l_{\circ} -r_{\circ} )(y)+ (r_{\circ}-l_{\circ}  )([x,y]) \Big)v \right \rangle    \\
&=\left \langle w^*,\Big(l_{[\cdot,\cdot]}(x)r_{\circ}(y) -l_{[\cdot,\cdot]}(x)l_{\circ}(y) +l_{[\cdot,\cdot]} (x)( l_{\circ} -r_{\circ} )(y)  \Big)v \right \rangle  =0,  \\
& \left \langle \Big(-l_{[\cdot,\cdot]}^*(x)l_{\circ}^*(y)+l_{\circ}^*([x,y])+ l_{\circ}^*(y)l_{[\cdot,\cdot]}^*(x) \Big)w^*,v \right \rangle    \\
&=\left \langle w^*, \Big( l_{[\cdot,\cdot]} (x) l_{\circ} (y)-l_{\circ} (y)l_{[\cdot,\cdot]} (x)-l_{\circ} ([x,y])\Big)v \right \rangle =0,    \\
& \left \langle \Big(( l_{[\cdot,\cdot]}^*+r_{[\cdot,\cdot]}^*)(y)( l_{\circ}^*-r_{\circ}^*)(x) +(l_{\circ}^*-r_{\circ}^*)([x,y])- l_{\circ}^*(x)( l_{[\cdot,\cdot]}^*+r_{[\cdot,\cdot]}^*)(y) \Big)w^*,v \right \rangle    \\
&=  \left \langle w^*,\Big(( l_{\circ} -r_{\circ} )(x)( l_{[\cdot,\cdot]} +r_{[\cdot,\cdot]} )(y) +(r_{\circ} - l_{\circ})([x,y])- ( l_{[\cdot,\cdot]} +r_{[\cdot,\cdot]} )(y) l_{\circ} (x)\Big)v \right \rangle    \\
&=  \left \langle w^*,\Big(  l_{\circ}  (x) l_{[\cdot,\cdot]}   (y)  - l_{[\cdot,\cdot]}(x)l_{\circ}(y)+ l_{\circ}(y)l_{[\cdot,\cdot]}(x)-  l_{[\cdot,\cdot]}  (y) l_{\circ} (x)\Big)v \right \rangle    \\
&=  - \left \langle w^*,   l_{\circ} ([x,y]+[y,x]) v \right \rangle   =0,\\
& \left \langle \Big(l_{[\cdot,\cdot]}^*(x \circ y) +l_{\circ}^*(x)l_{[\cdot,\cdot]}^*(y)+l_{\circ}^*(y)l_{[\cdot,\cdot]}^*(x)\Big)w^*,v \right \rangle    \\
&=  \left \langle w^*,\Big(-l_{[\cdot,\cdot]} (x \circ y) +l_{[\cdot,\cdot]} (y)l_{\circ} (x)+l_{[\cdot,\cdot]} (x)l_{\circ} (y)\Big)v \right \rangle    \\
&=  \left \langle w^*,\Big(-l_{\circ}(x)l_{[\cdot,\cdot]}(y)-l_{\circ}(y)l_{[\cdot,\cdot]}(x) +l_{[\cdot,\cdot]} (y)l_{\circ} (x)+l_{[\cdot,\cdot]} (x)l_{\circ} (y)\Big)v \right \rangle    \\
&=    \left \langle w^*,   l_{\circ} ([x,y]+[y,x]) v \right \rangle   =0,\\
& \left \langle \Big((l_{[\cdot,\cdot]}^*+r_{[\cdot,\cdot]}^*)(x) r_{\circ}^*(y)\Big)w^*,v \right \rangle   = \left \langle w^*,\Big(r_{\circ}(y)(l_{[\cdot,\cdot]} +r_{[\cdot,\cdot]} )(x) \Big)v \right \rangle =0, \\
& \left \langle \Big((l_{\circ}^*-r_{\circ}^*)(x) r_{[\cdot,\cdot]}^*(y)\Big)w^*,v \right \rangle   = \left \langle w^*,\Big( r_{[\cdot,\cdot]} (y)(l_{\circ} -r_{\circ} )(x)\Big)v \right \rangle =0, \\
& \left \langle \Big(-l_{\circ}^*([x,y]+[y,x])\Big)w^*,v \right \rangle   =\left \langle w^*,\Big( l_{\circ} ([x,y]+[y,x])\Big)v \right \rangle  = 0.
\end{align*}
Thus $(V^*;-l_{\circ}^*, -l_{\circ}^*+r_{\circ}^*, l_{[\cdot,\cdot]}^*, -l_{[\cdot,\cdot]}^*-r_{[\cdot,\cdot]}^*)$ is a representation of $(A, \circ,[\cdot,\cdot])$. Moreover, the remaining statement follows from Proposition \ref{repofdppa}.
\end{proof}

\begin{ex}\label{repdualtoqurs1}
(1) The regular representation $(A; L_{\circ},R_{\circ},L_{[\cdot,\cdot]},R_{[\cdot,\cdot]})$ of a dual pre-Poisson algebra  $(A, \circ,[\cdot,\cdot])$ gives the representation $(A^*;-L_{\circ}^*, -L_{\circ}^*+R_{\circ}^*, L_{[\cdot,\cdot]}^*, -L_{[\cdot,\cdot]}^*-R_{[\cdot,\cdot]}^*)$,
  called the \textbf{coregular representation} of $(A, \circ,[\cdot,\cdot])$. In this case,   there is a dual  pre-Poisson algebra structrue $A \ltimes_{-L_{\circ}^*, -L_{\circ}^*+R_{\circ}^*, L_{[\cdot,\cdot]}^*, -L_{[\cdot,\cdot]}^*-R_{[\cdot,\cdot]}^*} A^*$ on the direct sum  $A \oplus A^*$  of vector spaces. 
  
  (2) Continuing with Example \ref{dwdpp} (2), let $\{e_1^*, e_2^* \}$ be the dual basis of $\{e_1, e_2 \}$. Then there is a 4-dimensional dual pre-Poisson algebra  $A \ltimes_{-L_{\circ}^*, -L_{\circ}^*+R_{\circ}^*, L_{[\cdot,\cdot]}^*, -L_{[\cdot,\cdot]}^*-R_{[\cdot,\cdot]}^*} A^*$ whose nonzero products are explicitly given by Eq.~\eqref{dwdpp:1} and the following equations 
\begin{align}
  e_2 \circ e_1^* = e_2^*,  \quad   [ e_2 ,e_1^*] = -e_2^*,  \quad [e_1^*, e_2] = 2e_2^*.
\end{align}
\end{ex}

 \begin{defi}
A bilinear form $\mathcal{B}: A \otimes A \to \mathbb{F}$ on a dual pre-Poisson algebra $(A, \circ,[\cdot,\cdot])$ is called {\bf invariant} if
\begin{align}
\mathcal{B}(x\circ y,z) &= \mathcal{B}(x, y\circ z - z \circ y), \label{perminv}\\
\mathcal{B}([x, y],z) &=  \mathcal{B}(x, [y, z] + [z , y]), \quad \forall x,y,z \in A. \label{Leibinv}
\end{align}
A {\bf quadratic dual pre-Poisson algebra} $(A, \circ,[\cdot,\cdot],\mathcal{B})$ is a dual pre-Poisson algebra $(A, \circ,[\cdot,\cdot])$ with a
nondegenerate skew-symmetric invariant bilinear form $\mathcal{B}$.
\end{defi}

\begin{rmk}
The Eq.~\eqref{perminv} first appeared as the invariance condition for bilinear forms on permutative algebras in \cite{BYZ}, while Eq.~\eqref{Leibinv} did so for Leibniz algebras in \cite{TS}. If $\mathcal{B}$ is a nondegenerate skew-symmetric invariant bilinear form on a perm algebra  $(A, \circ)$, then the triple $(A, \circ, \mathcal{B})$ is called a {\bf quadratic perm algebra}. Similarly, a  Leibniz algebra $(A, [\cdot,\cdot])$ equipped with a nondegenerate skew-symmetric invariant bilinear form $\mathcal{B}$ is called a {\bf quadratic Leibniz  algebra} and denoted by $(A, [\cdot,\cdot],\mathcal{B})$. Therefore,  a quadratic dual pre-Poisson algebra  structure consists of a quadratic perm algebra structure and a quadratic Leibniz  algebra structure.
\end{rmk}

There are two natural constructions of quadratic dual pre-Poisson algebras from Poisson algebras and dual pre-Poisson algebras  respectively.

\begin{pro} Let $A$ be a vector space and $\mathcal{B}_{A \oplus A^{*}}$ be the  bilinear form on the direct sum  $A \oplus A^*$  of vector spaces  defined by 
\begin{align}
\mathcal{B}_{A \oplus A^{*}}\left(x+a^{*}, y+b^{*}\right)=  \left\langle  y,a^{*}\right\rangle - \left\langle x, b^{*}\right\rangle, \quad  \forall x, y \in A, a^{*}, b^{*} \in A^{*}. \label{abm}
\end{align}
 \begin{enumerate}
\item\label{itemPoi:1} If  $(A, \bullet,\{\cdot,\cdot\})$ is a Poisson algebra, then $(A \rightthreetimes_{-L_{\bullet}^*,\ad_{\{\cdot,\cdot\}}^*} A^*, \mathcal{B}_{A \oplus A^{*}})$ is a quadratic dual pre-Poisson algebra. 
\item\label{itemPoi:2} If  $(A, \circ,[\cdot,\cdot])$ is a dual pre-Poisson algebra, then $(A \ltimes_{-L_{\circ}^*, -L_{\circ}^*+R_{\circ}^*, L_{[\cdot,\cdot]}^*, -L_{[\cdot,\cdot]}^*-R_{[\cdot,\cdot]}^*} A^*,\mathcal{B}_{A \oplus A^{*}})$ is a quadratic dual pre-Poisson algebra. 
 \end{enumerate}
\end{pro}
\begin{proof}
\eqref{itemPoi:1}. Let $x,y \in A$, $a^*,b^* \in A^*$, we have
\begin{align*}
\mathcal{B}_{A \oplus A^{*}}\left( (x+a^{*}) \circ  (y+b^{*}) , z+c^{*}\right) &= \mathcal{B}_{A \oplus A^{*}}\left(  x \bullet y -L_{\bullet}^*(x)b^* , z+c^{*}\right) = \left\langle  z,-L_{\bullet}^*(x)b^*\right\rangle - \left\langle x \bullet y , c^{*}\right\rangle\\
&=\left\langle  x,-L_{\bullet}^*(z)b^*\right\rangle + \left\langle x  , L_{\bullet}^*(y)c^{*}\right\rangle\\
&=\left\langle  x+a^*,y\bullet z +L_{\bullet}^*(z)b^*-z\bullet y-L_{\bullet}^*(y)c^{*}\right\rangle\\
&=\mathcal{B}_{A \oplus A^{*}}\left(  x+a^{*}  , (y+b^{*}) \circ ( z+c^{*}) -  ( z+c^{*})\circ (y+b^{*})   \right),\\
\mathcal{B}_{A \oplus A^{*}}\left( [ x+a^{*}  ,  y+b^{*}] , z+c^{*}\right) &=   \mathcal{B}_{A \oplus A^{*}}\left(  \{x ,y \}+\ad_{\{\cdot,\cdot\}}^*(x)b^* , z+c^{*}\right)\\
&=   \left\langle    z ,\ad_{\{\cdot,\cdot\}}^*(x)b^* \right\rangle - \left\langle \{x ,y \}  , c^{*}\right\rangle 
 =   \left\langle    x ,-\ad_{\{\cdot,\cdot\}}^*(z)b^*  -   \ad_{\{\cdot,\cdot\}}^*(y)c^{*}\right\rangle\\
 &=\left\langle    x+a^*,[y,z]  +  \ad_{\{\cdot,\cdot\}}^*(y)c^{*}+[z,y]+\ad_{\{\cdot,\cdot\}}^*(z)b^*\right\rangle\\
&= \mathcal{B}_{A \oplus A^{*}}\left(  x+a^{*}  ,  [y+b^{*} , z+c^{*}] +  [ z+c^{*}, y+b^{*}] \right).
\end{align*}
Thus $(A \rightthreetimes_{-L_{\bullet}^*,\ad_{\{\cdot,\cdot\}}^*} A^*, \mathcal{B}_{A \oplus A^{*}})$ is a quadratic dual pre-Poisson algebra. 

\eqref{itemPoi:2}. It is straightforward to verify that $\mathcal{B}_{A \oplus A^{*}}$ is invariant on  $A \ltimes_{-L_{\circ}^*, -L_{\circ}^*+R_{\circ}^*, L_{[\cdot,\cdot]}^*, -L_{[\cdot,\cdot]}^*-R_{[\cdot,\cdot]}^*} A^*$.
\end{proof}

\begin{ex}
(1) In Example \ref{bigdualpre} (2), there is a 6-dimensional quadratic dual pre-Poisson algebra  $(A \rightthreetimes_{-L_{\bullet}^*,\ad_{\{\cdot,\cdot\}}^*} A^*, \mathcal{B}_{A \oplus A^{*}})$ whose nonzero products are explicitly given by Eqs.~\eqref{exdual1}-\eqref{exdual2} and 
$\mathcal{B}_{A \oplus A^{*}}$ is given by 
\begin{align*}
	\mathcal{B}_{A \oplus A^{*}}(e_i, e_j) = \mathcal{B}_{A \oplus A^{*}}(e_i^*, e_j^*) = 0, \quad \mathcal{B}_{A \oplus A^{*}}(e_i^*, e_j) = -\mathcal{B}_{A \oplus A^{*}}(e_i, e_j^*) = \begin{cases}
		1, i=j\\
		0, i \neq j
	\end{cases}   i, j \in  \{1, 2,3\}.
\end{align*} 

(2)  In Example \ref{repdualtoqurs1} (2), there is an invariant bilinear form  $\mathcal{B}_{A \oplus A^{*}}$  on the  4-dimensional  dual pre-Poisson algebra $A \ltimes_{-L_{\circ}^*, -L_{\circ}^*+R_{\circ}^*, L_{[\cdot,\cdot]}^*, -L_{[\cdot,\cdot]}^*-R_{[\cdot,\cdot]}^*} A^*$ given by 
\begin{align*}
	\mathcal{B}_{A \oplus A^{*}}(e_i, e_j) = \mathcal{B}_{A \oplus A^{*}}(e_i^*, e_j^*) = 0, \quad \mathcal{B}_{A \oplus A^{*}}(e_i^*, e_j) = -\mathcal{B}_{A \oplus A^{*}}(e_i, e_j^*) = \begin{cases}
		1, i=j\\
		0, i \neq j
	\end{cases}   i, j \in  \{1, 2\}, 
\end{align*} 
such that $(A \ltimes_{-L_{\circ}^*, -L_{\circ}^*+R_{\circ}^*, L_{[\cdot,\cdot]}^*, -L_{[\cdot,\cdot]}^*-R_{[\cdot,\cdot]}^*} A^*,\mathcal{B}_{A \oplus A^{*}})$ is a quadratic dual pre-Poisson algebra.

\end{ex}

\begin{pro}
With the conditions in Proposition \ref{sdztodpp}. Suppose that there is a bilinear form $\mathcal{B}$ on $A$ such that $(A, \circ, \mathcal{B})$ is a  quadratic perm algebra and $D_1, D_2$ are
 skew-symmetric with respect   to  $\mathcal{B}$ in the sense that
\begin{align*}
\mathcal{B}(D_i(x),y) + \mathcal{B}(x,D_i(y)) =0, \quad \forall x,y \in A, i \in \{1,2\}.
\end{align*}
Then $(A, \circ,[\cdot,\cdot],\mathcal{B})$ is a quadratic dual pre-Poisson algebra with $[\cdot,\cdot]$ given by Eq.~\eqref{DDdualprep}.
\end{pro}
\begin{proof}
It is known that $(A, \circ,[\cdot,\cdot])$ is a  dual pre-Poisson algebra by Proposition \ref{sdztodpp}. Moreover, for all $x,y,z \in A$, we have
\begin{align*}
\mathcal{B}([x, y],z)  &= \mathcal{B}(D_1(x)\circ D_2(y) - D_2(x) \circ D_1(y) ,z)\\
&= \mathcal{B}(D_1(x), D_2(y) \circ z - z \circ D_2(y)) - \mathcal{B}(D_2(x)  ,D_1(y) \circ z - z \circ D_1(y))  \\
&=  \mathcal{B}(x  ,D_2 D_1(y)  \circ z + D_1(y)  \circ D_2(z)  - D_2(z) \circ D_1(y) - z  \circ D_2 D_1(y)   )   \\
&\quad-\mathcal{B}(x, D_1 D_2(y) \circ z + D_2(y) \circ D_1(z) - D_1(z) \circ D_2(y) -  z \circ D_1 D_2(y))\\
&=  \mathcal{B}(x  , D_1(y)  \circ D_2(z)  - D_2(z) \circ D_1(y)   - D_2(y) \circ D_1(z) + D_1(z) \circ D_2(y) )\\
&= \mathcal{B}(x, [y, z] + [z , y]).
\end{align*}
Thus $(A, \circ,[\cdot,\cdot],\mathcal{B})$ is a quadratic dual pre-Poisson algebra.
\end{proof}
 
Recall \cite{NB} that a {\bf quadratic Poisson algebra} is a quadruple $(A, \bullet,\{\cdot,\cdot\},\mathcal{B})$ where $(A, \bullet,\{\cdot,\cdot\})$ is a Poisson algebra and $\mathcal{B}$ is a nondegenerate symmetric  bilinear form satisfying the following conditions
 \begin{align*}
 \mathcal{B}(x \bullet y,z) =  \mathcal{B}(x ,y \bullet z),\;  \mathcal{B}(\{x,y\},z) =  \mathcal{B}(x,\{y,z\}),\quad \forall x,y,z \in A.
 \end{align*}
 
Next we extend the construction of dual pre-Poisson algebras in Proposition \ref{ifdualpp} to the quadratic case.
 
\begin{pro}\label{ifqdualpp}
 Let $(A,\star,\mathcal{B}_A)$ be a quadratic permutative algebra and $(B, \bullet,\{\cdot,\cdot\},\mathcal{B}_B)$ be a quadratic Poisson algebra. Then $(A \otimes B , \circ_{A \otimes B},[\cdot,\cdot]_{A \otimes B},\mathcal{B}_{A \otimes B})$ is a quadratic dual  pre-Poisson algebra where $\circ_{A \otimes B}$ and $[\cdot,\cdot]_{A \otimes B}$ are   given in Proposition \ref{ifdualpp} and $\mathcal{B}_{A \otimes B}$ is defined by 
 \begin{align}
 \mathcal{B}_{A \otimes B}( x \otimes a,  y \otimes b) &:=  \mathcal{B}_{A }(x ,y)\mathcal{B}_{B}(a,b), \quad \forall x,y \in A, a,b \in B.
\end{align} 
\end{pro}
\begin{proof}
By Proposition \ref{ifdualpp}, it is enough to prove that $\mathcal{B}_{A \otimes B}$ is a nondegenerate skew-symmetric invariant bilinear form on the dual pre-Poisson algebra $ (A \otimes B, \circ_{A \otimes B},[\cdot,\cdot]_{A \otimes B})$. For all $x,y,z \in A, a,b,c \in B$, we have
\begin{align*}
&\mathcal{B}_{A \otimes B}( x \otimes a,  y \otimes b) =   \mathcal{B}_{A }(x ,y)\mathcal{B}_{B}(a,b) =   -\mathcal{B}_{A }(y ,x)\mathcal{B}_{B}(b,a) = \mathcal{B}_{A \otimes B}( y \otimes b,  x \otimes a).
\end{align*}
Thus $\mathcal{B}_{A \otimes B}$ is skew-symmetric and it is straightforward to verify that $\mathcal{B}_{A \otimes B}$ is  nondegenerate. Moreover   
\begin{align*}
&\mathcal{B}_{A \otimes B}( (x \otimes a) \circ_{A \otimes B} (y \otimes b), z \otimes c)  =\mathcal{B}_{A \otimes B}( (x \star y) \otimes (a \bullet b), z \otimes c) \\
& =\mathcal{B}_{A}(  x \star y ,z)\mathcal{B}_{B}( a \bullet b ,   c)  = \mathcal{B}_{A}(x, y \star z - z \star y)\mathcal{B}_{B}( a, b \bullet  c)\\
&= \mathcal{B}_{A}(x, y \star z  )\mathcal{B}_{B}( a, b \bullet  c) - \mathcal{B}_{A}(x,     z \star y)\mathcal{B}_{B}( a, c \bullet  b)\\
&=\mathcal{B}_{A \otimes B}( x \otimes a , (y \otimes b)\circ_{A \otimes B} (z \otimes c)  
- (z \otimes c)\circ_{A \otimes B}(y \otimes b) ), \\
&\mathcal{B}_{A \otimes B}( [ x \otimes a ,  y \otimes b]_{A \otimes B}, z \otimes c)=\mathcal{B}_{A \otimes B}(  (x \star y)  \otimes  \{a ,b\}, z \otimes c) \\
&=\mathcal{B}_{A}(  x \star y, z ) \mathcal{B}_{B}(    \{a ,b\},  c) = \mathcal{B}_{A}(x, y \star z - z \star y)\mathcal{B}_{B}(   a , \{b,  c\})\\
 &= \mathcal{B}_{A}(x, y \star z  )\mathcal{B}_{B}(   a , \{b,  c\})+ \mathcal{B}_{A}(x,  z \star y)\mathcal{B}_{B}(   a , \{c,  b\})\\
 &= \mathcal{B}_{A \otimes B}( x \otimes a , [  y \otimes b, z \otimes c]_{A \otimes B} +  [   z \otimes c,y \otimes b]_{A \otimes B}).
\end{align*}
 That is, $\mathcal{B}_{A \otimes B}$ is invariant on  $(A \otimes B , \circ_{A \otimes B},[\cdot,\cdot]_{A \otimes B})$. 
\end{proof}

\begin{ex}
Continuing with Example \ref{infdualppa}, there is a  quadratic dual  pre-Poisson algebra $ (A \otimes B, \circ_{A \otimes B},[\cdot,\cdot]_{A \otimes B},\mathcal{B}_{A \otimes B})$ with the bilinear form $\mathcal{B}_{A \otimes B}$ given by 
\begin{align*}
\mathcal{B}_{A}\left(t^{i}  \partial_{1}, t^{j}   \partial_{2}\right) & =-\mathcal{B}_{A}\left(t^{j}  \partial_{2}, t^{i}   \partial_{1}\right)=\delta_{i+j, 0},\quad
\mathcal{B}_{A}\left(t^{i}   \partial_{1}, t^{j}  \partial_{1}\right)   =\mathcal{B}_{A}\left(t^{i}   \partial_{2}, t^{j}   \partial_{2}\right)=0,  \\   
\mathcal{B}_{B}(x^{a} y^{b}, x^{c} y^{d}) &= \delta_{a+c+1,0}\delta_{b+d+1,0}, \quad \forall a,b,c,d,i,j\in \mathbb{Z}.
\end{align*}
{}
\end{ex}

\begin{lem}\label{invxzyil}
Let $\mathcal{B}$ be a skew-symmetric invariant bilinear form on a dual pre-Poisson algebra $(A, \circ,[\cdot,\cdot])$. Then $\mathcal{B}$ satisfies the following properties
\begin{align}
\mathcal{B}(x\circ y,z) &= \mathcal{B}(y,x\circ z),\\
\mathcal{B}([x, y],z) &= -\mathcal{B}(y,[x, z]), \quad \forall x,y,z \in A.
\end{align}
\end{lem}
\begin{proof}
For all $x, y, z \in A$, we have
\begin{align*}
 \mathcal{B}(x \circ y,z) &= \mathcal{B}(x, y\circ z - z \circ y) = -\mathcal{B}(x, z\circ y - y \circ z) =- \mathcal{B}(x \circ z,y) =   \mathcal{B}(y,x \circ z) ,\\
  \mathcal{B}([x, y],z) &= \mathcal{B}(x, [y, z] + [z , y])= \mathcal{B}(x, [z, y] + [y , z])=  \mathcal{B}([x, z],y)=   -\mathcal{B}(y,[x, z]).
\end{align*}
This completes the proof.
\end{proof}

\begin{pro}\label{dualdj:x}
 Let $(A, \circ,[\cdot,\cdot])$  be a dual pre-Poisson algebra. If there is a nondegenerate skew-symmetric invariant bilinear form  $\mathcal{B}$  such that $(A, \circ,[\cdot,\cdot],\mathcal{B})$ is a quadratic dual pre-Poisson algebra, then the two representations $(A; L_{\circ},R_{\circ},L_{[\cdot,\cdot]},R_{[\cdot,\cdot]})$   and  $(A^*;-L_{\circ}^*, -L_{\circ}^*+R_{\circ}^*, L_{[\cdot,\cdot]}^*, -L_{[\cdot,\cdot]}^*-R_{[\cdot,\cdot]}^*)$  of the dual pre-Poisson algebra $(A, \circ,[\cdot,\cdot])$ are equivalent. 
 
 Conversely, if the two representations $(A; L_{\circ},R_{\circ},L_{[\cdot,\cdot]},R_{[\cdot,\cdot]})$   and  $(A^*;-L_{\circ}^*, -L_{\circ}^*+R_{\circ}^*, L_{[\cdot,\cdot]}^*,$ $ -L_{[\cdot,\cdot]}^*-R_{[\cdot,\cdot]}^*)$ of the dual pre-Poisson algebra $(A, \circ,[\cdot,\cdot])$ are equivalent, then there exists a nondegenerate invariant bilinear form  $\mathcal{B}$  on $A$.
\end{pro}
\begin{proof}
Since $\mathcal{B}$ is nondegenerate, there exists a linear isomorphism $\varphi: A \to A^*$  defined by
\begin{align}
\langle \varphi(x) ,y\rangle := \mathcal{B}(x,y),\quad \forall x,y \in A.\label{mapinbyB}
\end{align}
By Lemma \ref{invxzyil}, for all $x, y, z \in A$, we have
\begin{align*}
\langle \varphi(L_{\circ}(x)y) ,z\rangle &= \mathcal{B}(x \circ y,z) = \mathcal{B}(y,x\circ z) =   \langle -L^*_{\circ}(x)\varphi(y),  z\rangle ,\\
\langle \varphi(R_{\circ}(y)x) ,z\rangle &= \mathcal{B}(x \circ y,z) = \mathcal{B}(x, y\circ z - z \circ y) = \langle \varphi(x), (L_{\circ} - R_{\circ})(y)z\rangle ,\\
& = \langle (-L_{\circ}^* + R_{\circ}^*)(y)\varphi(x), z\rangle,\\
\langle \varphi(L_{[\cdot,\cdot]}(x)y) ,z\rangle &= \mathcal{B}([x , y],z) = -\mathcal{B}(y,[x, z]) =   \langle  L^*_{[\cdot,\cdot]}(x)\varphi(y),  z\rangle ,\\
\langle \varphi(R_{[\cdot,\cdot]}(y)x) ,z\rangle &= \mathcal{B}([x , y],z) = \mathcal{B}(x, [y, z] + [z , y]) =   \langle  -(L^*_{[\cdot,\cdot]}+R^*_{[\cdot,\cdot]})(y)\varphi(x),  z\rangle.
\end{align*}
Thus the two representations $(A; L_{\circ},R_{\circ},L_{[\cdot,\cdot]},R_{[\cdot,\cdot]})$   and  $(A^*;-L_{\circ}^*, -L_{\circ}^*+R_{\circ}^*, L_{[\cdot,\cdot]}^*, -L_{[\cdot,\cdot]}^*-R_{[\cdot,\cdot]}^*)$ of  the dual pre-Poisson algebra  $(A, \circ,[\cdot,\cdot])$  are equivalent. The converse statement can be proved by reversing the argument.
\end{proof}

 \section{A bialgebra theory for dual pre-Poisson algebras}\label{Bialgebra}

 In this section, we introduce the notions of Manin triples of dual pre-Poisson algebras and dual pre-Poisson bialgebras. The equivalence between them is interpreted in terms of certain matched pairs of
dual pre-Poisson algebras. By studying a special class of dual pre-Poisson bialgebras, we are led to introduce the permutative-Leibniz Yang-Baxter equation (PLYBE), which consists of the permutative and the classical Leibniz Yang-Baxter equations. Symmetric solutions of this equation can give rise to such bialgebras. Finally, to construct such solutions, we introduce the concepts of $\mathcal{O}$-operators and pre-dual pre-Poisson algebras. The latter of which serve as the operadic disuccessors of dual pre-Poisson algebras.

  \subsection{Matched pair and Manin triple  of dual pre-Poisson  algebras}
  
  We first give the concept of a matched pair of dual pre-Poisson  algebras and a Manin triple of dual pre-Poisson  algebras.

 \begin{defi}\label{mpofdppa}
 Let $(A_1,\circ_1,[\cdot,\cdot]_1)$ and $(A_2,\circ_2,[\cdot,\cdot]_2)$  be two dual pre-Poisson algebras. Let  $l_{\circ_{1}}, r_{\circ_{1}}, l_{[\cdot,\cdot]_{1}},$ $ r_{[\cdot,\cdot]_{1}}: A_{1} \rightarrow \operatorname{End}\left(A_{2}\right)$  and  $l_{\circ_{2}}, r_{\circ_{2}}, l_{[\cdot,\cdot]_{2}}, r_{[\cdot,\cdot]_{2}}: A_{2} \rightarrow \operatorname{End}\left(A_{1}\right)$  be eight linear maps such that 
 \begin{enumerate}
\item $(A_{1}, A_{2}, l_{\circ_{1}}, r_{\circ_{1}}, l_{\circ_{2}},$ $ r_{\circ_{2}})$ is a {\bf matched pair of permutative algebras}:  $(A_2; l_{\circ_{1}}, r_{\circ_{1}})$ and
$(A_1; l_{\circ_{2}}, $ $r_{\circ_{2}})$  are representations of permutative algebras $(A_1, \circ_1)$ and $(A_2, \circ_2)$ respectively satisfying
\begin{align}
l_{\circ_{1}}\left(x\right)\left(a \circ_2 b\right) & =\left(l_{\circ_{1}}\left(x\right) a\right) \circ_2 b+l_{\circ_{1}}\left(r_{\circ_{2}}\left(a\right) x\right)b, \\
r_{\circ_{1}}\left(x\right)\left(a \circ_2 b\right) & =a \circ_2\left(r_{\circ_{1}}\left(x\right) b\right)+r_{\circ_{1}}\left(l_{\circ_{2}}\left(b\right) x\right) a, \\
l_{\circ_{2}}\left(a\right)\left(x \circ_{1} y\right) & =\left(l_{\circ_{2}}\left(a\right) x\right) \circ_{1} y+l_{\circ_{2}}\left(r_{\circ_{1}}\left(x\right) a\right) y,\\
r_{\circ_{2}}\left(a\right)\left(x \circ_{1} y\right) & =x \circ_{1}\left(r_{\circ_{2}}\left(a\right) y\right)+r_{\circ_{2}}\left(l_{\circ_{1}}\left(y\right) a\right) x,\\
\left(r_{\circ_{1}}\left(x\right) a\right) \circ_2 b+l_{\circ_{1}}\left(l_{\circ_{2}}\left(a\right) x\right) b & =a \circ_2\left(l_{\circ_{1}}\left(x\right) b\right)+r_{\circ_{1}}\left(r_{\circ_{2}}\left(b\right) x\right) a,\\
\left(r_{\circ_{2}}\left(a\right) x\right) \circ_{1} y+l_{\circ_{2}}\left(l_{\circ_{1}}\left(x\right) a\right) y & =x \circ_{1}\left(l_{\circ_{2}}\left(a\right) y\right)+r_{\circ_{2}}\left(r_{\circ_{1}}\left(y\right) a\right) x,\\
\left(l_{\circ_{1}}\left(x\right) a\right) \circ_2 b+l_{\circ_{1}}\left(r_{\circ_{2}}\left(a\right) x\right) b & =\left(r_{\circ_{1}}\left(x\right) a\right) \circ_2 b+l_{\circ_{1}}\left(l_{\circ_{2}}\left(a\right) x\right) b,\\
\left(l_{\circ_{2}}\left(a\right) x\right) \circ_{1} y+l_{\circ_{2}}\left(r_{\circ_{1}}\left(x\right) a\right) y & =\left(r_{\circ_{2}}\left(a\right) x\right) \circ_{1} y+l_{\circ_{2}}\left(l_{\circ_{1}}\left(x\right) a\right) y,\\
r_{\circ_{1}}\left(x\right)\left(a \circ_2 b\right) & =r_{\circ_{1}}\left(x\right)\left(b \circ_2 a\right),\\
r_{\circ_{2}}\left(a\right)\left(x \circ_{1} y\right) & =r_{\circ_{2}}\left(a\right)\left(y \circ_{1} x\right), \quad \forall x, y \in  A_1, a, b \in  A_2.
\end{align}
\item $(A_{1}, A_{2}, l_{[\cdot,\cdot]_{1}}, r_{[\cdot,\cdot]_{1}}, l_{[\cdot,\cdot]_{2}}, r_{[\cdot,\cdot]_{2}})$ is a {\bf  matched pair of  Leibniz algebras}:  $(A_2; l_{[\cdot,\cdot]_{1}}, $ $r_{[\cdot,\cdot]_{1}})$ and
$(A_1; l_{[\cdot,\cdot]_{2}}, r_{[\cdot,\cdot]_{2}})$  are representations of  Leibniz algebras $(A_1, [\cdot,\cdot]_1)$ and $(A_2, [\cdot,\cdot]_2)$ respectively satisfying
\begin{align}
&r_{[\cdot,\cdot]_{1}}(x)[a, b]_{2}\!-\!\left[a, r_{[\cdot,\cdot]_{1}}(x) b\right]_{2}\!+\!\left[b, r_{[\cdot,\cdot]_{1}}(x) a\right]_{2}\!-\!r_{[\cdot,\cdot]_{1}}\!\left(l_{[\cdot,\cdot]_{2}}(b) x\right) a  \!+\!r_{[\cdot,\cdot]_{1}}\!\left(l_{[\cdot,\cdot]_{2}}(a) x\right) b\!=\!0,\\
&l_{[\cdot,\cdot]_{1}}(x)[a, b]_{2}\!-\!\left[l_{[\cdot,\cdot]_{1}}(x) a, b\right]_{2}\!-\!\left[a, l_{[\cdot,\cdot]_{1}}(x) b\right]_{2}\!-\!l_{[\cdot,\cdot]_{1}}\left(r_{[\cdot,\cdot]_{2}}(a) x\right) b  \!-\!r_{[\cdot,\cdot]_{1}}\left(r_{[\cdot,\cdot]_{2}}(b) x\right) a\!=\!0,\\
&{\left[l_{[\cdot,\cdot]_{1}}(x) a, b\right]_{2}\!+\!l_{[\cdot,\cdot]_{1}}\left(r_{[\cdot,\cdot]_{2}}(a) x\right) b\!+\!\left[r_{[\cdot,\cdot]_{1}}(x) a, b\right]_{2}\!+\!l_{[\cdot,\cdot]_{1}}\left(l_{[\cdot,\cdot]_{2}}(a) x\right) b\!=\!0,} \\
&r_{[\cdot,\cdot]_{2}}(a)[x, y]_{1}\!\!-\!\!\left[x, r_{[\cdot,\cdot]_{2}}(a) y\right]_{1}\!+\!\!\left[y, r_{[\cdot,\cdot]_{2}}(a) x\right]_{1}\!\!-\! r_{[\cdot,\cdot]_{2}}\left(l_{[\cdot,\cdot]_{1}}(y) a\right) x \!+\!r_{[\cdot,\cdot]_{2}}\!\left(l_{[\cdot,\cdot]_{1}}(x) a\right) y\!=\!0, \\
&l_{[\cdot,\cdot]_{2}}(a)[x, y]_{1}\!-\!\!\left[l_{[\cdot,\cdot]_{2}}(a) x, y\right]_{1}\!-\!\left[x, l_{[\cdot,\cdot]_{2}}(a) y\right]_{1}\!-\!l_{[\cdot,\cdot]_{2}}\!\left(r_{[\cdot,\cdot]_{1}}(x) a\right) y  \!-\!r_{[\cdot,\cdot]_{2}}\!\left(r_{[\cdot,\cdot]_{1}}(y) a\right) x\!=\!0, \\
&{\left[l_{[\cdot,\cdot]_{2}}(a) x, y\right]_{1}\!+\!l_{[\cdot,\cdot]_{2}}\left(r_{[\cdot,\cdot]_{1}}(x) a\right) y\!+\!\left[r_{[\cdot,\cdot]_{2}}(a) x, y\right]_{1}\!+\!l_{[\cdot,\cdot]_{2}}\left(l_{[\cdot,\cdot]_{1}}(x) a\right) y\!=\!0}, \quad \forall x, y \in  A_1, a, b \in  A_2.
\end{align}
\item  $(A_{2};l_{\circ_{1}}, r_{\circ_{1}}, l_{[\cdot,\cdot]_{1}}, r_{[\cdot,\cdot]_{1}})$ and $(A_{1};l_{\circ_{2}}, r_{\circ_{2}}, l_{[\cdot,\cdot]_{2}}, r_{[\cdot,\cdot]_{2}})$ are representations of dual pre-Poisson algebras  $ (A_{1},  \circ_{1},[\cdot,\cdot]_{1})$ and  $ (A_{2},  \circ_{2},[\cdot,\cdot]_{2})$ respectively satisfying the following equations
\begin{align}
&[x,l_{\circ_{2}}(a)y]_{1}  + r_{[\cdot,\cdot]_{2}}(r_{\circ_{1}}(y)a)x = r_{[\cdot,\cdot]_{2}}(a) x \circ_{1} y + l_{\circ_{2}}(a)([ x , y]_{1}) + l_{\circ_{2}}(l_{[\cdot,\cdot]_{1}}(x)a)y, \label{mpdualppc1}\\
&[a,l_{\circ_{1}}(x)b]_{2}  + r_{[\cdot,\cdot]_{1}}(r_{\circ_{2}}(b)x)a = r_{[\cdot,\cdot]_{1}}(x) a \circ_{2} b + l_{\circ_{1}}(x)([ a , b]_{2}) + l_{\circ_{1}}(l_{[\cdot,\cdot]_{2}}(a)x)b,\label{mpdualppc2}\\
&[x,r_{\circ_{2}}(a)y]_{1}  + r_{[\cdot,\cdot]_{2}}(l_{\circ_{1}}(y)a)x = y \circ_{1} r_{[\cdot,\cdot]_{2}}(a) x   + r_{\circ_{2}}(a)([ x , y]_{1}) + r_{\circ_{2}}(l_{[\cdot,\cdot]_{1}}(x)a)y,\label{mpdualppc3}\\
&[a,r_{\circ_{1}}(x)b]_{2}  + r_{[\cdot,\cdot]_{1}}(l_{\circ_{2}}(b)x)a = b \circ_{2} r_{[\cdot,\cdot]_{1}}(x) a   + r_{\circ_{1}}(x)([ a , b]_{2}) + r_{\circ_{1}}(l_{[\cdot,\cdot]_{2}}(a)x)b,\label{mpdualppc4}\\
&l_{[\cdot,\cdot]_{2}}(a) (x \circ_{1} y) = l_{\circ_{2}}(r_{[\cdot,\cdot]_{1}}(x)a)y + r_{\circ_{2}}(r_{[\cdot,\cdot]_{1}}(y)a)x +  l_{[\cdot,\cdot]_{2}}(a)x  \circ_{1} y + x \circ_{1}  l_{[\cdot,\cdot]_{2}}(a) y,\label{mpdualppc5}\\
&l_{[\cdot,\cdot]_{1}}(x) (a \circ_{2} b) = l_{\circ_{1}}(r_{[\cdot,\cdot]_{2}}(a)x)b + r_{\circ_{1}}(r_{[\cdot,\cdot]_{2}}(b)x)a +  l_{[\cdot,\cdot]_{1}}(x)a  \circ_{2} b + a \circ_{2}  l_{[\cdot,\cdot]_{1}}(x) b , \label{mpdualppc6}\\
&[r_{\circ_{2}}(a)x,y]_{1} + l_{[\cdot,\cdot]_{2}}(l_{\circ_{1}}(x)a)y = l_{\circ_{2}}(a)([x,y]) + x \circ_{1} l_{[\cdot,\cdot]_{2}}(a) y + r_{\circ_{2}}(r_{[\cdot,\cdot]_{1}}(y)a)x,\label{mpdualppc7} \\
&[r_{\circ_{2}}(a)x,y]_{1} + l_{[\cdot,\cdot]_{2}}(l_{\circ_{1}}(x)a)y =  [l_{\circ_{2}}(a)x,y]_{1} + l_{[\cdot,\cdot]_{2}}(r_{\circ_{1}}(x)a)y,\label{mpdualppc8}\\
&r_{[\cdot,\cdot]_{2}}(a)(x \circ_{1} y) = x \circ_{1}  r_{[\cdot,\cdot]_{2}}(a) y + r_{\circ_{2}}(l_{[\cdot,\cdot]_{1}}(y)a)x + y \circ_{1}  r_{[\cdot,\cdot]_{2}}(a) x + r_{\circ_{2}}(l_{[\cdot,\cdot]_{1}}(x)a)y, \label{mpdualppc9}\\
&[l_{\circ_{1}}(x)a,b]_{2}+  l_{[\cdot,\cdot]_{1}}(r_{\circ_{2}}(a)x)b  =     l_{\circ_{1}}(x)([a,b]_{2})+ a \circ_{2}  l_{[\cdot,\cdot]_{1}}(x) b + r_{\circ_{1}}(r_{[\cdot,\cdot]_{2}}(b)x)a , \label{mpdualppc10}\\
&[l_{\circ_{1}}(x)a,b]_{2}+  l_{[\cdot,\cdot]_{1}}(r_{\circ_{2}}(a)x)b  =      [r_{\circ_{1}}(x)a,b]_{2}+ l_{[\cdot,\cdot]_{1}}(l_{\circ_{2}}(a)x)b, \label{mpdualppc11}\\
&r_{[\cdot,\cdot]_{1}}(x)(a \circ_{2} b) = a \circ_{2}  r_{[\cdot,\cdot]_{1}}(x) b + r_{\circ_{1}}(l_{[\cdot,\cdot]_{2}}(b)x)a + b \circ_{2}  r_{[\cdot,\cdot]_{1}}(x) a + r_{\circ_{1}}(l_{[\cdot,\cdot]_{2}}(a)x)b, \label{mpdualppc12} \\
&  l_{\circ_{2}}(l_{[\cdot,\cdot]_{1}}(x) a) y + r_{[\cdot,\cdot]_{2}}(a) x \circ_{1} y  = - l_{[\cdot,\cdot]_{2}}(a) x \circ_{1} y - l_{\circ_{2}}(r_{[\cdot,\cdot]_{1}}(x) a) y, \label{mpdualppc13} \\
&l_{\circ_{1}}(l_{[\cdot,\cdot]_{2}}(a)x)b + r_{[\cdot,\cdot]_{1}}(x) a\circ_{2}b = - l_{[\cdot,\cdot]_{1}}(x) a\circ_{2}b-l_{\circ_{1}}(r_{[\cdot,\cdot]_{2}}(a)x)b ,\label{mpdualppc14} \\
&r_{\circ_{2}}(a)([x,y]_{1}) = -r_{\circ_{2}}(a)([y,x]_{1}),\label{mpdualppc15} \\
&r_{\circ_{1}}(x)([a,b]_{2}) = -r_{\circ_{1}}(x)([b,a]_{2}), \quad \forall x, y \in  A_1, a, b \in  A_2. \label{mpdualppc16}
\end{align}
 \end{enumerate}
 Then we call $(A_{1},A_{2},l_{\circ_{1}}, r_{\circ_{1}}, l_{[\cdot,\cdot]_{1}}, r_{[\cdot,\cdot]_{1}},l_{\circ_{2}}, r_{\circ_{2}}, l_{[\cdot,\cdot]_{2}}, r_{[\cdot,\cdot]_{2}})$  a {\bf matched pair of dual pre-Poisson  algebras}.
 \end{defi}
 
 \begin{pro}\label{djdandm}
  Let $(A_1,\circ_1,[\cdot,\cdot]_1)$ and $(A_2,\circ_2,[\cdot,\cdot]_2)$  be two dual pre-Poisson algebras. Let  $l_{\circ_{1}}, r_{\circ_{1}}, l_{[\cdot,\cdot]_{1}}, r_{[\cdot,\cdot]_{1}}: A_{1} \rightarrow \operatorname{End}\left(A_{2}\right)$  and  $l_{\circ_{2}}, r_{\circ_{2}}, l_{[\cdot,\cdot]_{2}}, r_{[\cdot,\cdot]_{2}}: A_{2} \rightarrow \operatorname{End}\left(A_{1}\right)$  be eight linear maps. Define two bilinear operations  $\circ_{A\oplus B},[\cdot,\cdot]_{A\oplus B}:\left(A_{1} \oplus A_{2}\right) \otimes\left(A_{1} \oplus A_{2}\right) \rightarrow A_{1} \oplus A_{2}$  on  $A_{1} \oplus A_{2}$  by
\begin{align}
(x+a) \circ_{A\oplus B} (y+b) &=x \circ_{1} y+r_{\circ_{2}}(b) x+l_{\circ_{2}}(a) y+l_{\circ_{1}}(x) b+r_{\circ_{1}}(y) a+a \circ_{2} b,\\
 [x+a,y+b]_{A\oplus B}  &=[x ,y]_{1}  + r_{[\cdot,\cdot]_{2}}(b) x+l_{[\cdot,\cdot]_{2}}(a) y+l_{[\cdot,\cdot]_{1}}(x) b+r_{[\cdot,\cdot]_{1}}(y) a+[a , b]_{2},
\end{align}
where  $x, y \in A_{1}, a, b \in A_{2}$. Then $(A_{1} \oplus A_{2}, \circ_{A\oplus B},[\cdot,\cdot]_{A\oplus B})$ is a dual pre-Poisson algebra  if and only if $(A_{1},A_{2},l_{\circ_{1}}, r_{\circ_{1}}, l_{[\cdot,\cdot]_{1}}, r_{[\cdot,\cdot]_{1}},l_{\circ_{2}}, r_{\circ_{2}}, l_{[\cdot,\cdot]_{2}}, r_{[\cdot,\cdot]_{2}})$ is a matched pair of dual pre-Poisson algebras. We denote this dual pre-Poisson algebra by  $A_{1} \bowtie_{l_{\circ_{1}}, r_{\circ_{1}}, l_{[\cdot,\cdot]_{1}}, r_{[\cdot,\cdot]_{1}}}^{l_{\circ_{2}}, r_{\circ_{2}}, l_{[\cdot,\cdot]_{2}}, r_{[\cdot,\cdot]_{2}}} A_{2}$  or simply  $A_{1} \bowtie A_{2}$.
 \end{pro}
\begin{proof}
It can be checked directly by Definitions \ref{def1} and \ref{mpofdppa}.
\end{proof}

\begin{defi}  A {\bf (standard) Manin triple of  dual pre-Poisson algebras } is a triple  $((A \oplus A^{*}, \circ_{A \oplus A^{*}},$ $  [\cdot,\cdot]_{A \oplus A^{*}}, \mathcal{B}_{A \oplus A^{*}}), (A, \circ_A,[\cdot,\cdot]_A), (A^{*},\circ_{A^*},[\cdot,\cdot]_{A^*}))$  for which
\begin{enumerate}
\item  as a vector space,  $A \oplus A^{*}$  is the direct sum of  $A$  and  $A^{*}$.
\item  $(A, \circ_A,[\cdot,\cdot]_A)$ and $(A^{*},\circ_{A^*},[\cdot,\cdot]_{A^*})$  are dual pre-Poisson  subalgebras of dual pre-Poisson algebra  $((A \oplus A^{*}, \circ_{A \oplus A^{*}}, [\cdot,\cdot]_{A \oplus A^{*}} )$.
\item  $(A \oplus A^{*}, \circ_{A \oplus A^{*}}, [\cdot,\cdot]_{A \oplus A^{*}}, \mathcal{B}_{A \oplus A^{*}})$ is a quadratic dual pre-Poisson algebra  where the natural nondegenerate skew-symmetric bilinear form  $\mathcal{B}_{A \oplus A^{*}}$ is defined by Eq.~\eqref{abm}.
\end{enumerate}
We  denote it by $(A \oplus A^{*}, A, A^{*})$ simply.
\end{defi}

\begin{ex}
Let  $(A, \circ_A,[\cdot,\cdot]_A)$  be a dual pre-Poisson algebra and $\mathcal{B}_{A \oplus A^{*}}$ be the bilinear form on $A \oplus A^{*}$  defined by Eq.~\eqref{abm}. Then   $((A \ltimes_{-L_{\circ}^*, -L_{\circ}^*+R_{\circ}^*, L_{[\cdot,\cdot]}^*, -L_{[\cdot,\cdot]}^*-R_{[\cdot,\cdot]}^*} A^*, \mathcal{B}_{A \oplus A^{*}}),$ $ (A, \circ_A,[\cdot,\cdot]_A), $ $(A^{*},\circ_{A^*},[\cdot,\cdot]_{A^*}))$ is a  (standard) Manin triple of  dual pre-Poisson algebras where $(A^{*},\circ_{A^*},[\cdot,\cdot]_{A^*})$ is a trivial dual pre-Poisson algebra.
\end{ex}

\begin{pro}\label{mpandmtofdppa}
Let  $(A, \circ_A,[\cdot,\cdot]_A)$  be a dual pre-Poisson algebra. Suppose $(A^{*}, \circ_{A^{*}},[\cdot,\cdot]_{A^{*}})$ is a dual pre-Poisson algebra structure on the dual space  $A^{*}$. Then  $((A \oplus A^{*}, \circ_{A \oplus A^{*}}, [\cdot,\cdot]_{A \oplus A^{*}},$ $ \mathcal{B}_{A \oplus A^{*}}), (A, \circ_A,[\cdot,\cdot]_A), $ $(A^{*},\circ_{A^*},[\cdot,\cdot]_{A^*}))$  is a (standard) Manin triple of dual pre-Poisson algebras with $\mathcal{B}_{A \oplus A^{*}}$ given by Eq.~\eqref{abm} if and only if $(A,A^*,-L_{\circ}^*, -L_{\circ}^*+R_{\circ}^*, L_{[\cdot,\cdot]}^*, -L_{[\cdot,\cdot]}^*-R_{[\cdot,\cdot]}^*,-\mathcal{L}_{\circ}^*, -\mathcal{L}_{\circ}^*+\mathcal{R}_{\circ}^*,  \mathcal{L}_{[\cdot,\cdot]}^*, -\mathcal{L}_{[\cdot,\cdot]}^*-\mathcal{R}_{[\cdot,\cdot]}^*)$  is a matched pair of dual pre-Poisson algebras.
\end{pro}
\begin{proof}
 If  $((A \oplus A^{*}, \circ_{A \oplus A^{*}}, [\cdot,\cdot]_{A \oplus A^{*}},$ $ \mathcal{B}_{A \oplus A^{*}}), (A, \circ_A,[\cdot,\cdot]_A), (A^{*},\circ_{A^*},[\cdot,\cdot]_{A^*}))$  is a standard Manin triple of dual pre-Poisson algebras, then we set 
\begin{align*}
x \circ a^{*}&=l_{\circ_{A}}(x) a^{*}+r_{\circ_{A^{*}}}(a^{*}) x, && a^{*}\circ x=l_{\circ_{A^{*}}}(a^{*}) x+r_{\circ_A}(x) a^{*}, \\
[x , a^{*}]&=l_{[\cdot,\cdot]_{A}}(x) a^{*}+r_{[\cdot,\cdot]_{A^{*}}}(a^{*}) x, && [a^{*},x]=l_{[\cdot,\cdot]_{A^{*}}}(a^{*}) x+r_{[\cdot,\cdot]_A}(x) a^{*},\; \forall x \in A, a^{*} \in A^{*}.
\end{align*}
By Proposition \ref{djdandm},  $(A, A^{*}, l_{\circ_{A}}, r_{\circ_{A}},l_{[\cdot,\cdot]_{A}}, r_{[\cdot,\cdot]_{A}}, l_{\circ_{A^*}}, r_{\circ_{A^*}},l_{[\cdot,\cdot]_{A^*}}, r_{[\cdot,\cdot]_{A^*}})$  is a matched pair of dual pre-Poisson algebras. For all $x, y \in A$, $a^{*}, b^{*} \in A^{*}$, note that
\begin{align*}
\left\langle  y,l_{\circ_{A}}(x) a^{*}\right\rangle &= \mathcal{B}_{A \oplus A^{*}}( l_{\circ_{A}}(x) a^{*}, y) =\mathcal{B}_{A \oplus A^{*}}( x \circ a^{*} - r_{\circ_{A^{*}}}(a^{*}) x, y)=\mathcal{B}_{A \oplus A^{*}}( x \circ a^{*}, y) \\
&= \mathcal{B}_{A \oplus A^{*}}(  a^{*},   x  \circ_{A} y) =\left\langle    y,  -L_{\circ}^{*}(x)(a^{*}) \right\rangle.\\
\left\langle  y,r_{\circ_{A}}(x) a^{*}\right\rangle &= \mathcal{B}_{A \oplus A^{*}}( r_{\circ_{A}}(x) a^{*}, y) =\mathcal{B}_{A \oplus A^{*}}(a^{*} \circ x - l_{\circ_{A^{*}}}(a^{*}) x, y)=\mathcal{B}_{A \oplus A^{*}}( a^{*} \circ x , y) \\
&= \mathcal{B}_{A \oplus A^{*}}( a^{*}  , x \circ_{A}  y - y \circ_{A} x)   =  \left\langle  y, (-L_{\circ}^*+R_{\circ}^*)(x)a^{*}\right\rangle.\\
\left\langle  y,l_{[\cdot,\cdot]_{A}}(x) a^{*}\right\rangle &= \mathcal{B}_{A \oplus A^{*}}( l_{[\cdot,\cdot]_{A}}(x) a^{*}, y) =\mathcal{B}_{A \oplus A^{*}}( [x , a^{*}] - r_{[\cdot,\cdot]_{A^{*}}}(a^{*}) x, y)=\mathcal{B}_{A \oplus A^{*}}( [x,a^{*}], y) \\
&= -\mathcal{B}_{A \oplus A^{*}}(  a^{*},   [x,y]_{A}) =\left\langle    y,   L_{[\cdot,\cdot]}^{*}(x)(a^{*}) \right\rangle.\\
\left\langle  y,r_{[\cdot,\cdot]_{A}}(x) a^{*}\right\rangle &= \mathcal{B}_{A \oplus A^{*}}( r_{[\cdot,\cdot]_{A}}(x) a^{*}, y) =\mathcal{B}_{A \oplus A^{*}}([a^{*},x] - l_{[\cdot,\cdot]_{A^{*}}}(a^{*}) x, y)=\mathcal{B}_{A \oplus A^{*}}( [a^{*},x] , y) \\
&= \mathcal{B}_{A \oplus A^{*}}( a^{*}  , [x  ,y]_{A} + [y,x]_{A})   =  \left\langle  y,- (L_{[\cdot,\cdot]}^*+R_{[\cdot,\cdot]}^*)(x)a^{*}\right\rangle.
\end{align*}
 Hence  $l_{\circ_{A}}=-L_{\circ}^{*},\; r_{\circ_{A}}=-L_{\circ}^*+R_{\circ}^*,\; l_{[\cdot,\cdot]_{A}}=L_{[\cdot,\cdot]}^{*},\; r_{[\cdot,\cdot]_{A}}=- L_{[\cdot,\cdot]}^*-R_{[\cdot,\cdot]}^*$. Similarly, we have
 $l_{\circ_{A^{*}}}=-\mathcal{L}_{\circ}^{*},\; r_{\circ_{A^{*}}}=-\mathcal{L}_{\circ}^*+\mathcal{R}_{\circ}^*,\; l_{[\cdot,\cdot]_{A^{*}}}=\mathcal{L}_{[\cdot,\cdot]}^{*},\; r_{[\cdot,\cdot]_{A^{*}}}=- \mathcal{L}_{[\cdot,\cdot]}^*-\mathcal{R}_{[\cdot,\cdot]}^*$.

Conversely, if  $(A,A^*,-L_{\circ}^*, -L_{\circ}^*+R_{\circ}^*, L_{[\cdot,\cdot]}^*, -L_{[\cdot,\cdot]}^*-R_{[\cdot,\cdot]}^*,-\mathcal{L}_{\circ}^*, -\mathcal{L}_{\circ}^*+\mathcal{R}_{\circ}^*,  \mathcal{L}_{[\cdot,\cdot]}^*, -\mathcal{L}_{[\cdot,\cdot]}^*-\mathcal{R}_{[\cdot,\cdot]}^*)$   is a matched pair of dual pre-Poisson algebras, then there is a dual pre-Poisson algebra  $A \bowtie  A^{*}$  obtained from Proposition \ref{djdandm}, which includes  $(A, \circ_A,[\cdot,\cdot]_A)$ and $(A, \circ_{A^{*}},[\cdot,\cdot]_{A^{*}})$ as  dual pre-Poisson subalgebra. It is enough to show that the skew-symmetric bilinear  form  $\mathcal{B}_{A \oplus A^{*}}$ defined by Eq.~\eqref{abm}   is invariant. In fact, for all $x,y,z\in A, a^{*},b^{*},c^{*}\in A^{*}$, we have 
\begin{align*}
&\mathcal{B}_{A \oplus A^{*}}((x + a^*)\circ_{A \oplus A^*}( y + b^*), z + c^*)\\
&= \left\langle z , a^* \circ_{A^{*}} b^*   -L_{\circ}^*(x)b^*-(  L_{\circ}^*-R_{\circ}^*)(y)a^*  \right\rangle - \left\langle x \circ_{A} y   -\mathcal{L}_{\circ}^*(a^{*}) y-( \mathcal{L}_{\circ}^*-\mathcal{R}_{\circ}^*)(b^{*})x, c^* \right\rangle \\
&=  \left\langle z, a^* \circ_{A^{*}} b^*   \right\rangle  -\left\langle z, L_{\circ}^*(x)b^*\right\rangle -\left\langle z,(  L_{\circ}^*-R_{\circ}^*)(y)a^*\right\rangle -\left\langle x \circ_{A} y,c^* \right\rangle +\left\langle  \mathcal{L}_{\circ}^*(a^{*}) y, c^* \right\rangle \\
&\quad + \left\langle ( \mathcal{L}_{\circ}^*-\mathcal{R}_{\circ}^*)(b^*)x, c^* \right\rangle,\\
&\mathcal{B}_{A \oplus A^{*}}( x + a^*,   (y + b^*) \circ_{A \oplus A^*}(z + c^*)  - (z + c^*)\circ_{A \oplus A^*} (y + b^*) )\\
&= \mathcal{B}_{A \oplus A^{*}}( x + a^*,y \circ_{A} z \!-\! L_{\circ}^*(y) c^*\!-\!(  L_{\circ}^*\!-\!R_{\circ}^*)(z) b^{*}+ b^* \circ_{A^{*}} c^* \!-\! \mathcal{L}_{\circ}^*(b^*)z\!-\!( \mathcal{L}_{\circ}^*\!-\!\mathcal{R}_{\circ}^*)(c^*)y ) \\
&\quad \!-\!\mathcal{B}_{A \oplus A^{*}}( x + a^*,z \circ_{A} y \!-\! L_{\circ}^*(z) b^*\!-\!(  L_{\circ}^*\!-\!R_{\circ}^*)(y) c^{*}+ c^* \circ_{A^{*}} b^* \!-\! \mathcal{L}_{\circ}^*(c^*)y\!-\!( \mathcal{L}_{\circ}^*\!-\!\mathcal{R}_{\circ}^*)(b^*)z ) \\
&= \left\langle  a^*,y \circ_{A} z   - \mathcal{L}_{\circ}^*(b^*)z-( \mathcal{L}_{\circ}^*-\mathcal{R}_{\circ}^*)(c^*)y   \right\rangle   +\left\langle  x,  L_{\circ}^*(y) c^*+(  L_{\circ}^*-R_{\circ}^*)(z) b^{*}- b^* \circ_{A^{*}} c^* \right\rangle \\
&\quad -\left\langle    a^*,z \circ_{A} y   - \mathcal{L}_{\circ}^*(c^*)y-( \mathcal{L}_{\circ}^*-\mathcal{R}_{\circ}^*)(b^*)z  \right\rangle  +\left\langle   x  ,  - L_{\circ}^*(z) b^*-(  L_{\circ}^*-R_{\circ}^*)(y) c^{*}+ c^* \circ_{A^{*}} b^*  ) \right\rangle \\
&= \left\langle  a^*,y \circ_{A} z  + \mathcal{R}_{\circ}^* (c^*)y   \right\rangle   -\left\langle  x,      R_{\circ}^* (z) b^{*}+ b^* \circ_{A^{*}} c^* \right\rangle  -\left\langle    a^*, z \circ_{A} y   +\mathcal{R}_{\circ}^*(b^*)z  \right\rangle \\
& \quad +\left\langle   x  ,    R_{\circ}^* (y) c^{*}+ c^* \circ_{A^{*}} b^*  ) \right\rangle. 
\end{align*}
  Thus $\mathcal{B}_{A \oplus A^{*}}$ satisfies Eq.~\eqref{perminv}. Moreover,  we have
 \begin{align*}
&\mathcal{B}_{A \oplus A^{*}}([x + a^*, y + b^*]_{A \oplus A^*}, z + c^*)\\
&= \left\langle z ,  [a^*,b^*]_{A^{*}}    \!+\!L_{[\cdot,\cdot]}^*(x)b^*\!-\!(  L_{[\cdot,\cdot]}^*\!+\!R_{[\cdot,\cdot]}^*)(y)a^*  \right\rangle \!-\! \left\langle  [x,y]_{A}    \!+\!\mathcal{L}_{[\cdot,\cdot]}^*(a^{*}) y\!-\!( \mathcal{L}_{[\cdot,\cdot]}^*\!+\!\mathcal{R}_{[\cdot,\cdot]}^*)(b^{*})x, c^* \right\rangle \\
&=  \left\langle z,  [a^*,b^*]_{A^{*}}    \right\rangle  +\left\langle z, L_{[\cdot,\cdot]}^*(x)b^*\right\rangle -\left\langle z,(  L_{[\cdot,\cdot]}^*+R_{[\cdot,\cdot]}^*)(y)a^*\right\rangle -\left\langle  [x,y]_{A}  ,c^* \right\rangle -\left\langle  \mathcal{L}_{[\cdot,\cdot]}^*(a^{*}) y, c^* \right\rangle \\
&\quad + \left\langle ( \mathcal{L}_{[\cdot,\cdot]}^*+\mathcal{R}_{[\cdot,\cdot]}^*)(b^*)x, c^* \right\rangle,\\
&\mathcal{B}_{A \oplus A^{*}}( x + a^*,   [y + b^*   ,  z + c^*]_{A \oplus A^*}+ [z + c^* ,   y + b^*]_{A \oplus A^*} )\\
&= \mathcal{B}_{A \oplus A^{*}}( x \!+\! a^*, [y,z]_{A}   \!\!+\!\! L_{[\cdot,\cdot]}^*(y) c^*\!\!-\!\!(  L_{[\cdot,\cdot]}^*\!\!+\!\!R_{[\cdot,\cdot]}^*)(z) b^{*}\!+\!  [b^*,c^*]_{A^{*}}  \!\!+\!\! \mathcal{L}_{[\cdot,\cdot]}^*(b^*)z\!\!-\!\!( \mathcal{L}_{[\cdot,\cdot]}^*\!\!+\!\!\mathcal{R}_{[\cdot,\cdot]}^*)(c^*)y ) \\
&\quad \!\!+\!\!\mathcal{B}_{A \oplus A^{*}}( x \!+\! a^*, [z,y]_{A}  \!\!+\!\! L_{[\cdot,\cdot]}^*(z) b^*\!\!-\!\!(  L_{[\cdot,\cdot]}^*\!\!+\!\!R_{[\cdot,\cdot]}^*)(y) c^{*}\!+\! [c^* ,b^*]_{A^{*}}  \!\!+\!\! \mathcal{L}_{[\cdot,\cdot]}^*(c^*)y\!\!-\!\!( \mathcal{L}_{[\cdot,\cdot]}^*\!\!+\!\!\mathcal{R}_{[\cdot,\cdot]}^*)(b^*)z ) \\
&= \left\langle  a^*, [y,z]_{A}    \!+\! \mathcal{L}_{[\cdot,\cdot]}^*(b^*)z\!-\!( \mathcal{L}_{[\cdot,\cdot]}^*\!+\!\mathcal{R}_{[\cdot,\cdot]}^*)(c^*)y   \right\rangle   \!-\!\left\langle  x,  L_{[\cdot,\cdot]}^*(y) c^*\!-\!(  L_{[\cdot,\cdot]}^*\!+\!R_{[\cdot,\cdot]}^*)(z) b^{*}\!+\! [ b^*,c^*]_{A^{*}}  \right\rangle \\
&\quad \!+\!\left\langle    a^*, [z,y]_{A}    \!+\!\mathcal{L}_{[\cdot,\cdot]}^*(c^*)y\!-\!( \mathcal{L}_{[\cdot,\cdot]}^*\!+\!\mathcal{R}_{[\cdot,\cdot]}^*)(b^*)z  \right\rangle  \!-\!\left\langle   x  ,   L_{[\cdot,\cdot]}^*(z) b^*\!-\!(  L_{[\cdot ,\cdot]}^*\!+\!R_{[\cdot,\cdot]}^*)(y) c^{*}\!+\! [c^*\!,\!b^* ]_{A^{*}}  ) \right\rangle \\
&= \left\langle  a^*, [y,z]_{A}    \! - \mathcal{R}_{[\cdot,\cdot]}^* (c^*)y   \right\rangle   +\left\langle  x,        R_{[\cdot,\cdot]}^* (z) b^{*}- [ b^*,c^*]_{A^{*}}  \right\rangle \\
&\quad \!+\!\left\langle    a^*, [z,y]_{A}      - \mathcal{R}_{[\cdot,\cdot]}^* (b^*)z  \right\rangle  +\left\langle   x  ,       R_{[\cdot,\cdot]}^* (y) c^{*}- [c^*\!,\!b^* ]_{A^{*}}  ) \right\rangle.
\end{align*}
Therefore $\mathcal{B}_{A \oplus A^{*}}$ is invariant on $A \bowtie A^{*}$.
\end{proof}

   \subsection{Dual pre-Poisson  bialgebras}
   
   Dualizing the notion of a dual pre-Poisson algebra, we give the notion of a dual pre-Poisson coalgebra.

\begin{defi} \label{defco1}
A   \textbf{dual pre-Poisson coalgebra} is a triple $(A, \delta_{\circ},\delta_{[\cdot,\cdot]})$ where $A$ is a vector space with two bilinear maps $\delta_{\circ},\delta_{[\cdot,\cdot]}: A \to A \otimes A$ such that 
\begin{enumerate}
\item $(A,\delta_{\circ})$ is a {\bf permutative coalgebra}:
\begin{eqnarray}
(\id \otimes \delta_{\circ})\delta_{\circ} = ( \delta_{\circ}  \otimes \id )\delta_{\circ} =(\tau \otimes \id)( \delta_{\circ}  \otimes \id )\delta_{\circ}. \label{permco}
\end{eqnarray}
\item $(A,\delta_{[\cdot,\cdot]})$ is a {\bf Leibniz coalgebra}:
\begin{eqnarray}
&(\id \otimes \delta_{[\cdot,\cdot]})\delta_{[\cdot,\cdot]} = ( \delta_{[\cdot,\cdot]}  \otimes \id )\delta_{[\cdot,\cdot]}+(\tau \otimes \id)(\id \otimes \delta_{[\cdot,\cdot]})\delta_{[\cdot,\cdot]}.
 \label{Leibnizco}
\end{eqnarray}
\item The following three compatibility conditions hold:
\begin{align}
(\id \otimes \delta_{\circ})\delta_{[\cdot,\cdot]}  &= ( \delta_{[\cdot,\cdot]} \otimes \id )\delta_{\circ}  +  (\tau \otimes \id)( \id \otimes  \delta_{[\cdot,\cdot]} )\delta_{\circ} ,\label{xdpp1co} \\
(\delta_{\circ} \otimes  \id )\delta_{[\cdot,\cdot]}  &= (  \id \otimes \delta_{[\cdot,\cdot]}  )\delta_{\circ}  +  (\tau \otimes \id)(  \id \otimes \delta_{[\cdot,\cdot]}  )\delta_{\circ} , \label{xdpp2co} \\
( \delta_{[\cdot,\cdot]} \otimes \id )\delta_{\circ}  &= -(\tau \otimes \id)( \delta_{[\cdot,\cdot]} \otimes \id )\delta_{\circ}  . \label{xdpp3co}
\end{align}
\end{enumerate}
\end{defi}

The notion of a dual pre-Poisson coalgebra could be viewed as the duality of the notion of a dual pre-Poisson algebra. 

\begin{pro}
Let $A$ be a finite-dimensional vector space and $\delta_{\circ},\delta_{[\cdot,\cdot]}: A \to A \otimes A$ be linear maps. If  the operations $\circ_{A^*}, [\cdot,\cdot]_{A^*}: A^* \otimes A^* \to A^*$ are the linear duals of $\delta_{\circ},\delta_{[\cdot,\cdot]}$ respectively, that is, $\circ_{A^*}$ and $[\cdot,\cdot]_{A^*}$ are respectively defined by 
\begin{align*}
\langle a^* \circ_{A^{*}} b^*, x \rangle := \langle a^* \otimes b^*, \delta_{\circ}(x) \rangle, \; \langle [a^* , b^*]_{A^{*}}, x \rangle := \langle a^* \otimes b^*, \delta_{[\cdot,\cdot]}(x) \rangle,\;\forall x \in A, a^*,b^* \in A^{*}.
\end{align*}
Then $(A,\delta_{\circ},\delta_{[\cdot,\cdot]})$ is a dual pre-Poisson coalgebra if and only if $(A^*,\circ_{A^*},[\cdot,\cdot]_{A^*})$ is a dual pre-Poisson algebra.
\end{pro}

   \begin{defi} \label{defbi1}
A   \textbf{dual pre-Poisson bialgebra} is a sextuple $(A,\circ,[\cdot,\cdot], \delta_{\circ},\delta_{[\cdot,\cdot]})$ where $(A,\circ,[\cdot,\cdot])$ is a dual pre-Poisson  algebra and $(A,\delta_{\circ},\delta_{[\cdot,\cdot]})$ is a dual pre-Poisson coalgebra satisfying the following compatibility conditions:
\begin{align}
&\delta_{\circ}(x \circ y)=  (L_{\blacksquare}(x) \otimes \operatorname{id}) \delta_{\circ}(y)+(\operatorname{id} \otimes R_{\circ}(y)) \delta_{\circ}(x), \label{pbi1}\\
&(\operatorname{id} \otimes R_{\circ}(x)  ) \tau\delta_{\circ}(y)=   (R_{\circ}(y) \otimes \operatorname{id})\delta_{\circ}(x), \label{pbi2}\\
&\delta_{\circ}(x \circ y)=     (L_{\blacksquare}(y) \otimes \operatorname{id}) \delta_{\blacksquare}(x)+(\operatorname{id} \otimes L_{\circ}(x)) \delta_{\circ}(y)  ,\label{pbi3}\\
& \left(\mathrm{id} \otimes R_{[\cdot,\cdot]}(x)\right) \tau\delta_{[\cdot,\cdot]}(y)=\left(R_{[\cdot,\cdot]}(y) \otimes \mathrm{id}\right) \delta_{[\cdot,\cdot]}(x), \label{Leibbi1}\\
&\delta_{[\cdot,\cdot]}\left([x, y]\right)=\left(\mathrm{id}  \!\otimes\!  R_{[\cdot,\cdot]}(y)-L_{\square}(y)  \!\otimes\!  \mathrm{id}\right)  \delta_{\square}(x)+\left(\mathrm{id}  \!\otimes\!  L_{[\cdot,\cdot]}(x)+L_{[\cdot,\cdot]}(x)  \!\otimes\!  \mathrm{id}\right)\delta_{[\cdot,\cdot]}(y),\label{Leibbi2}\\
&\delta_{\circ}([x,y])  = (\id \otimes  L_{[\cdot,\cdot]}(x) +   L_{[\cdot,\cdot]}(x) \otimes \id  )\delta_{\circ}(y)  + ( L_{\blacksquare}(y) \otimes \id   -   \id \otimes  R_{\circ}(y) )\delta_{\square}(y) ,\label{dppbi1} \\
&\delta_{[\cdot,\cdot]}(x \circ y) = (\id \otimes  L_{\circ}(x))\delta_{[\cdot,\cdot]}(y) +   (\id \otimes  R_{\circ}(y))\delta_{[\cdot,\cdot]}(x) -   (  L_{\square}(x) \otimes \id )\delta_{\circ}(y) \nonumber\\
&\hspace{2.3cm} -   (  L_{\square}(y) \otimes \id )\delta_{\blacksquare}(x) , \label{dppbi2} \\
&(   \id \otimes R_{[\cdot,\cdot]}(x))\tau\delta_{\circ}(y)=-( R_{\circ}(y) \otimes \id )\delta_{[\cdot,\cdot]}(x)  , \label{dppbi3}\\
&\delta_{\circ}([x,y])  = (\id \otimes  L_{\circ}(x) -   L_{\circ}(x)  \otimes \id  )\delta_{[\cdot,\cdot]}(y)  + (\id \otimes  R_{[\cdot,\cdot]}(y) - L_{\square}(y)  \otimes \id )\delta_{\blacksquare}(x) ,\label{dppbi4} \\
&\delta_{\square}(x \circ y) = (\id \otimes  L_{\circ}(x))\delta_{\square}(y) +   (\id \otimes  L_{\circ}(y))\delta_{\square}(x)-   (  L_{[\cdot,\cdot]}(x) \otimes \id )\delta_{\blacksquare}(y)\nonumber\\
&\hspace{2.14cm} -   (  L_{[\cdot,\cdot]}(y) \otimes \id )\delta_{\blacksquare}(x), \label{dppbi5} \\
&\delta_{[\cdot,\cdot]}(x \circ y) = ( L_{\circ}(x) \otimes  \id )\delta_{[\cdot,\cdot]}(y) -  (\id \otimes  R_{[\cdot,\cdot]}(y))\delta_{\blacksquare}(x) +   (    \id \otimes L_{[\cdot,\cdot]}(x))\delta_{\circ}(y)   \nonumber\\
&\hspace{2.2cm} - (  R_{\blacksquare}(y) \otimes \id )\delta_{\square}(x), \label{dppbi6} \\
&\delta_{\circ}(x \square y) =   (   \id \otimes L_{\square}(x) )\delta_{\circ}(y) +   (    L_{\square}(x)\otimes\id   )\tau \delta_{\circ}(y)- (  \id \otimes R_{\blacksquare}(y) )\delta_{[\cdot,\cdot]}(x) \nonumber\\
&\hspace{1.9cm} -(  R_{\blacksquare}(y)\otimes  \id )\tau\delta_{[\cdot,\cdot]}(x) , \label{dppbi7}
\end{align}
where $x \blacksquare y = x \circ y - y \circ x,\; x \square y = [x ,y] + [y, x]$,  $\delta_{\blacksquare} = \delta_{\circ}-\tau\delta_{\circ},\; \delta_{\square} = \delta_{[\cdot,\cdot]}+\tau\delta_{[\cdot,\cdot]}$  for all $x,y \in A$.
\end{defi} 

\begin{rmk}\label{Bia:sm}
Recall \cite{BYZ} that a {\bf permutative bialgebra} is a triple $(A,\circ,\delta_{\circ})$ that $(A,\circ)$ is a permutative  algebra and  $(A,\delta_{\circ})$ is a permutative  coalgebra satisfying Eqs.~\eqref{pbi1}-\eqref{pbi3}. Moreover, a triple $(A,[\cdot,\cdot],\delta_{[\cdot,\cdot]})$ is called a {\bf  Leibniz bialgebra} \cite{TS} if $(A,[\cdot,\cdot])$ is a Leibniz  algebra and $(A,\delta_{[\cdot,\cdot]})$ is a Leibniz  coalgebra satisfying Eqs.~\eqref{Leibbi1}-\eqref{Leibbi2}.  Therefore, a sextuple $(A,\circ,[\cdot,\cdot], \delta_{\circ},\delta_{[\cdot,\cdot]})$ is a  dual pre-Poisson bialgebra if and only if $(A,\circ,  \delta_{\circ})$ is a permutative bialgebra and  $(A, [\cdot,\cdot],\delta_{[\cdot,\cdot]})$  is a Leibniz bialgebra  satisfying Eqs.~\eqref{xdpp1}-\eqref{xdpp3},  \eqref{xdpp1co}-\eqref{xdpp3co} and \eqref{dppbi1}-\eqref{dppbi7}.
\end{rmk}

\begin{pro}\label{mpdppba} Let  $(A,\circ_A,[\cdot,\cdot]_A)$ and $(A^*,\circ_{A^*},[\cdot,\cdot]_{A^*})$  be   dual pre-Poisson algebras. Let  linear maps  $\delta_{\circ},\delta_{[\cdot,\cdot]}: A \rightarrow A \otimes A $ be the linear duals of  $\circ_{A^{*}}$ and  $[\cdot.\cdot]_{A^{*}}$ respectively. Then $(A,A^*,-L_{\circ}^*, -L_{\circ}^*+R_{\circ}^*, L_{[\cdot,\cdot]}^*, -L_{[\cdot,\cdot]}^*-R_{[\cdot,\cdot]}^*,-\mathcal{L}_{\circ}^*, -\mathcal{L}_{\circ}^*+\mathcal{R}_{\circ}^*,  \mathcal{L}_{[\cdot,\cdot]}^*, -\mathcal{L}_{[\cdot,\cdot]}^*-\mathcal{R}_{[\cdot,\cdot]}^*)$  is a matched pair of dual pre-Poisson algebras if and only if  $(A,\circ_{A},[\cdot,\cdot]_{A}, \delta_{\circ},\delta_{[\cdot,\cdot]})$  is a dual pre-Poisson bialgebra. 
\end{pro}
\begin{proof}
Define  bilinear operations $ \blacksquare,\square:  A \otimes A \rightarrow A$ and linear maps $ \delta_{\blacksquare},\delta_{\square}: A \rightarrow A \otimes A $ respectively by
\begin{align}
&x  \blacksquare  y :=x \circ y -  y \circ x, && x  \square  y := [x, y] +   [y, x],  \label{bss}\\
&\delta_{\blacksquare}(x) := \delta_{\circ}  - \tau\delta_{\circ}, && \delta_{\square}(x) := \delta_{[\cdot,\cdot]}  + \tau\delta_{[\cdot,\cdot]}, \quad \forall x,y\in A. \label{dbss} 
\end{align}
It is well known that $(A,A^*,-L_{\circ}^*, -L_{\blacksquare}^*,-\mathcal{L}_{\circ}^*, -\mathcal{L}_{\blacksquare}^*)$ is a matched pair of permutative algebras if and only if   $(A,\circ,  \delta_{\circ})$  is a permutative bialgebra \cite[Theorem 3.16]{BYZ},  and $(A,A^*, L_{[\cdot,\cdot]}^*,$ $  -L_{\square}^*, \mathcal{L}_{[\cdot,\cdot]}^*, -\mathcal{L}_{\square}^*)$  is a matched pair of Leibniz algebras if and only if   $(A,[\cdot,\cdot],  \delta_{[\cdot,\cdot]})$  is a Leibniz bialgebra \cite[Theorem 2.14]{TS}. Moreover, for all $x,y \in A$, $a^*,b^* \in A^*$, we have
\begin{align*}
\left\langle [x,-\mathcal{L}_{\circ}^*(a^*)y]_{A}, b^* \right\rangle &= -\left\langle y,a^* \circ L_{[\cdot,\cdot]}^*(x)b^* \right\rangle =  \left\langle (\id \otimes L_{[\cdot,\cdot]}(x))\delta_{\circ}(y),a^* \otimes  b^* \right\rangle, \\
 \left\langle  -\mathcal{L}_{\square}^*(-L_{\blacksquare}^*(y)a^*)x, b^* \right\rangle &=  \left\langle (   L_{\blacksquare}(y) \otimes \id)\delta_{\square}(x),a^* \otimes  b^* \right\rangle,\\
\left\langle -\mathcal{L}_{\square}^*(a^*) x \circ_{A} y, b^* \right\rangle &=  \left\langle (     \id \otimes R_{\circ}(y))\delta_{\square}(x),a^* \otimes  b^* \right\rangle,\\
 \left\langle  -\mathcal{L}_{\circ}^*(a^*)([ x , y]_{A}), b^* \right\rangle &=  \left\langle  \delta_{\circ}([x,y]),a^* \otimes  b^* \right\rangle,\\
\left\langle  -\mathcal{L}_{\circ}^*(L_{[\cdot,\cdot]}^*(x)a^*)y, b^* \right\rangle &= - \left\langle (     L_{[\cdot,\cdot]}(x) \otimes \id)\delta_{\circ}(y),a^* \otimes  b^* \right\rangle.\\
 \left\langle[a^*,-L_{\circ}^*(x)b^*]_{A^{*}}, y \right\rangle &=  \left\langle ( \id \otimes L_{\circ}(x))\delta_{[\cdot,\cdot]}(y),a^* \otimes  b^* \right\rangle,\\
 \left\langle - L_{\square}^*(-\mathcal{L}_{\blacksquare}^*(b^*)x)a^*, y \right\rangle &=  -\left\langle    a^*, \mathcal{L}_{\blacksquare}^*(b^*)x \square  y \right\rangle  = -\left\langle ( L_{\square}(y) \otimes  \id )\delta_{\blacksquare}(x),a^* \otimes  b^* \right\rangle,\\
 \left\langle - L_{\square}^*(x) a^* \circ_{A^{*}} b^*, y \right\rangle &=    \left\langle ( L_{\square}(x) \otimes  \id )\delta_{\circ}(y),a^* \otimes  b^* \right\rangle,\\
 \left\langle -L_{\circ}^*(x)([ a^*, b^*]_{A^{*}}) , y \right\rangle &=  \left\langle \delta_{[\cdot,\cdot]}(x \circ y),a^* \otimes  b^* \right\rangle,\\
 \left\langle -L_{\circ}^*(\mathcal{L}_{[\cdot,\cdot]}^*(a^*)x)b^*, y \right\rangle &=  -\left\langle ( \id \otimes   R_{\circ}(y))\delta_{[\cdot,\cdot]}(x),a^* \otimes  b^* \right\rangle.
\end{align*}
In Eqs.~\eqref{mpdualppc1}-\eqref{mpdualppc16},  take
 \begin{align*}
 &l_{\circ_{1}} =  -L_{\circ}^*,\quad r_{\circ_{1}} = -L_{\blacksquare}^*,\quad l_{[\cdot,\cdot]_{1}}   = L_{[\cdot,\cdot]}^*, \quad  r_{[\cdot,\cdot]_{1}} =- L_{\square}^*,  \\
 &l_{\circ_{2}} =  -\mathcal{L}_{\circ}^*,\quad r_{\circ_{2}} = -\mathcal{L}_{\blacksquare}^*,\quad l_{[\cdot,\cdot]_{2}}   = \mathcal{L}_{[\cdot,\cdot]}^*, \quad  r_{[\cdot,\cdot]_{2}} =- \mathcal{L}_{\square}^*. 
 \end{align*}
Hence Eqs.~\eqref{mpdualppc1}-\eqref{mpdualppc2} hold  if and only if Eqs.~\eqref{dppbi1}-\eqref{dppbi2}  hold respectively. Similarly, we have
\begin{align*}
\text{Eq.~}(\ref{dppbi1}) &  \Longrightarrow \text{Eq.~}(\ref{mpdualppc3}),  && \text{Eq.~}(\ref{dppbi2})  \Longrightarrow \text{Eq.~}(\ref{mpdualppc4}), &&\text{Eq.~}(\ref{dppbi1})  \Longleftrightarrow  \text{Eq.~}(\ref{mpdualppc5})  , \\
\text{Eq.~}(\ref{dppbi2}) &  \Longleftrightarrow \text{Eq.~}(\ref{mpdualppc6}),  && \text{Eq.~}(\ref{dppbi3})  \Longleftrightarrow \text{Eq.~}(\ref{mpdualppc8}), &&\text{Eq.~}(\ref{dppbi4})  \Longleftrightarrow  \text{Eq.~}(\ref{mpdualppc7}), \\
\text{Eq.~}(\ref{dppbi5}) &  \Longleftrightarrow \text{Eq.~}(\ref{mpdualppc9}),  && \text{Eq.~}(\ref{dppbi6})  \Longleftrightarrow \text{Eq.~}(\ref{mpdualppc10}), &&\text{Eq.~}(\ref{dppbi3})  \Longleftrightarrow \text{Eq.~}(\ref{mpdualppc11})  , \\
\text{Eq.~}(\ref{dppbi7}) &  \Longleftrightarrow \text{Eq.~}(\ref{mpdualppc12}),  && \text{Eq.~}(\ref{dppbi3})  \Longleftrightarrow \text{Eq.~}(\ref{mpdualppc13}), &&\text{Eq.~}(\ref{dppbi3})  \Longleftrightarrow \text{Eq.~}(\ref{mpdualppc14})   ,\\
\text{Eq.~}(\ref{dppbi7}) &  \Longrightarrow \text{Eq.~}(\ref{mpdualppc15}),  && \text{Eq.~}(\ref{dppbi5})  \Longrightarrow \text{Eq.~}(\ref{mpdualppc16}).
\end{align*}
Thus $(A,A^*,-L_{\circ}^*, -L_{\blacksquare}^*, L_{[\cdot,\cdot]}^*, -L_{\square}^*,-\mathcal{L}_{\circ}^*, -\mathcal{L}_{\blacksquare}^*,  \mathcal{L}_{[\cdot,\cdot]}^*, -\mathcal{L}_{\square}^*)$  is a matched pair of dual pre-Poisson algebras if and only if  $(A,\circ_{A},[\cdot,\cdot]_{A}, \delta_{\circ},\delta_{[\cdot,\cdot]})$  is a dual pre-Poisson bialgebra. 
\end{proof}

 Combining Propositions \ref{mpandmtofdppa} and \ref{mpdppba}, we obtain the following conclusion.
 
 \begin{thm}\label{sandengjia}
Let  $(A,\circ_A,[\cdot,\cdot]_A)$ and $(A^*,\circ_{A^*},[\cdot,\cdot]_{A^*})$  be   dual pre-Poisson algebras. Let  linear maps  $\delta_{\circ},\delta_{[\cdot,\cdot]}: A \rightarrow A \otimes A $ be the linear duals of  $\circ_{A^{*}}$ and  $[\cdot,\cdot]_{A^{*}}$ respectively.  Then the following conditions are equivalent.
\begin{enumerate}
\item  There is a Manin triple $((A \oplus A^{*}, \circ_{A \oplus A^{*}}, [\cdot,\cdot]_{A \oplus A^{*}},  \mathcal{B}_{A \oplus A^{*}}), (A, \circ_A,[\cdot,\cdot]_A), (A^{*},\circ_{A^*}, $ $[\cdot,\cdot]_{A^*}))$    of dual pre-Poisson algebras with $\mathcal{B}_{A \oplus A^{*}}$ given by Eq.~\eqref{abm}.
\item $(A,A^*,-L_{\circ}^*, -L_{\circ}^*+R_{\circ}^*, L_{[\cdot,\cdot]}^*, -L_{[\cdot,\cdot]}^*-R_{[\cdot,\cdot]}^*,-\mathcal{L}_{\circ}^*, -\mathcal{L}_{\circ}^*+\mathcal{R}_{\circ}^*,  \mathcal{L}_{[\cdot,\cdot]}^*, -\mathcal{L}_{[\cdot,\cdot]}^*-\mathcal{R}_{[\cdot,\cdot]}^*)$  is a matched pair of dual pre-Poisson algebras.
    \item $(A,\circ_{A},[\cdot,\cdot]_{A}, \delta_{\circ},\delta_{[\cdot,\cdot]})$  is a dual pre-Poisson bialgebra.
\end{enumerate}
 \end{thm}

   \subsection{The permutative-Leibniz Yang-Baxter equations, $\mathcal{O}$-operators and pre-dual pre-Poisson algebras}

\begin{defi} Let  $(A,\circ,[\cdot,\cdot])$   be a dual pre-Poisson algebra and  $r = \sum\limits_i a_i \otimes b_i \in A \otimes A$. Define  bilinear operations $ \blacksquare,\square:  A \otimes A \rightarrow A$ by Eq.~\eqref{bss}. Set
\begin{align}
\mathbf{P} (r):&=     r_{13} \circ r_{23} -  r_{12} \circ r_{23} +  r_{13} \blacksquare r_{12},\label{PYBE}\\
\mathbf{L} (r):&=  [r_{13},r_{23}] + [r_{12},r_{23}] -  r_{12} \square r_{13},\label{LYBE}
\end{align}
where 
\begin{align*}
&r_{13} \circ r_{23} = \sum_{i,j}a_i \!\otimes\! a_j \!\otimes\!  [b_i,b_j],\;r_{12} \circ r_{23} = \sum_{i,j}a_i \!\otimes\! b_i \circ a_j  \!\otimes\!  b_j,\; r_{13} \blacksquare r_{12} =  \sum_{i,j}a_i \blacksquare a_j \!\otimes\! b_j  \!\otimes\!   b_i,\\
&[r_{13}, r_{23}] = \sum_{i,j}a_i \!\otimes\! a_j \!\otimes\!  [b_i,b_j],\;[r_{12}, r_{23}] = \sum_{i,j}a_i   \!\otimes\! [b_i,a_j] \!\otimes\!  b_j,\;r_{12} \square r_{13} =  \sum_{i,j} a_i \square  a_j \!\otimes\! b_i  \!\otimes\! b_j.
\end{align*}
 Then $r$ is called a solution of the {\bf permutative-Leibniz Yang-Baxter equation (PLYBE)} in $(A,\circ,[\cdot,\cdot])$ if $\mathbf{P} (r) = \mathbf{L} (r) = 0$.
\end{defi}

\begin{rmk}\label{YBE:sm}
For a permutative algebra $(A,\circ)$, $\mathbf{P} (r)=0$ is called the {\bf permutative  Yang-Baxter equation}, whose symmetric solutions give a special class of permutative bialgebras \cite{BYZ}. Moreover, let $(A,[\cdot,\cdot])$ be a Leibniz algebra. Recall \cite{BLST} that  $r \in A \otimes A$ is called a solution of the {\bf  classical Leibniz Yang-Baxter equation}  if it satisfies 
\begin{align}
   [r_{12},r_{13}] +[r_{23},r_{13}]-  r_{12} \square r_{23} = 0. \label{clybe}
\end{align}
Under the symmetric condition, Eq. \eqref{clybe} holds if and only if $\mathbf{L} (r) = 0$. In fact, suppose $r = \sum\limits_i a_i \otimes b_i$, we have
\begin{align*}
  & [r_{12},r_{13}] +[r_{23},r_{13}]-  r_{12} \square r_{23}\\
    &= \sum_{i,j} [a_i,a_j]\otimes b_i \otimes b_j +  a_j \otimes a_i \otimes [b_i,b_j]-   a_i \otimes b_i \square a_j \otimes b_j\\
    &=(\tau \otimes \id) \left(\sum_{i,j} b_i\otimes [a_i,a_j] \otimes b_j +  a_i \otimes  a_j \otimes [b_i,b_j]-    b_i \square a_j \otimes a_i \otimes b_j \right)\\
   &=(\tau \otimes \id) \left(\sum_{i,j} a_i\otimes [b_i,a_j] \otimes b_j +  a_i \otimes  a_j \otimes [b_i,b_j]-    a_i \square a_j \otimes b_i \otimes b_j \right)\\
      &=(\tau \otimes \id)\mathbf{L} (r).
\end{align*}
As the Leibniz analogue of the classical Yang–Baxter equation, the classical Leibniz Yang-Baxter equations are introduced to construct a  class of Leibniz bialgebras \cite{TS}.
\end{rmk}

\begin{pro}\label{cobdppb}
Let  $(A,\circ,[\cdot,\cdot])$   be a dual pre-Poisson algebra and  $r = \sum\limits_i a_i \otimes b_i \in A \otimes A$. Define bilinear operations   $\blacksquare$ and $\square$   by
Eq.~\eqref{bss} and linear maps ${\delta_{\circ,r}}, {\delta_{[\cdot,\cdot],r}}: A \rightarrow A \otimes A$  by
\begin{align}
{\delta_{\circ,r}}(x) := E(x)r ,\quad  {\delta_{[\cdot,\cdot],r}}(x):=
F(x)r, \quad \forall x \in A, \label{EF}
\end{align}
where the linear maps  $E,F: A \rightarrow
\operatorname{End}(A \otimes A)$ are defined
respectively by
\begin{align}
E(x)&:=    R_{\blacksquare}(x) \otimes  \mathrm{id}+ \mathrm{id} \otimes R_{\circ}(x) , \label{eq:e}\\
F(x)&:= L_{\square}(x) \otimes \mathrm{id} - \mathrm{id} \otimes R_{[\cdot,\cdot]}(x).\label{eq:f}
\end{align}
\begin{enumerate}
\item\label{tdppb:co1} Eq. ~\eqref{xdpp1co} holds if and only if
    \begin{align}
   &  (\tau \otimes \id)( \id \otimes L_{\square}(x) \otimes \id)\mathbf{P} (r)+(  \id \otimes R_{\blacksquare}(x) \otimes \id  +\id \otimes  \id \otimes  R_{\circ}(x)  ) \mathbf{L} (r)  \nonumber \\
  &    - \sum_{i}  \big(( L_{\square}(x) \otimes \id )  \tau E(a_i)+   ( L_{\square}(a_i) \otimes \id )  \tau E(x)\big)(r - \tau(r)) \otimes b_i = 0.\label{tdppb:eqco1}
    \end{align}
    \item\label{tdppb:co2} Eq. ~\eqref{xdpp2co} holds if and only if
            \begin{align}
&(\id \otimes \id \otimes  R_{[\cdot,\cdot]}(x) )\mathbf{P} (r)  - (\tau \otimes \id +   \id \otimes \id)(\id \otimes R_{\blacksquare}(x)\otimes \id  )\mathbf{L} (r) \nonumber\\
&+  \sum_{i}\big((    L_{\square}(a_i) \otimes \id   )\tau E(x)-(    R_{\blacksquare}(x) \otimes \id   )\tau F(a_i)  \big )(r -\tau(r))  \otimes b_i= 0. \label{tdppb:eqco2}
    \end{align}
    \item\label{tdppb:co3} Eq. ~\eqref{xdpp3co} holds if and only if
        \begin{align}
 \sum_{i} F( a_i \blacksquare x )(r -\tau(r))\otimes b_i  -   (\tau \otimes \id +   \id \otimes \id)(\id \otimes \id \otimes R_{\circ}(x))\mathbf{L} (r)  = 0. \label{tdppb:eqco3}
    \end{align}
    \item\label{tdppb:cm1} Eq. ~\eqref{dppbi1} holds if and only if
    \begin{align}
    (\id \otimes  L_{\blacksquare}(x) -R_{\circ}(x)  \otimes  \id) F(y)(r -\tau(r))      = 0, \;\; \forall x,y \in A. \label{tdppb:eqcm1}
    \end{align}
        \item\label{tdppb:cm2} Eq. ~\eqref{dppbi2} holds if and only if
    \begin{align}
    ( \id \otimes L_{\square}(x) ) E(y)(r -\tau(r))  = 0, \;\; \forall x,y \in A.\label{tdppb:eqcm2}
    \end{align}
      \item\label{tdppb:cm3} Eq. ~\eqref{dppbi3} holds if and only if
    \begin{align}
   ( R_{\circ}(x)  \otimes  \id) F(y)(r -\tau(r))   = 0, \;\; \forall x,y \in A.\label{tdppb:eqcm3}
    \end{align}
          \item\label{tdppb:cm4} Eq. ~\eqref{dppbi4} holds if and only if
    \begin{align}
   (     \id \otimes L_{\square}(x) - R_{[\cdot,\cdot]}(x)\otimes \id) E(y)(r -\tau(r)) = 0, \;\; \forall x,y \in A.\label{tdppb:eqcm4}
    \end{align}
       \item\label{tdppb:cm5} Eq. ~\eqref{dppbi5} holds if and only if
    \begin{align}
  & (   \id  \otimes  L_{\square}(x)  ) E(y)(r-\tau(r)  ) +  (   L_{\blacksquare}(x) \otimes \id)  \tau(F(y) ( r-\tau(r) ))\nonumber\\
   &  + (    R_{\circ}(x) \otimes \id)  F(y)( r-\tau(r))  +  (   R_{\circ}(y)   \otimes \id )  F(x) ( r-\tau(r))    = 0, \;\; \forall x,y \in A. \label{tdppb:eqcm5}
    \end{align}
     \item\label{tdppb:cm6} Eq. ~\eqref{dppbi6} holds if and only if
    \begin{align}
   (   R_{[\cdot,\cdot]}(x)   \otimes \id )  E(y) ( r-\tau(r)) - (   \id \otimes  R_{\blacksquare}(x) )  F(y)( r-\tau(r))    = 0, \;\; \forall x,y \in A. \label{tdppb:eqcm6}
    \end{align}
     \item\label{tdppb:cm7} Eq. ~\eqref{dppbi7} holds if and only if
    \begin{align}
   (      \id \otimes L_{\square}(x) )  E(y) ( r-\tau(r)) - (   \id \otimes  R_{\blacksquare}(y) )  F(x)( r-\tau(r))    = 0, \;\; \forall x,y \in A.\label{tdppb:eqcm7}
    \end{align}
 \end{enumerate}

    \end{pro}
    \begin{proof}
We only give an explicit proof of  \eqref{tdppb:co1}  and  \eqref{tdppb:cm1}  as examples, whereas the others are obtained similarly. 

 \eqref{tdppb:co1}. Let $x \in A$, we have 
\begin{align*}
&   ( \delta_{[\cdot,\cdot],r} \otimes \id )\delta_{\circ,r}(x)  + (\tau \otimes \id)( \id \otimes  \delta_{[\cdot,\cdot],r} )\delta_{\circ,r}(x)-(\id \otimes \delta_{\circ,r})\delta_{[\cdot,\cdot],r}(x) \\
&=\!\sum_{i,j}\!\Big((a_i \blacksquare x) \square a_j \otimes   b_j \!\otimes\! b_i -   a_j\! \otimes \!  [b_j,a_i \blacksquare x] \otimes b_i + a_i \square a_j \otimes   b_j \otimes b_i  \circ x -   a_j \!\otimes \! [ b_j ,a_i]\otimes b_i  \circ x\\
&\quad+b_i \square a_j \otimes   a_i \blacksquare x \otimes b_j -   a_j\! \otimes \!   a_i \blacksquare x  \otimes [b_j,b_i]+ (b_i \circ x)\square a_j  \otimes   a_i \otimes b_j  -   a_j \otimes   a_i \otimes [b_j,b_i  \circ x]\\
&\quad- x \square a_i  \otimes  a_j \blacksquare b_i  \otimes b_j - x \square a_i \otimes  a_j   \otimes b_j \circ b_i + a_i \otimes a_j \blacksquare [b_i,x] \otimes b_j + a_i \otimes a_j   \otimes b_j \circ [b_i,x]\Big)\\
&=S_{1}+ S_{2}+ S_{3},
 \end{align*}
where
\begin{align*}
S_{1} &=\sum_{i,j}\!\Big(\!(a_i \blacksquare x) \square a_j \!\otimes\!   b_j \!\otimes\! b_i  \!+ \!(b_i \circ x)\square a_j  \!\otimes\!   a_i \!\otimes\! b_j\!-\! x \square a_i  \!\otimes\!  a_j \blacksquare b_i  \!\otimes\! b_j \!-\!(x \square a_i) \!\otimes\!  a_j   \!\otimes\! b_j \circ b_i\!\Big)\\
&=\sum_{i,j}\!\Big(\![a_j,a_i \blacksquare x] \!\otimes\!   b_j \!\otimes\! b_i  + (b_j \circ x)\square a_i  \!\otimes\!   a_j \!\otimes\! b_i\!-(L_{\square}(x) \otimes \id \otimes \id)(r_{23} \blacksquare r_{12}+ r_{23}\circ r_{13})\Big)\\
&=\sum_{i,j}\!\Big(\![a_j,a_i \blacksquare x] \!\otimes\!   b_j \!\otimes\! b_i  + (b_j \circ x)\square a_i  \!\otimes\!   a_j \!\otimes\! b_i\!-(\tau \otimes \id)(\id \otimes  L_{\square}(x)  \otimes \id)\mathbf{P} (r) \\
&\quad\;+ ( L_{\square}(x)  \otimes \id)(  \id  \otimes R_{\blacksquare}(a_i) )(r -\tau(r))\otimes b_i - x \square (b_i \circ a_j) \otimes a_i \otimes b_j\Big)\\
&=\sum_{i,j}\!\Big(\![a_j,a_i \blacksquare x] \!\otimes\!   b_j \!\otimes\! b_i  + (b_j \circ x)\square a_i  \!\otimes\!   a_j \!\otimes\! b_i\!-(\tau \otimes \id)(\id \otimes  L_{\square}(x)  \otimes \id)\mathbf{P} (r) \\
&\quad\;+ ( L_{\square}(x)  \otimes \id)(\tau E(a_i)-  R_{\circ}(a_i)  \otimes \id  )(r -\tau(r))\otimes b_i - x \square (b_i \circ a_j) \otimes a_i \otimes b_j\Big)\\
&=\sum_{i,j}\!\Big(\!  ( L_{\square}(x)  \otimes \id)(\tau E(a_i) )(r -\tau(r))\otimes b_i -(\tau \otimes \id)(\id \otimes  L_{\square}(x)  \otimes \id)\mathbf{P} (r) \\
&\quad\; +[a_j,a_i \blacksquare x] \!\otimes\!   b_j \!\otimes\! b_i  + (b_j \circ x)\square a_i  \!\otimes\!   a_j \!\otimes\! b_i\! -  x \square (a_j \circ a_i) \otimes b_j \otimes b_i \Big).\\
S_{2} &=\sum_{i,j}\!\Big( b_i \square a_j \otimes   a_i \blacksquare x \!\otimes\! b_j -   a_j\! \otimes \!   a_i \blacksquare x  \!\otimes \![b_j,b_i]  -   a_j\! \otimes \!  [b_j,a_i \blacksquare x] \!\otimes\! b_i + a_i \!\otimes\! a_j \blacksquare [b_i,x] \!\otimes\!b_j\Big)\\
&=\sum_{i}\!\Big( (\id \otimes  R_{\blacksquare}(x))(L_{\square}(a_i)  \otimes \id)   (\tau(r)-r)\otimes b_i - (\id \otimes  R_{\blacksquare}(x)  \otimes \id)  \mathbf{L} (r) \Big)\\
&=\sum_{i}\!\Big((L_{\square}(a_i)  \otimes \id)(\tau E(x) - R_{\circ}(x)  \otimes \id)   (\tau(r)-r)\otimes b_i - (\id \otimes  R_{\blacksquare}(x)  \otimes \id)  \mathbf{L} (r) \Big)\\
&=\sum_{i,j}\!\Big(\!(L_{\square}(a_i)  \!\otimes\! \id)(\tau E(x)  )   (\tau(r)-r)\!\otimes\! b_i - (\id \!\otimes\!  R_{\blacksquare}(x)  \!\otimes\! \id)   \mathbf{L} (r) +a_i \square (a_j \circ x)  \!\otimes\!  b_j\!\otimes\! b_i \\
&\quad\;  - a_i \square (b_j \circ x)  \otimes  a_j\otimes b_i  \Big).\\
S_{3} &=\sum_{i,j}\!\Big(a_i \square a_j \!\otimes\!   b_j \!\otimes\! b_i  \circ x -   a_j \! \otimes  \! [ b_j ,a_i]\!\otimes\! b_i  \circ x-   a_j \!\otimes\!   a_i \!\otimes\! [b_j,b_i  \circ x] + a_i \!\otimes\! a_j   \!\otimes\! b_j \circ [b_i,x] \Big) \\
&=-(\id \otimes \id  \otimes  R_{\circ}(x) )(  \mathbf{L} (r)).
\end{align*}
 Note that
 \begin{align*}
 \sum_{i,j}\Big(  a_i \square (a_j \circ x)   \otimes   b_j \otimes  b_i +[a_j,a_i \blacksquare x]  \otimes    b_j  \otimes  b_i   -  x \square (a_j \circ a_i) \otimes b_j \otimes b_i \Big) = 0.
 \end{align*}
Thus 
 \begin{align*}
S_{1}+S_{2}+S_{3} &=  \sum_{i}  \big(( L_{\square}(x) \otimes \id )  \tau E(a_i)+   ( L_{\square}(a_i) \otimes \id )  \tau E(x)\big)(r - \tau(r)) \otimes b_i    \\
  & \quad    -(\tau \!\otimes\! \id)( \id \!\otimes\! L_{\square}(x) \!\otimes\! \id)\mathbf{P} (r)-(  \id \!\otimes\! R_{\blacksquare}(x) \!\otimes\! \id  +\id \!\otimes\!  \id \!\otimes\!  R_{\circ}(x)  ) \mathbf{L} (r). 
 \end{align*}
Hence Eq. ~\eqref{xdpp1co} holds if and only if   Eq. ~\eqref{tdppb:eqco1} holds.

 \eqref{tdppb:cm1}. Let $x,y \in A$ and $\delta_{\square,r} := \delta_{[\cdot,\cdot],r}+ \tau \delta_{[\cdot,\cdot],r}$, we have 
\begin{align*}
&\delta_{\circ,r}([x,y])-  (\id \otimes  L_{[\cdot,\cdot]}(x) +   L_{[\cdot,\cdot]}(x) \otimes \id  )\delta_{\circ,r}(y) - ( L_{\blacksquare}(y) \otimes \id   -   \id \otimes  R_{\circ}(y) )\delta_{\square,r}(y)\\ 
 &=  - (  R_{\blacksquare}(y) \otimes  L_{[\cdot,\cdot]}(x) + \id \otimes L_{[\cdot,\cdot]}(x)  R_{\circ}(y) + L_{[\cdot,\cdot]}(x)R_{\blacksquare}(y)  \otimes \id + L_{[\cdot,\cdot]}(x)  \otimes R_{\circ}(y))r \\
    &\quad - (  L_{\blacksquare}(y) L_{\square}(x) \otimes \id - L_{\blacksquare}(y)  \otimes R_{[\cdot,\cdot]}(x) -    L_{\square}(x)  \otimes R_{\circ}(y)+  \id \otimes  R_{\circ}(y) R_{[\cdot,\cdot]}(x) )r \\
    &\quad - (  L_{\blacksquare}(y)  \otimes L_{\square}(x) - L_{\blacksquare}(y) R_{[\cdot,\cdot]}(x) \otimes \id   -  \id    \otimes R_{\circ}(y)L_{\square}(x) +  R_{[\cdot,\cdot]}(x)   \otimes  R_{\circ}(y) )\tau(r) \\
    &\quad+      (R_{\blacksquare}([x,y]) \otimes \id + \id \otimes R_{\circ}([x,y]))r \\
 &=   (  R_{\blacksquare}(y) \otimes  L_{\square}(x))(\tau(r) -r)   +( R_{[\cdot,\cdot]}(x)  \otimes R_{\circ}(y))(r -\tau(r))   +( L_{\blacksquare}(y) R_{[\cdot,\cdot]}(x) \otimes \id )\tau(r) \\
    &\quad+      \big( (R_{\blacksquare}([x,y])-    L_{\blacksquare}(y) L_{\square}(x)- L_{[\cdot,\cdot]}(x)R_{\blacksquare}(y)) \otimes \id \big)r \\
    &\quad+ \big(\id \otimes (R_{\circ}([x,y]) -R_{\circ}(y) R_{[\cdot,\cdot]}(x)-L_{[\cdot,\cdot]}(x)  R_{\circ}(y) )\big)r \\
     &=   (  R_{\blacksquare}(y) \otimes  L_{\square}(x))(\tau(r) -r)      +( L_{\blacksquare}(y) R_{[\cdot,\cdot]}(x) \otimes \id )(\tau(r)-r)\\
    &\quad+ ( R_{[\cdot,\cdot]}(x)  \otimes R_{\circ}(y) - \id \otimes R_{\circ}(y)  L_{\square}(x))(r -\tau(r))\\
        &= (L_{\blacksquare}(y) \otimes  \id)(R_{[\cdot,\cdot]}(x) \otimes \id  - \id \otimes L_{\square}(x))(\tau(r)-r)\\
    &\quad+ (\id \otimes R_{\circ}(y))( R_{[\cdot,\cdot]}(x) \otimes \id   - \id \otimes   L_{\square}(x))(r -\tau(r))\\ 
            &= (L_{\blacksquare}(y) \otimes  \id-\id \otimes R_{\circ}(y))(\tau F(x))(r-\tau(r) ).
\end{align*}
Hence Eq. ~\eqref{dppbi1} holds if and only if   Eq. ~\eqref{tdppb:eqcm1} holds.
    \end{proof}
   
    A direct consequence of Proposition \ref{cobdppb} is given as follows.

\begin{thm}\label{tcdppab:1}
Let  $(A,\circ,[\cdot,\cdot])$   be a dual pre-Poisson algebra and  $r  \in A \otimes A$. Define  linear maps $\delta_{\circ,r}, \delta_{[\cdot,\cdot],r}: A \otimes A \to A$  by Eq.~\eqref{EF}. 
If $r$ is a symmetric  solution of the PLYBE, then $(A,\circ,[\cdot,\cdot],{\delta_{\circ,r}}, {\delta_{[\cdot,\cdot],r}})$ is a dual pre-Poisson bialgebra.
\end{thm}
\begin{proof}
Sicen $r$ is a symmetric solution of the PLYBE, $(A,\circ ,{\delta_{\circ,r}})$ is a permutative bialgebra by \cite[Corollary 3.32]{BYZ} and $(A,[\cdot,\cdot],{\delta_{[\cdot,\cdot],r}})$ is a Leibniz bialgebra  by Remark \ref{YBE:sm} and    \cite[Corollary 4.12]{TS}.  Moreover, by Proposition \ref{cobdppb}, Eqs.~\eqref{tdppb:eqco1}-\eqref{tdppb:eqcm7} hold. Thus $(A,\circ,[\cdot,\cdot],{\delta_{\circ,r}}, {\delta_{[\cdot,\cdot],r}})$ is a dual pre-Poisson bialgebra.  
\end{proof}

\begin{defi} Let  $(A,\circ,[\cdot,\cdot])$   be a dual pre-Poisson algebra  and $(V;l_{\circ},r_{\circ},l_{[\cdot,\cdot]},r_{[\cdot,\cdot]})$ be a representation of $(A,\circ,[\cdot,\cdot])$. A linear map $T: V \rightarrow A$ is called an \textbf{$\mathcal{O}$-operator on  $(A,\circ,[\cdot,\cdot])$  associated to $(V;l_{\circ},r_{\circ},l_{[\cdot,\cdot]},r_{[\cdot,\cdot]})$} if 
\begin{align}
T(u) \circ  T(v) &= T\left(l_{\circ}(T(u)) v+r_{\circ}(T(v)) u \right), \label{ofP}\\
[T(u) , T(v)] &= T\left(l_{[\cdot,\cdot]}(T(u)) v+r_{[\cdot,\cdot]}(T(v)) u \right),\quad \forall u, v \in V.\label{ofL} 
\end{align}
\end{defi}

\begin{rmk}In \cite{BYZ}, a linear map $T: V \rightarrow A$ is called an \textbf{$\mathcal{O}$-operator} on a permutative algebra $(A,\circ)$  associated to a representation $(V;l_{\circ},r_{\circ})$  of   $(A,\circ)$ if $T$ satisfies Eq.~\eqref{ofP}. It was used to provide a symmetric solution  of the permutative  Yang-Baxter equation, giving a permutative bialgebra. On the other hand, let  $(A, [\cdot,\cdot])$   be a Leibniz algebra  and $(V; l_{[\cdot,\cdot]},r_{[\cdot,\cdot]})$ be a representation of $(A,[\cdot,\cdot])$. If a linear map $T: V \rightarrow A$ satisfies Eq.~\eqref{ofL} on $(A, [\cdot,\cdot])$, then $T$ is called an \textbf{$\mathcal{O}$-operator}  associated to   $(V;l_{[\cdot,\cdot]},r_{[\cdot,\cdot]})$, which is also called a {\bf relative Rota-Baxter operator} in  \cite{TS} and introduced to give rise to a solution of the classical Leibniz Yang-Baxter equation in a larger Leibniz algebra. 
\end{rmk}

\begin{ex}\label{RBOdppa}
 Let  $(A,\circ,[\cdot,\cdot])$   be a dual pre-Poisson algebra. If $P$ is an $\mathcal{O}$-operator  on  $(A,\circ,[\cdot,\cdot])$  associated to the regular representation $(A; L_{\circ},R_{\circ},L_{[\cdot,\cdot]},R_{[\cdot,\cdot]})$, then   $P$  is called a {\bf Rota-Baxter operator  on $(A,\circ,[\cdot,\cdot])$}, that is, $P$ satisfies
\begin{align}
P(x) \circ P (y) &= P\left(P  (x) \circ y+ x \circ  P(y) \right),\label{RBOdppa1}\\
[P(x) ,P (y)] &= P\left([P  (x) , y]+ [x , P(y)]  \right),\quad \forall x, y \in A.\label{RBOdppa2}
\end{align}
\end{ex}

 Let $A$  be a vector space. For any  $r \in A \otimes A$, $r$  can induce a linear map $\widetilde{r}: A^{*} \to A$  in the following way
\begin{align*}
\left\langle \widetilde{r}\left(u^{*}\right), v^{*}\right\rangle :=\left\langle r, u^{*} \otimes v^{*}\right\rangle,     \quad \forall u^{*}, v^{*} \in A^{*} .
\end{align*}
We say that  $r \in V \otimes V$  is nondegenerate if the linear map  $\widetilde{r}$  is invertible. 

\begin{pro}\label{wssqof:0}
 Let  $(A,\circ,[\cdot,\cdot])$   be a dual pre-Poisson algebra  and  $r   \in A \otimes A$ be symmetric.  Then   $r$  is a solution of the PLYBE if and only if $\widetilde{r}$ is an $\mathcal{O}$-operator  on  $(A,\circ,[\cdot,\cdot])$  associated to the coregular representation $(A^*; -L_{\circ}^*,-L_{\blacksquare}^*,L_{[\cdot,\cdot]}^*,-L_{\square}^*)$, that is, $\widetilde{r}$ satisfies the following equations
 \begin{align}
   \widetilde{r}\left(a^{*}\right) \circ \widetilde{r}\left(b^{*}\right)&=\widetilde{r}\left(-L_{\circ}^{*}\left(\widetilde{r}\left(a^{*}\right)\right) b^{*}-L_{\blacksquare}^{*}\left(\widetilde{r}\left(b^{*}\right)\right) a^{*}\right),\label{szxsPYBE}\\
   [\widetilde{r}\left(a^{*}\right),\widetilde{r}\left(b^{*}\right)]&=\widetilde{r}\left(L_{[\cdot,\cdot]}^{*}\left(\widetilde{r}\left(a^{*}\right)\right) b^{*}-L_{\square}^{*}\left(\widetilde{r}\left(b^{*}\right)\right) a^{*}\right),  \quad \forall a^{*}, b^{*} \in A^{*},\label{szxsLYBE}
   \end{align}
where  $\blacksquare$ and $\square$ are defined by Eq.~\eqref{bss}.
\end{pro}
\begin{proof}Let $r=\sum\limits_{i} a_i \otimes b_i$. For $a^{*},b^{*},c^{*} \in A^*$, we have
\begin{align*}
\left\langle \widetilde{r}\left(a^{*}\right) \circ\widetilde{r}\left(b^{*}\right), c^{*}\right\rangle &= -\left\langle  \widetilde{r}\left(b^{*}\right), L_{\circ}^*(\widetilde{r}(a^{*}))c^{*}\right\rangle =  -\sum_{i}\left\langle  b^{*} ,a_i\right\rangle  \left\langle   a^{*} \otimes R_{\circ}^{*}(b_i)c^{*} ,r  \right\rangle\\
&= -\sum_{i,j} \left\langle   a^{*}  ,a_j  \right\rangle \left\langle  b^{*} ,a_i\right\rangle \left\langle R_{\circ}^{*}(b_i)c^{*} ,b_j \right\rangle \\
&= \left\langle  a^{*}  \otimes b^{*} \otimes c^{*},r_{13}\circ r_{23}\right\rangle,\\
\left\langle \widetilde{r}(-L_{\circ}^{*}\left(\widetilde{r}\left(a^{*}\right)\right) b^{*}), c^{*}\right\rangle &= \left\langle  -L_{\circ}^{*}\left(\widetilde{r}\left(a^{*}\right)\right) b^{*} \otimes c^{*}, r\right\rangle  =  -\sum_{i}\left\langle  R_{\circ}^{*}\left(a_i\right)b^{*}  , \widetilde{r}\left(a^{*}\right)\right\rangle \left\langle   c^{*}, b_i\right\rangle \\
&=  -\sum_{i,j}\left\langle   a^{*} , a_j\right\rangle  \left\langle   R_{\circ}^{*}\left(a_i\right)b^{*}  , b_j\right\rangle \left\langle   c^{*}, b_i\right\rangle\\
&= \left\langle   a^{*} \otimes b^{*}\otimes c^*  , r_{12} \circ r_{23} \right\rangle, \\
\left\langle \widetilde{r}(-L_{\blacksquare}^{*}\left(\widetilde{r}\left(b^{*}\right)\right) a^{*}), c^{*}\right\rangle &= -\left\langle  L_{\blacksquare}^{*}\left(\widetilde{r}\left(b^{*}\right)\right) a^{*} \otimes c^{*}, r\right\rangle= -\sum_{i}\left\langle   a^{*}  , L_{\blacksquare} \left(a_i\right)\widetilde{r}\left(b^{*}\right)\right\rangle \left\langle   c^{*}, b_i\right\rangle\\
&=  \sum_{i}\left\langle   L_{\blacksquare}^{*}\left(a_i\right)a^{*}  , \widetilde{r}\left(b^{*}\right)\right\rangle \left\langle   c^{*}, b_i\right\rangle  = - \sum_{i,j} \left\langle    a^{*}  , a_i \blacksquare a_j\right\rangle \left\langle   b^{*} , b_j\right\rangle \left\langle   c^{*}, b_i\right\rangle\\
&=  -\left\langle   a^{*} \otimes b^{*}\otimes c^*  ,  r_{13} \blacksquare r_{12}\right\rangle.
\end{align*}
Then  $\mathbf{P} (r) = 0$ if and only if $\widetilde{r}$ satisfies Eq.~\eqref{szxsPYBE}. Similarly, we have
\begin{align*}
\left\langle [\widetilde{r}\left(a^{*}\right), \widetilde{r}\left(b^{*}\right)], c^{*}\right\rangle &=    \langle  a^{*}  \otimes b^{*} \otimes c^{*},\sum_{i,j} a_j \otimes a_i \otimes [b_j ,b_i] \rangle
=: \left\langle  a^{*}  \otimes b^{*} \otimes c^{*},[r_{13}, r_{23}]\right\rangle,\\
\left\langle
\widetilde{r}( L_{[\cdot,\cdot]}^{*}\left(\widetilde{r}\left(a^{*}\right)\right)
b^{*}), c^{*}\right\rangle &=  \langle  a^{*}  \otimes b^{*}
\otimes c^{*},-\sum_{i,j} a_j \otimes [b_j,a_i] \otimes b_i
 \rangle
=: \left\langle   a^{*} \otimes b^{*}\otimes c^*  , -[r_{12}, r_{23}] \right\rangle, \\
 \left\langle \widetilde{r}(-L_{\square}^{*}\left(\widetilde{r} \left(b^{*}\right)\right) a^{*}), c^{*}\right\rangle &=   \langle  a^{*}
\otimes b^{*} \otimes c^{*},\sum_{i,j} a_j \square  a_i \otimes b_j
\otimes  b_i \rangle=:  \left\langle   a^{*} \otimes
b^{*}\otimes c^*  , r_{12} \square r_{13}\right\rangle.
\end{align*}
Thus   $\mathbf{L} (r) = 0$ if and only if $\widetilde{r}$ satisfies Eq.~\eqref{szxsLYBE}.   This completes the proof.
\end{proof}

\begin{pro}\label{tcdppab:2}
 Let  $(A,\circ,[\cdot,\cdot])$   be a dual pre-Poisson algebra  and  $r   \in A \otimes A$ be symmetric.  Define  linear maps $\delta_{\circ,r}, \delta_{[\cdot,\cdot],r}: A \otimes A \to A$  by Eq.~\eqref{EF} and $\circ_{r},[\cdot,\cdot]_r$ be the linear dual of $\delta_{\circ,r}, \delta_{[\cdot,\cdot],r}$ respectively. Suppose that $r$ is a solution of the PLYBE in $(A,\circ,[\cdot,\cdot])$. Then the dual pre-Poisson algebra structure $\circ_{r},[\cdot,\cdot]_r$ on  $A^{*}$ satisfies
\begin{align}
a^{*} \circ_{r} b^{*}&= -L_{\circ}^{*}\left(\widetilde{r}\left(a^{*}\right)\right) b^{*}-L_{\blacksquare}^{*}\left(\widetilde{r}\left(b^{*}\right)\right) a^{*}, \label{dualdpp:rp}\\
[a^{*} , b^{*}]_r&=L_{[\cdot,\cdot]}^{*}\left(\widetilde{r}\left(a^{*}\right)\right) b^{*}-L_{\square}^{*}\left(\widetilde{r}\left(b^{*}\right)\right) a^{*}, \quad \forall a^{*}, b^{*} \in A^{*},\label{dualdpp:rl}
\end{align}
where  $\blacksquare$ and $\square$ are defined by Eq.~\eqref{bss}. Moreover, $\widetilde{r}: (A^*,\circ_{r},[\cdot,\cdot]_r) \to (A,\circ,[\cdot,\cdot]) $  is a  homomorphism of dual pre-Poisson algebra algebras.
\end{pro}
\begin{proof}
Let $x \in A$, $a^{*}, b^{*} \in A^{*}$, we have
       \begin{align*}
\left\langle a^{*} \circ_{r} b^{*}, x\right\rangle=&\left\langle a^{*} \otimes b^{*}, \delta_{\circ,r}(x)\right\rangle =\left\langle  a^{*} \otimes b^{*},   (R_{\blacksquare}(x) \otimes  \mathrm{id}+ \mathrm{id} \otimes R_{\circ}(x))r  \right\rangle\\
=&\!-\!\left\langle R_{\blacksquare}^{*}(x)a^{*} \!\otimes \! b^{*} ,  r\right\rangle  - \left\langle   a^{*} \!\otimes\! R_{\circ}^{*}(x)b^{*},  r \right\rangle
= \!-\!\left\langle  R_{\blacksquare}^{*}(x)a^{*} ,  \widetilde{r}\left(b^{*}\right)\right\rangle   - \left\langle  R_{\circ}^{*}(x)b^{*},  \widetilde{r}\left(a^{*}\right) \right\rangle\\
=& -\left\langle  L_{\blacksquare}^{*}( \widetilde{r}\left(b^{*}\right))a^{*} ,x \right\rangle   - \left\langle  L_{\circ}^{*}(\widetilde{r}\left(a^{*}\right) )b^{*}, x \right\rangle =\left\langle -L_{\circ}^{*}\left(\widetilde{r}\left(a^{*}\right)\right) b^{*}-L_{\blacksquare}^{*}\left(\widetilde{r}\left(b^{*}\right)\right) a^{*},  x \right\rangle.\\
\left\langle [a^{*} , b^{*}]_{r}, x\right\rangle=&\left\langle a^{*} \otimes b^{*}, \delta_{[\cdot,\cdot],r}(x)\right\rangle =\left\langle  a^{*} \otimes b^{*},   (L_{\square}(x) \otimes  \mathrm{id}- \mathrm{id} \otimes R_{[\cdot,\cdot]}(x))r  \right\rangle\\
=&\!-\!\left\langle L_{\square}^{*}(\!x\!)a^{*} \!\otimes \! b^{*} ,  r\right\rangle  + \left\langle   a^{*} \!\otimes\! R_{[\cdot,\cdot]}^{*}(x)b^{*},  r \right\rangle
= \!-\!\left\langle  L_{\square}^{*}(\!x\!)a^{*} ,\!  \widetilde{r}\left(b^{*}\right)\right\rangle \!  +\! \left\langle  R_{[\cdot,\cdot]}^{*}(\!x\!)b^{*},  \widetilde{r}\left(a^{*}\right) \right\rangle\\
=&  \!-\!\left\langle  L_{\square}^{*}( \widetilde{r}\left(b^{*}\right))a^{*} ,x \right\rangle   + \left\langle  L_{[\cdot,\cdot]}^{*}(\widetilde{r}\left(a^{*}\right) )b^{*}, x \right\rangle =\left\langle L_{[\cdot,\cdot]}^{*}\left(\widetilde{r}\left(a^{*}\right)\right) b^{*}\!-\!L_{\square}^{*}\left(\widetilde{r}\left(b^{*}\right)\right) a^{*},  x \right\rangle. 
\end{align*}
Thus Eqs.~\eqref{dualdpp:rp}-\eqref{dualdpp:rl} hold. Moreover, by Proposition \ref{wssqof}, we  have 
\begin{align*}
\widetilde{r}(a^{*} \circ_{r} b^{*})&=\widetilde{r}\left(-L_{\circ}^{*}\left(\widetilde{r}\left(a^{*}\right)\right) b^{*}-L_{\blacksquare}^{*}\left(\widetilde{r}\left(b^{*}\right)\right) a^{*}\right) = \widetilde{r}\left(a^{*}\right) \circ \widetilde{r}\left(b^{*}\right),\\
\widetilde{r}([ a^{*} , b^{*}]_r)&=\widetilde{r}\left(L_{[\cdot,\cdot]}^{*}\left(\widetilde{r}\left(a^{*}\right)\right) b^{*}-L_{\square}^{*}\left(\widetilde{r}\left(b^{*}\right)\right) a^{*}\right) = [\widetilde{r}\left(a^{*}\right),\widetilde{r}\left(b^{*}\right)].
\end{align*}
This completes the proof.
\end{proof}

\begin{thm}\label{tcdppab:3}
Let  $(A,\circ,[\cdot,\cdot])$   be a dual pre-Poisson algebra  and $(V;l_{\circ},r_{\circ},l_{[\cdot,\cdot]},r_{[\cdot,\cdot]})$ be a representation of $(A,\circ,[\cdot,\cdot])$. Let $T: V \rightarrow A$  be a linear map which can be identified as an element in $A \otimes V^{*} \subseteq\left(A \ltimes_{-l_{\circ}^*, -l_{\circ}^*+r_{\circ}^*, l_{[\cdot,\cdot]}^*, -l_{[\cdot,\cdot]}^*-r_{[\cdot,\cdot]}^*} V^*\right) \otimes\left(A \ltimes_{-l_{\circ}^*, -l_{\circ}^*+r_{\circ}^*, l_{[\cdot,\cdot]}^*, -l_{[\cdot,\cdot]}^*-r_{[\cdot,\cdot]}^*} V^*\right)$  through  $\operatorname{Hom}(V, A) \cong A \otimes V^{*}$. 
Then  $r=T + \tau(T)$  is a solution of the PLYBE in the dual pre-Poisson algebra $A \ltimes_{-l_{\circ}^*, -l_{\circ}^*+r_{\circ}^*, l_{[\cdot,\cdot]}^*, -l_{[\cdot,\cdot]}^*-r_{[\cdot,\cdot]}^*} V^*$ if and only if  $T$ is an  $\mathcal{O}$-operator on   $(A,\circ,[\cdot,\cdot])$  associated to the representation $(V;l_{\circ},r_{\circ},l_{[\cdot,\cdot]},r_{[\cdot,\cdot]})$.
\end{thm}
\begin{proof}
It follows from \cite[Theorem 3.36]{BYZ} and \cite[Theorem 5.2]{TS}.
\end{proof}

Let $A$ be a vector space and $r   \in A \otimes A$ be nondegenerate. Then we can define a nondegenerate bilinear form $\mathcal{B}$ on $A$  by
  \begin{align}
 \mathcal{B}(x,y):=\left\langle \widetilde{r}^{-1}(x) ,y \right\rangle,\quad \forall x,y \in A. \label{BF:inducebyr}
\end{align}
and $\mathcal{B}$ is called the {\bf  induced bilinear form}  by $r$.

\begin{pro}\label{wssqof}
 Let  $(A,\circ,[\cdot,\cdot])$   be a dual pre-Poisson algebra  and  $r   \in A \otimes A$ be symmetric and nondegenerate. Let $\mathcal{B}$ be the induced bilinear form by $r$.  Then   $r$  is a solution of the PLYBE if and only if  $\mathcal{B}$ satisfies the “closed” conditions
\begin{align}
\mathcal{B}(x \circ y,z) &= \mathcal{B}(y,x \circ z) + \mathcal{B}(x,y \blacksquare z),\label{closed:1}\\
\mathcal{B}([x , y],z) &= -\mathcal{B}(y,[x , z]) + \mathcal{B}(x,y \square z),\quad \forall x,y,z\in A, \label{closed:2}
\end{align}
where  $\blacksquare$ and $\square$ are defined by Eq.~\eqref{bss}.
\end{pro}
\begin{proof}
Since $r   \in A \otimes A$ is symmetric, then  $\mathcal{B}$ is symmetric.  Since $\widetilde{r}: A^* \to A$  is invertible, for all $x, y,z \in A$, there are $a^*,b^*,c^* \in A^*$ such that  $\widetilde{r}(a^*) = x$, $\widetilde{r}(b^*) = y$. Then  we have
\begin{align*}
 \mathcal{B}(x \circ y,z) &=   \mathcal{B}(\widetilde{r}(a^*) \circ \widetilde{r}(b^*),z),\\
\mathcal{B}(y,x \circ z) + \mathcal{B}(x,y \blacksquare z) &= \left\langle    b^{*},x\circ z \right\rangle  + \left\langle   a^{*}, y   \blacksquare z  \right\rangle  =\left\langle  -L_{\circ}^{*}\left(\widetilde{r}\left(a^{*}\right)\right) b^{*}-L_{\blacksquare}^{*}\left(\widetilde{r}\left(b^{*}\right)\right) a^{*} ,z \right\rangle \\
&= \mathcal{B}(\widetilde{r}\left(-L_{\circ}^{*}\left(\widetilde{r}\left(a^{*}\right)\right) b^{*}-L_{\blacksquare}^{*}\left(\widetilde{r}\left(b^{*}\right)\right) a^{*}\right),z).
\end{align*}
Thus Eq.~\eqref{szxsPYBE} if and only if  $\mathcal{B}$ satisfies Eq.~\eqref{closed:1}. A similar argument shows that Eq.~\eqref{szxsLYBE} if and only if  $\mathcal{B}$ satisfies Eq.~\eqref{closed:2}. By  Proposition \ref{wssqof:0},  $r$  is a solution of the PLYBE if and only if  $\mathcal{B}$ satisfies Eq.~\eqref{closed:1}-\eqref{closed:2}.
\end{proof}

\begin{rmk}
A nondegenerate symmetric bilinear form $\mathcal{B}$ satisfying Eq.~\eqref{closed:2} on a Leibniz algebra $(A,[\cdot,\cdot])$ is called a {\bf symplectic structure} on $(A,[\cdot,\cdot])$, and $(A,[\cdot,\cdot],\mathcal{B})$ is called a {\bf symplectic Leibniz algebra} \cite{TXS}, whose the underlying algebraic structure is called {\bf Leibniz-dendriform algebra} (Definition \ref{predppadefi}). It is closely related to the phase spaces of Leibniz algebras \cite{TXS}.
\end{rmk}

Building on Semenov-Tian-Shansky’s work \cite{STS}, we generalize the corresponding result to the framework of dual pre-Poisson algebras.

\begin{pro} \label{rbfna1}
Let $(A, \circ,[\cdot,\cdot],\mathcal{B})$ be a quadratic dual pre-Poisson algebra  and  $\phi: A \rightarrow A^{*}$  be the invertible  linear map given by Eq. \eqref{mapinbyB}. Let $r \in A \otimes A$ be symmetric.  Then $r$ is a solution of the PLYBE if and only if $P_{r}:= \widetilde{r}\circ \phi$
is a Rota-Baxter operator on dual pre-Poisson algebra  $(A, \circ,[\cdot,\cdot])$.
\end{pro}
\begin{proof}
For all  $x,y \in A$, there exist $a^*, b^* \in A^*$ such that $x=\phi^{-1}(a^{*})$, $y=\phi^{-1}(b^{*})$. By  Proposition \ref{dualdj:x}, we have
\begin{align*}
&P_r(x) \circ   P_r (y) - P_r (P_r  (x) \circ   y)-P_r (x \circ  P_{r}(y))\\
&= \widetilde{r}(a^{*}) \circ \widetilde{r}(b^{*})-\widetilde{r}\circ \phi(L_{\circ}(\widetilde{r}(a^{*}))\phi^{-1}(b^{*}))-\widetilde{r}\circ \phi(R_{\circ}(\widetilde{r}(b^{*}))\phi^{-1}(a^{*}))\\
&= \widetilde{r}(a^{*}) \circ \widetilde{r}(b^{*}) +\widetilde{r}( L_{\circ}^{*}(\widetilde{r}(a^{*}))b^{*}) +\widetilde{r}(L_{\blacksquare}^*(\widetilde{r}(b^{*}))a^{*}),
\end{align*}
where  $\blacksquare$ is defined by Eq.~\eqref{bss}. Combining  \cite[Theorem 4.17]{TS} and Proposition \ref{wssqof:0},   $r$ is a solution of the PLYBE if and only if $P_{r}$
is a Rota-Baxter operator on  dual pre-Poisson algebra $(A, \circ,[\cdot,\cdot])$.
\end{proof}

Next we introduce the notion of a pre-dual pre-Poisson algebra as the underlying algebraic structure of a Rota–Baxter operator on a dual pre-Poisson algebra.
  
 \begin{defi}\label{predppadefi}
A  \textbf{pre-dual pre-Poisson algebra} is a quintuple $(A, \rhd,\lhd,\succ,\prec)$  such that  $A$  is a vector space and $\rhd,\lhd,\succ,\prec: A \otimes A \rightarrow A$  are  bilinear operations satisfying the following conditions:
\begin{enumerate}
\item \cite{BYZ} $(A,\rhd,\lhd)$ is a {\bf pre-permutative algebra}:
\begin{align*}
x \lhd (y \lhd z + y \rhd z) &= (x \lhd  y )\lhd z =( y \rhd x)\lhd  z =  y \rhd (x \lhd  z),\\
x \rhd ( y \rhd z) & =  (x \lhd  y  + x \rhd  y)\rhd  z =(y \lhd  x  + y \rhd  x)\rhd z,\quad \forall x,y,z \in A.
\end{align*}
\item \cite{TS} $(A,\succ,\prec)$ is a {\bf Leibniz-dendriform algebra}:
\begin{align*}
(x \prec y+x \succ y) \succ z  & =x \succ(y \succ z)- y \succ(x \succ z),\\
(x \succ y) \prec z & =- ( y \prec x) \prec z, \\
x \prec(y \prec z+y \succ z) & =(x \prec y) \prec z+y \succ(x \prec z), \quad \forall x,y,z \in A.
\end{align*}
\item The following  equations hold:
\begin{align}
x \prec ( y \rhd z + y  \lhd z) &= (x \prec y)\lhd z + y \rhd (x \prec z), \label{pdpp1}\\
x \succ (y \lhd z) &= (x \succ y) \lhd z + y \lhd (x \succ z + x \prec z), \label{pdpp2}\\
x \succ (y \rhd z) &= (x \succ y + x \prec y) \rhd z  + y\rhd (x \succ z), \label{pdpp3} \\
(x \lhd y) \prec z &= x \lhd  (y \succ z + y \prec z)    + y \rhd (x \prec z) , \label{pdpp4} \\
(x \rhd y + x \lhd y) \succ z &= x \rhd  (y \succ z  )    + y \rhd  (x \succ z  )  , \label{pdpp5} \\
(x \rhd y - y \lhd x) \prec z &= 0 , \label{pdpp6} \\
(x \succ y + y \prec x) \lhd z &= 0 , \label{pdpp7} \\
(x \succ y + x \prec y + y \succ x + y \prec x) \rhd z &= 0,\quad \forall x,y,z \in A.\label{pdpp8}  
\end{align}
\end{enumerate}
\end{defi}

\begin{pro}\label{saforpdpp}
    Let $(A, \rhd,\lhd,\succ,\prec)$ be a pre-dual pre-Poisson algebra.
    \begin{enumerate}
\item\label{pre:item1}  Then the bilinear operations $\circ$ and $[\cdot,\cdot]$ respectively given by 
\begin{align}
x \circ y := x \rhd y + x \lhd y,\quad [x, y] := x \succ y + x \prec y, \quad \forall x,y \in A, \label{pdppalj}
\end{align}
define a dual pre-Poisson algebra $(A, \circ,[\cdot,\cdot])$,  called the {\bf sub-adjacent dual pre-Poisson algebra} of  $(A, \rhd,\lhd,\succ,\prec)$ and $(A, \rhd,\lhd,\succ,\prec)$ is called the {\bf compatible pre-dual pre-Poisson algebra structure} on $(A, \circ,[\cdot,\cdot])$. Moreover, $(A; L_{\rhd}, R_{\lhd},$ $L_{\succ}, R_{\prec})$ is a representation of $(A, \circ,[\cdot,\cdot])$. 
\item\label{pre:item2} Suppose $x \lhd y = y \rhd x$ and  $x \prec y = -y \succ x$ for all $x,y \in A$. Then the bilinear operations $\diamond$ and $\odot$ respectively given by 
\begin{align*}
x \diamond y &:= x \succ y,\quad x \odot y := x \rhd  y,  \quad \forall x,y \in A,  
\end{align*}
define a  pre-Poisson algebra $(A, \diamond,\odot)$,  called the {\bf  associated  pre-Poisson algebra} of  $(A, \rhd,\lhd,\succ,\prec)$. Moreover, there is a Poisson algebra structure $(A, \bullet,\{\cdot,\cdot\})$ on $A$ with $\bullet,\{\cdot,\cdot\}$ given by 
\begin{align*}
x \bullet y :=  x \odot y + y \odot x, \quad  \{x,y\}:= x \diamond y - y \diamond x,  \quad \forall x,y \in A.
\end{align*}
    \end{enumerate}
\end{pro} 
\begin{proof}\eqref{pre:item1}. It is direct to check that $(A, \circ)$ is a  permutative algebra and  $(A,[\cdot,\cdot])$ is a  Leibniz algebra. Let $x,y,z \in A$, we have
\begin{align*}
&[x, y \circ z] -[x, y ]\circ z - y \circ[x,  z] \\
&= x \succ( y \rhd z + y \lhd z) + x \prec ( y \rhd z + y \lhd z) - (x \succ y + x \prec y) \rhd z - (x \succ y + x \prec y) \lhd z\\
&\quad -  y \rhd ( x \succ z+ x \prec z) - y \lhd ( x \succ z+ x \prec z) \\
&=  x \prec ( y \rhd z + y \lhd z)  - ( x \prec y) \lhd z -  y \rhd ( x \prec z) = 0. \\
&[x \circ y,   z] - x \circ [ y,   z] - y \circ [ x,   z]\\
&=(x \rhd y + x \lhd y)\succ z + (x \rhd y + x \lhd y)\prec z - x \rhd (y \succ z + y \prec z) - x \lhd (y \succ z + y \prec z)\\
&\quad - y \rhd (x \succ z + x \prec z) - y \lhd (x \succ z + x \prec z) \\
&=  (x \rhd y )\prec z - x \rhd (  y \prec z) - y \lhd (x \succ z + x \prec z) \\
&=  (y \lhd x )\prec z - x \rhd (  y \prec z) - y \lhd (x \succ z + x \prec z) =0. \\
&[x, y ]\circ z + [y, x]\circ z\\
&= (x \prec y + x \succ y) \rhd z + (x \prec y + x \succ y) \lhd z + (y \prec x + y \succ x) \rhd z + (y \prec x + y \succ x) \lhd z \\
&=   (x \prec y + x \succ y) \lhd z   + (y \prec x + y \succ x) \lhd z  =0.
\end{align*}
Thus Eqs.~\eqref{xdpp1}-\eqref{xdpp3} hold and $(A, \circ,[\cdot,\cdot])$ is a dual pre-Poisson algebra. Moreover, it is straightforward to show that $(A; L_{\rhd}, R_{\lhd}, L_{\succ}, R_{\prec})$ is a representation of $(A, \circ,[\cdot,\cdot])$. 

\eqref{pre:item2}. It follows from a straightforward verification.
\end{proof}

\begin{pro} \label{OonVector}
Let $(A, \circ,[\cdot,\cdot])$  be a  dual pre-Poisson algebra.
\begin{enumerate}
\item\label{item1} Let  $T: V \rightarrow A$  be an  $\mathcal{O}$-operator of  $(A, \circ,[\cdot,\cdot])$  associated to the representation  $(V;l_{\circ},r_{\circ},l_{[\cdot,\cdot]},r_{[\cdot,\cdot]})$. Then there exists a pre-dual pre-Poisson algebra structure $(V, \rhd_{V},$ $\lhd_{V},\succ_{V},\prec_{V})$ on  $V$, where $\rhd_{V}$, $\lhd_{V}$, $\succ_{V}$ and $\prec_{V}$ are respectively defined by
\begin{align}
 &u \rhd_{V} v:=l_{\circ}(T(u)) v,  && u \lhd_{V} v:=r_{\circ}(T(v)) u, \\
 &u \succ_{V} v:=l_{[\cdot,\cdot]}(T(u)) v, &&  u \prec_{V} v:=r_{[\cdot,\cdot]}(T(v)) u, \quad  \forall u, v \in V.
\end{align}
Moreover, $T$ is a homomorphism of dual pre-Poisson algebras from the sub-adjacent dual pre-Poisson algebra of  $(V, \rhd_{V},\lhd_{V},\succ_{V},$ $\prec_{V})$ to $(A, \circ,[\cdot,\cdot])$.
\item\label{item2} Let $P: A \to A$ be a Rota-Baxter operator on $(A, \circ,[\cdot,\cdot])$. Define  bilinear operations $\rhd$, $\lhd$, $\succ$ and $\prec$  on  $A$  by
\begin{align*}
&x \rhd y := P(x) \circ  y, && x \lhd y := x \circ  P(y), \\
&x \succ y := [P(x),y], &&  x \prec y := [x,P(y)], \quad  \forall x, y \in A.
\end{align*}
Then  $(A, \rhd,\lhd,\succ,\prec)$  is a pre-dual pre-Poisson algebra. Moreover, $T$ is a homomorphism of dual pre-Poisson algebras from the sub-adjacent dual pre-Poisson algebra of  $(A, \rhd,\lhd,\succ,\prec)$  to $(A, \circ,[\cdot,\cdot])$.
\end{enumerate}
\end{pro}  
\begin{proof}
(\ref{item1}). By \cite[Proposition 3.38]{BYZ} and \cite[Proposition 5.10]{TS}, $(V, \rhd_{V},\lhd_{V})$ is a permutative algebra and $(V,  \succ_{V}, \prec_{V})$ is a Leibniz-dendriform algebra.  Moreover, let $u,v,w \in A$, w have
\begin{align*}
&u \prec_V ( v \rhd_V w + v  \lhd_V w) - (u \prec_V v)\lhd_V w - v \rhd_V (u \prec_V w)\\
&=   r_{[\cdot,\cdot]}(T( l_{\circ}(T(v)) w +  r_{\circ}(T(w)) v))u - r_{\circ}(T(w))r_{[\cdot,\cdot]}(T(v)) u -  l_{\circ}(T(v))r_{[\cdot,\cdot]}(T(w)) u\\
&=   r_{[\cdot,\cdot]}( T(v) \circ  T(w))u - r_{\circ}(T(w))r_{[\cdot,\cdot]}(T(v)) u -  l_{\circ}(T(v))r_{[\cdot,\cdot]}(T(w)) u =0.\\
&u \succ_V (v \lhd_V w) - (u \succ_V v) \lhd_V w - v \lhd_V (u \succ_V w + u \prec_V w)\\
&=l_{[\cdot,\cdot]}(T(u)) r_{\circ}(T(w)) v - r_{\circ}(T(w))l_{[\cdot,\cdot]}(T(u)) v - 
r_{\circ}(T(l_{[\cdot,\cdot]}(T(u)) w + r_{[\cdot,\cdot]}(T(w)) u))v\\
&=l_{[\cdot,\cdot]}(T(u)) r_{\circ}(T(w)) v - r_{\circ}(T(w))l_{[\cdot,\cdot]}(T(u)) v - 
r_{\circ}( [T(u),T(w)])v = 0.\\
&u \succ_V (v \rhd_V w) - (u \succ_V v + u \prec_V v) \rhd_V w  - v\rhd_V (u \succ_V w) \\
&=l_{[\cdot,\cdot]}(T(u))l_{\circ}(T(v)) w - l_{\circ}(T(l_{[\cdot,\cdot]}(T(u)) v + r_{[\cdot,\cdot]}(T(v)) u)) w - l_{\circ}(T(v))l_{[\cdot,\cdot]}(T(u))w\\
&=l_{[\cdot,\cdot]}(T(u))l_{\circ}(T(v)) w - l_{\circ}( [T(u) ,T(v)]) w - l_{\circ}(T(v))l_{[\cdot,\cdot]}(T(u))w = 0.\\
&(u \lhd_V v) \prec_V w - u \lhd_V  (v \succ_V w + v \prec_V w)   - v \rhd_V (u \prec_V w)  \\
&=   r_{[\cdot,\cdot]}(T(w))   r_{\circ}(T(v)) u -  r_{\circ}(T(l_{[\cdot,\cdot]}(T(v)) w + r_{[\cdot,\cdot]}(T(w)) v))u - l_{\circ}(T(v))r_{[\cdot,\cdot]}(T(w)) u\\
&=   r_{[\cdot,\cdot]}(T(w))   r_{\circ}(T(v)) u -  r_{\circ}([T(v),T(w)])u - l_{\circ}(T(v))r_{[\cdot,\cdot]}(T(w)) u = 0.\\
& (u \rhd_V v + u \lhd_V v) \succ_V w - u \rhd_V  (v \succ_V w )  - v \rhd_V  (u \succ_V w)\\
&= l_{[\cdot,\cdot]}(l_{\circ}(T(u)) v + r_{\circ}(T(v)) u)w - l_{\circ}(T(u)) l_{[\cdot,\cdot]}(T(v))w  - l_{\circ}(T(v)) l_{[\cdot,\cdot]}(T(u))w\\
&= l_{[\cdot,\cdot]}( T(u) \circ T(v) )w - l_{\circ}(T(u)) l_{[\cdot,\cdot]}(T(v))w  - l_{\circ}(T(v)) l_{[\cdot,\cdot]}(T(u))w = 0. 
\end{align*}
Thus Eqs.~\eqref{pdpp1}-\eqref{pdpp5} hold. It is straightforward to show that   Eqs.~\eqref{pdpp6}-\eqref{pdpp8} also hold. Thus  $(V, \rhd_{V}, \lhd_{V},\succ_{V},\prec_{V})$ is a  pre-dual pre-Poisson algebra.  Moreover, we have
\begin{align*}
T(u \circ_V v) &= T(u \rhd_{V} v + u \lhd_{V} v) = T(l_{\circ}(T(u)) v+ r_{\circ}(T(v)) u) =  T(u) \circ  T(v), \\
T([u , v]_V) &= T(u \succ_{V} v + u \prec_{V} v) = T(l_{[\cdot,\cdot]}(T(u)) v+ r_{[\cdot,\cdot]}(T(v)) u) =  [T(u),T(v)].
\end{align*}

(\ref{item2}). It follows from Example \ref{RBOdppa} and Item \eqref{item1}.
\end{proof}

\begin{ex} \label{ex:predppa}
 Continuing with Example \ref{dwdpp} (2), let $(A,\circ, [\cdot, \cdot])$ be the 2-dimensional dual pre-Poisson algebra  with a basis $\{e_1,e_2\}$ whose nonzero products are given by Eq.~\eqref{dwdpp:1}.  Let $P : A \to A$ be a linear map given by $
P(e_1) = \frac{1}{2}e_1$ and $P(e_2) = e_{2}$. 
Then  $P$ is a Rota-Baxter   operator on $(A,\circ, [\cdot, \cdot])$. Thus by  Proposition \ref{OonVector} \eqref{item2}, there is a 2-dimensional pre-dual pre-Poisson algebra $(A, \rhd,\lhd,\succ,\prec)$  whose nonzero products are explicitly given by 
\begin{align}
e_2 \rhd e_2 = e_2 \lhd e_2  =   e_2 \prec e_2  = e_2 \succ e_2  =e_1. \label{ex:predppa:cfb}
\end{align}

\end{ex}

We give a sufficient and necessary condition for the existence of a compatible pre-dual pre-Poisson algebra structure on a  dual pre-Poisson  algebra.

\begin{thm}\label{iogpdpp:0}
Let $(A, \circ,[\cdot,\cdot])$  be a  dual pre-Poisson algebra. There is a compatiable pre-dual pre-Poisson algebra structure on $(A, \circ,[\cdot,\cdot])$ if and only if there exists an invertible  $\mathcal{O}$-operator on  $(A, \circ,[\cdot,\cdot])$.
\end{thm} 

\begin{proof} Suppose that $(A, \rhd,\lhd,\succ,\prec)$ is a pre-dual pre-Poisson algebra whose the sub-adjacent dual pre-Poisson algebra is $(A, \circ,[\cdot,\cdot])$. Then we have
\begin{align*}
 x \circ y  &=   x \rhd y + x \lhd  y  = \id(L_{\rhd}(\id(x))y+ R_{\lhd}(\id(y)) x), \\
 [x,y]  &=   x \succ y + x \prec  y  = \id(L_{\succ}(\id(x))y+ R_{\prec}(\id(y)) x), \quad \forall x,y \in A.
\end{align*}
Thus the identity map $\mathrm{id}:  A \rightarrow A$ is an invertible $\mathcal{O}$-operator on  $(A, \circ,[\cdot,\cdot])$   associated to the representation  $(A; L_{\rhd}, R_{\lhd}, L_{\succ}, R_{\prec})$.

 Conversely, if  $T: V \to A$  is an invertible  $\mathcal{O}$-operator on $(A, \circ,[\cdot,\cdot])$ associated to the representation  $(V;l_{\circ},r_{\circ},l_{[\cdot,\cdot]},r_{[\cdot,\cdot]})$, then by Proposition \ref{OonVector}, there is a compatiable pre-dual pre-Poisson algebra structure on   $(A, \circ,[\cdot,\cdot])$   given by
\begin{align}
 x \rhd y&:=T\left(l_{\circ}(x)\left(T^{-1}(y)\right)\right),&& x \lhd y:=T\left(r_{\circ}(y)\left(T^{-1}(x)\right)\right), \label{ReOondpp1}\\
x \succ y&:=T\left(l_{[\cdot,\cdot]}(x)\left(T^{-1}(y)\right)\right), && x \prec y:=T\left(r_{[\cdot,\cdot]}(y)\left(T^{-1}(x)\right)\right),\quad \forall x, y \in A. \label{ReOondpp2} 
\end{align}
This completes the proof.
\end{proof}


 \begin{pro}
Let $\mathcal{B}$  be the nondegenerate symmetric bilinear form on a dual pre-Poisson algebra  $(A, \circ,[\cdot,\cdot])$ satisfying the closed conditions  \eqref{closed:1}-\eqref{closed:2}. Then there exists a compatiable pre-dual pre-Poisson algebra structure $(A, \rhd,\lhd,\succ,\prec)$ on $(A, \circ,[\cdot,\cdot])$ with $\rhd,\lhd,\succ,\prec$ given by
\begin{align*}
\mathcal{B}(x \rhd y,z) &= \mathcal{B}(y,x \circ z),&&\mathcal{B}(x \lhd y,z)  =   \mathcal{B}(x,y  \blacksquare   z ),\\
\mathcal{B}(x \succ y,z)&= -\mathcal{B}(y,[x , z]),&&\mathcal{B}(x \prec y,z)  =   \mathcal{B}(x, y \square z ),\quad \forall x,y,z \in A, 
\end{align*}
where  $\blacksquare$ and $\square$ are defined by Eq.~\eqref{bss}.   Moreover,  the  representations  $(A;L_{\circ}, R_{\circ},
L_{[\cdot,\cdot]}, R_{[\cdot,\cdot]})$    and
$(A^*;-L_{\rhd}^{*},$ $-L_{\rhd}^{*}+R_{\lhd}^{*},  L_{\succ}^{*},-L_{\succ}^{*}-R_{\prec}^{*})$  of the
dual pre-Poisson algebra $(A,\circ,[\cdot,\cdot])$   are equivalent.
 \end{pro}
 \begin{proof}
 With the nondegeneracy of $\mathcal{B}$, there exists a linear isomorphism $\varphi: A \to A^*$  defined by Eq.~\eqref{mapinbyB}.  For all $u^*,v^*,w^* \in A^*$, there exist $x,y,z \in A$ such that $u^* = \varphi(x), v^* = \varphi(y), w^* = \varphi(z)$. Therefore
\begin{align*}
\langle \varphi(x \circ y), z \rangle  &=  \langle  -L_{\circ}^* (x)\varphi(y),z\rangle - \langle
L_{\blacksquare}^*(y)\varphi(x), z \rangle,\\
 \langle \varphi([x,y]), z
\rangle  &=  \langle L_{[\cdot,\cdot]}^*  (x)\varphi(y),z\rangle - \langle
L_{\square}^* (y)\varphi(x), z \rangle.
\end{align*}
Then we obtain  
\begin{align*}
 \varphi^{-1}(u^*) \circ \varphi^{-1}(v^*)  &= \varphi^{-1}\big(-L_{\circ}^* (\varphi^{-1}(u^*))v^* -L_{\blacksquare}^*(\varphi^{-1}(v^*))u^*\big),\\
 [\varphi^{-1}(u^*),\varphi^{-1}(v^*)] &= \varphi^{-1}\big( L_{[\cdot,\cdot]}^*  (\varphi^{-1}(u^*))v^*   -L_{\square}^* (\varphi^{-1}(v^*))u^*\big).
\end{align*}
Hence $\varphi^{-1}: A^* \to A$ is an invertible $\mathcal{O}$-operator   on $(A, \circ,[\cdot,\cdot])$ associated to the coregular representation $(A^*; -L_{\circ}^*,-L_{\blacksquare}^*,L_{[\cdot,\cdot]}^*,-L_{\square}^*)$. By Proposition \ref{OonVector},   there is
a compatiable pre-dual pre-Poisson algebra structure $(A, \rhd,\lhd,\succ,\prec)$  on 
$(A, \circ,[\cdot,\cdot])$ with $\rhd,\lhd,\succ,\prec$   defined by
Eqs.~\eqref{ReOondpp1}-\eqref{ReOondpp2} in taking $T := \varphi^{-1}, l_{\circ}:=L_{\circ}^{*}, r_{\circ}:= -L_{\blacksquare}^*, l_{[\cdot,\cdot]}:=
L_{[\cdot,\cdot]}^*, r_{[\cdot,\cdot]}:= -L_{\square}^*$. Thus we have
\begin{align*}
\mathcal{B}(x \rhd y, z) &=  \langle \varphi (x \rhd y), z\rangle
= -\langle   L_{\circ}^{*} (x)\varphi(y), z\rangle =
  \mathcal{B}(y, x \circ z ), \\
\mathcal{B}(x \lhd y, z) &=  \langle \varphi (x \lhd y), z\rangle
= -\langle   L_{\blacksquare}^* (y)\varphi (x), z\rangle =
\mathcal{B}(x, y \blacksquare z).\\
\mathcal{B}(x \succ y, z) &=  \langle \varphi (x \succ y), z\rangle
= \langle  L_{[\cdot,\cdot]}^*(x)\varphi(y), z\rangle =
 -\mathcal{B}(y, [x,z]), \\
\mathcal{B}(x \prec y, z) &=  \langle \varphi (x \prec y), z\rangle
= -\langle   L_{\square}^* (y)\varphi (x), z\rangle =
\mathcal{B}(y, x \square z).
\end{align*}
The rest is straightforward.
\end{proof}

\begin{cor} \label{pretoBialgebra}
    Let $(A, \rhd,\lhd,\succ,\prec)$ be a pre-dual pre-Poisson algebra and $(A, \circ,[\cdot,\cdot])$  be the sub-adjacent dual pre-Poisson algebra of $(A, \rhd,\lhd,\succ,\prec)$. Let  $\left\{e_{1}, \ldots, e_{n}\right\}$  be a basis of  $A$  and  $\left\{e_{1}^{*}, \ldots, e_{n}^{*}\right\}$  be the dual basis. Then
\begin{align}
r=\sum_{i=1}^{n}\left(e_{i} \otimes e_{i}^{*}+e_{i}^{*} \otimes e_{i}\right)
\end{align}
is a symmetric solution of the PLYBE in the dual pre-Poisson algebra  $\hat{A}$ defined by
\begin{align*}
\hat{A}:=A \ltimes_{-L_{\rhd}^{*},-L_{\rhd}^{*}+R_{\lhd}^{*},  L_{\succ}^{*},-L_{\succ}^{*}-R_{\prec}^{*}} A^{*}.
\end{align*}
\end{cor} 
\begin{proof} By Theorem \ref{iogpdpp:0}, the identity map $\mathrm{id} :  A \rightarrow A$  is an invertible $\mathcal{O}$-operator on  $(A, \circ,[\cdot,\cdot])$  associated to the representation  $(A; L_{\rhd}, R_{\lhd}, L_{\succ}, R_{\prec})$. Hence it follows from Theorem \ref{tcdppab:3}. 
\end{proof}

We conclude this section by presenting an example illustrating the above construction.

\begin{ex}
 Continuing with Example \ref{ex:predppa}, let $(A, \rhd,\lhd,\succ,\prec)$  be the 2-dimensional pre-dual pre-Poisson algebra  whose nonzero products are explicitly given by Eq.~\eqref{ex:predppa:cfb}. Thus by  Corollary \ref{pretoBialgebra},  $r= \sum\limits_{i=1}^{2} (e_{i} \otimes e_{i}^{*}+e_{i}^{*} \otimes e_{i})$ is a symmetric solution of the PLYBE in the dual pre-Poisson algebra  $\hat{A}$ whose nonzero products are explicitly given by   the following equations
\begin{align*}
e_2 \circ e_2 =  [e_2,e_2] = 2e_1,\quad
e_2 \circ e_1^* = e_2^*,\quad  [e_2,e_1^*] = -e_2^*,\quad [e_1^*,e_2] = 2e_2^*.
\end{align*}
By Theorem \ref{tcdppab:1}, there is a 4-dimensional dual pre-Poisson bialgebra $(\hat{A},{\delta_{\circ,r}}, {\delta_{[\cdot,\cdot],r}})$  where ${\delta_{\circ,r}}, {\delta_{[\cdot,\cdot],r}}$ are given by
 \begin{align*}
&\delta_{\circ,r}(e_1) = 0, &&\delta_{\circ,r}(e_2) =e_2^* \otimes e_1  ,&&\delta_{\circ,r}(e_1^*) = 2e_2^* \otimes e_2^* ,&&\delta_{\circ,r}(e_2^*) = 0,\\
&\delta_{[\cdot,\cdot],r}(e_1) = 0, &&\delta_{[\cdot,\cdot],r}(e_2) = 2 e_1 \otimes  e_2^* - e_2^* \otimes e_1,&&\delta_{[\cdot,\cdot],r}(e_1^*) = 2e_2^* \otimes e_2^* ,&&\delta_{[\cdot,\cdot],r}(e_2^*) = 0.
\end{align*}

\end{ex}

\medskip

 \noindent
  {\bf Acknowledgements.} The author wishes to express sincere gratitude to Professor Vladimir Dotsenko for his patient and generous assistance. Special thanks are also extended to  Siyuan Chen for his valuable discussions and insightful suggestions.

\medskip

 \noindent


\end{document}